%% file: lcubic-ar.tex
\documentclass[12pt]{amsart}

\usepackage[colorlinks=true, urlcolor=blue,bookmarks=true,bookmarksopen=true,
citecolor=blue
]{hyperref}
\usepackage{graphicx}

\usepackage{epsfig}
\usepackage{amscd}
\usepackage[mathscr]{eucal}
\usepackage{amssymb}
\usepackage{amsxtra}
\usepackage{amsmath}
\usepackage[all]{xy}
\usepackage{enumerate}
\usepackage{mathrsfs}
\usepackage{latexsym}
\usepackage{amscd}

\usepackage{mathtools}
\usepackage{centernot}
\usepackage{color}
\theoremstyle{plain}
\newtheorem{mainthm}{Theorem}
\newtheorem{maincor}[mainthm]{Corollary}
\newtheorem{mainprop}[mainthm]{Proposition}
\newtheorem{thm}{Theorem}[subsection]
\newtheorem{cor}[thm]{Corollary}
\newtheorem{lem}[thm]{Lemma}
\newtheorem{prop}[thm]{Proposition}

\newtheorem{ex-t}[thm]{Example}

\theoremstyle{definition}

\newtheorem*{assumption-L}{Assumption $\mathscr{L}$}

\theoremstyle{remark}
\newtheorem{rem}[thm]{Remark}
\newtheorem*{remnonum}{Remark}

\newtheorem*{remsnonum}{Remarks}

\newcommand{\cntrs}{\setcounter{thm}{0}
  \renewcommand{\thethm}{\thesection.\Alph{thm}}}
\newcommand{\cntrsb}{\setcounter{thm}{0}
  \renewcommand{\thethm}{\thesubsection.\Alph{thm}}}

\newcommand{\Qed}{\hfill \qedsymbol \medskip}

\newcommand{\rrightarrow}{\mathrel{\mathrlap{\longrightarrow}\mkern1mu
    \rightarrow \;}}

\newcommand{\cmred}{}

\newcommand{\pbred}{}

\begin{document}

\title{The Lagrangian Cubic Equation}

\author{Paul Biran}

\author{Cedric Membrez} \thanks{The second author
  was partially supported by the Swiss National Science Foundation
(grant number 200021\_134747).}

\date{\today}

\address{Paul Biran, Department of Mathematics, ETH Z\"urich,
  R\"amistrasse 101, 8092 Z\"urich, Switzerland}
\email{biran@math.ethz.ch}

\address{Cedric Membrez, School of Mathematical Sciences, Tel-Aviv
  University, Tel-Aviv 69978, Israel}
\email{cedric.membrez@math.ethz.ch}

\bibliographystyle{alphanum}

\maketitle

%
%

\input{abstract.tex}
\input{intro.tex}

\input{organization-ar.tex}

\input{floer-setting.tex}
\input{lag-cubic-eq.tex}
\input{lag-cob.tex}


\input{examples.tex}

\input{examples-full-ring-ar.tex}

\input{enumerative-ar.tex}
\input{non-monotone-ar.tex}

\input{app-spectral.tex}

\input{lcubic-ar.bbl}


\end{document}

%% file: abstract.tex
\begin{abstract}
   Let $M$ be a closed symplectic manifold and $L \subset M$ a
   Lagrangian submanifold.  Denote by $[L]$ the homology class induced
   by $L$ viewed as a class in the quantum homology of $M$. The
   present paper is concerned with properties and identities involving
   the class $[L]$ in the quantum homology ring. We also study the
   relations between these identities and invariants of $L$ coming
   from Lagrangian Floer theory. We pay special attention to the case
   when $L$ is a Lagrangian sphere.
\end{abstract}

%% file: intro.tex
\section{Introduction and main results} \cntrs
\label{s:intro}

Let $M^{2n}$ be a closed symplectic $2n$-dimensional manifold. Assume
further that $M$ is monotone with minimal Chern number $C_M$
(see~\S\ref{sb:monotone} below for the definitions). Denote by $QH(M)$
the quantum homology of $M$ with coefficients in the ring
$\mathbb{\mathbb{Z}}[q]$, where the degree of the variable $q$ is
$|q|=-2$. Denote by $*$ the quantum product on $QH(M)$ and for a class
$a \in QH(M)$, $k \in \mathbb{N}$, we write $a^{*k}$ for the $k$'th
power of $a$ with respect to this product.

Let $S \subset M$ be an oriented Lagrangian $n$-sphere. Denote by $[S]
\in QH_n(M)$ the homology class represented by $S$ in the quantum
homology of $M$. Our first result shows that $[S]$ always satisfies a
cubic or quadratic equation of a very specific type:
\begin{mainthm} \label{t:cubic-eq-sphere}
   \begin{enumerate}
     \item \label{i:n-odd} If $n=$~odd then $[S]*[S]=0$.
     \item Assume $n=$~even. Then:
      \begin{enumerate}[(i)]
        \item \label{i:div} If $C_M | n$ then there exists a unique
         $\gamma_S \in \mathbb{Z}$ such that $[S]^{*3} = \gamma_S [S]
         q^{n}$. If we assume in addition that $2C_M \centernot| n$,
         then $\gamma_S$ is divisible by $4$, while if $2C_M | n$ then
         $\gamma_S$ is either $0 (\bmod \,4)$ or $1 (\bmod \,4)$.
        \item \label{i:ndiv} If $C_M \centernot| n$ then $[S]^{*3} =
         0$.
      \end{enumerate}
   \end{enumerate}
\end{mainthm}

The proof of Theorem~\ref{t:cubic-eq-sphere}, given
in~\S\ref{sb:prf-cubic-eq-S}, follows from a simple argument involving
Lagrangian Floer homology.  The cases~\eqref{i:n-odd},~\eqref{i:ndiv}
are particularly simple, whereas case~\eqref{i:div} splits into two
sub-cases:
\begin{enumerate}
  \item [(2i-a)] \label{i:2CM-n} $2C_M | n$.
  \item[(2i-b)] \label{i:CM-n} $C_M | n$, but $2C_M \centernot| n$.
\end{enumerate}
We will see below that out of these two sub-cases the most interesting
is~(2i-a). In that case the constant $\gamma_S$ has other
interpretations coming from Floer theory and enumerative geometry of
holomorphic disks. These will be explained in detail in the sequel.

\begin{rem}
   \begin{enumerate}
     \item When $n$ is even it is easy to see that $[S]  
     \in H_n(M)$ is neither $0$ nor a torsion class. Therefore in that
      case $\gamma_S$ is uniquely determined.
     \item Points~\eqref{i:n-odd} and~\eqref{i:ndiv} of the theorem
      cover the symplectically aspherical case (i.e.
      $[\omega]|_{\pi_2(M)}=0$) if we set $C_M = \infty$. Of course,
      the statement in that case is completely obvious.
     \item A version of Theorem \ref{t:cubic-eq-sphere} also holds 
     in the non-monotone case for Lagrangian 2-spheres, the precise 
     statement can be found in~\S\ref{s:non-monotone}.
    \item Theorem~\ref{t:cubic-eq-sphere} continues to hold also when
     $S$ is a $\mathbb{Z}$-homology sphere, except possibly when $2C_M
     | n$.  The difference between the case $2C_M | n$ and the others
     is that in that case $[S]$ a priori satisfies only the cubic
     equation~\eqref{eq:cubic-L} from Theorem~\ref{t:cubic-eq} below.
     For the vanishing of the coefficient of $[S]^{*2}$ we will use
     the Dehn twist along $S$ (see Corollary~\ref{c:sig=0} and the
     short discussion after it), hence we need to assume that $S$ is
     diffeomorphic to a sphere. At the same time we are not aware of
     interesting computable examples where $S$ is a
     $\mathbb{Z}$-homology sphere yet not a genuine sphere.
   \end{enumerate}
\end{rem}

For the rest of the introduction we concentrate on case~(2i-a) and its
possible generalizations. Assume from now on that $L \subset M$ is a
Lagrangian submanifold (not necessarily a sphere). Denote by
$HF_*(L,L)$ the self Floer homology of $L$ with coefficients in
$\mathbb{Z}$. See~\S\ref{s:floer-setting} for the Floer theoretical
setting. In what follows we will recurringly appeal to the following
set of assumptions or to a subset of it:
\begin{assumption-L}
   \begin{enumerate}
     \item $L$ is closed (i.e. compact without boundary). Furthermore
      $L$ is monotone with minimal Maslov number $N_L$ that satisfies
      $N_L \mid n$ (see~\S\ref{sb:monotone} for the definitions). Set
      $\nu = n/N_L$.
     \item $L$ is oriented. Moreover we assume that $L$ is spinable
      (i.e. can be endowed with a spin structure).
     \item $HF_n(L,L)$ has rank $2$.
     \item Write $\chi = \chi(L)$ for Euler-characteristic of $L$. We
      assume that $\chi \neq 0$.
   \end{enumerate}
\end{assumption-L}
Note that conditions $(1)$ and $(2)$ together imply that $n=$ even,
since orientable Lagrangians have $N_L=$ even. Independently,
conditions $(2)$ and $(4)$ also imply that $n=$even. As we will see
later there are many Lagrangian submanifolds that satisfy
Assumption~$\mathscr{L}$ -- for example, even dimensional Lagrangian
spheres in monotone symplectic manifolds $M$ with $2C_M | n$.
See~\S\ref{sb:exps} and~\S\ref{s:examples} for more examples.

\pbred{Unless otherwise stated, from now on we implicitly assume all
  Lagrangian submanifolds to be connected.}

\subsection{The Lagrangian cubic equation} \label{sb:lcubic-eq} Here
we need to work with $\mathbb{Q}$ as the base ring. Denote by $QH(M;
\mathbb{Q}[q])$ the quantum homology of $M$ with coefficients in the
ring $\mathbb{Q}[q]$. Given an oriented Lagrangian submanifold $L
\subset M$ denote by $[L] \in QH_n(M;\mathbb{Q}[q])$ its homology
class in the quantum homology of the ambient manifold $M$. We will
also make use of the following notation $\varepsilon =
(-1)^{n(n-1)/2}$.

Our first result is the following.
\begin{mainthm}[The Lagrangian cubic equation] \label{t:cubic-eq} Let
   $L \subset M$ be a Lagrangian submanifold satisfying
   assumption~$\mathscr{L}$. Then there exist unique \cmred{constants
     $\sigma_L \in \tfrac{1}{\chi^2}\mathbb{\mathbb{Z}}$, $\tau_L \in
     \tfrac{1}{\chi^3} \mathbb{Z}$} such that the following equation
   holds in $QH(M;\mathbb{Q}[q])$:
   \begin{equation} \label{eq:cubic-L} [L]^{*3} - \varepsilon \chi
      \sigma_L [L]^{*2}q^{n/2} - \chi^2 \tau_L [L]q^{n} = 0.
   \end{equation}
   \cmred{If $\chi$ is square-free then $\sigma_L \in
     \tfrac{1}{\chi}\mathbb{Z}$ and $\tau_L \in
     \tfrac{1}{\chi^2}\mathbb{Z}$.} Moreover, the constant $\sigma_L$
   can be expressed in terms of genus $0$ Gromov-Witten invariants as
   follows:
   \begin{equation} \label{eq:GW-LLL} \sigma_L = \frac{1}{\chi^2}
      \sum_A GW^M_{A,3}([L],[L],[L]),
   \end{equation}
   where the sum is taken over all classes $A \in H_2(M)$ with
   $\langle c_1, A \rangle = n/2$.
\end{mainthm}

In~\S\ref{s:lag-cubic-eq} we will prove a more general result
concerning a Lagrangian submanifold $L$ and an arbitrary class $c \in
H_n(M)$ which satisfies $c \cdot [L] \neq 0$. We will prove that they
satisfy a mixed equation of degree three involving $[L]$ and $c$.
Equation~\eqref{eq:cubic-L} is the special case $c = [L]$.

Here is an immediate corollary of Theorem~\ref{t:cubic-eq}:
\begin{maincor} \label{c:sig=0} Let $L \subset M$ be a Lagrangian
   submanifold satisfying Assumption $\mathscr{L}$. Assume in addition
   that there exists a symplectic diffeomorphism $\varphi: M
   \longrightarrow M$ such that $\varphi_*([L]) = -[L]$. Then
   $\sigma_L = 0$, hence equation~\eqref{eq:cubic-L} reads in this
   case:
   \begin{equation*}
      [L]^{*3} - \chi^2 \tau_L [L]q^{n} = 0.
   \end{equation*}
\end{maincor}
When $L$ is a Lagrangian sphere in a symplectic manifold $M$ with
$2C_M | n$ then point~\eqref{i:div} of Theorem~\ref{t:cubic-eq-sphere}
follows from Corollary~\ref{c:sig=0}. Indeed, we can take $\varphi$ to
be the Dehn twist along $L$. The Picard-Lefschetz formula (see
e.g.~\cite{Dimca:sing-hyper-book, Ar:sing-theory-I}) gives
$\varphi_*([L]) = -[L]$ since $n=\dim L$ is even and $\chi = 2$.
Corollary~\ref{c:sig=0} then implies that $\sigma_L = 0$ (and we have
$\gamma_L = 4 \tau_L$).  \cmred{ Note that in this case we have
  $\tau_L \in \tfrac{1}{4}\mathbb{Z}$.}

\begin{proof}[Proof of Corollary~\ref{c:sig=0}]
   Applying $\varphi_*$ to the equation~\eqref{eq:cubic-L} and
   comparing the result to~\eqref{eq:cubic-L} yields $\varepsilon \chi
   \sigma_L [L]^{*2}=0$. Since $\chi \neq 0$ it follows that $\sigma_L
   [L]^{*2}=0$. But $[L] \cdot [L] = \varepsilon \chi \neq 0$, hence
   $[L]^{*2} \neq 0$. This implies that $\sigma_L=0$.
\end{proof}

\subsection{The discriminant} \label{sb:discr-intro} \cntrsb Let $A$
be a quadratic algebra over $\mathbb{Z}$. By this we mean that $A$ is
a commutative unital ring such that $\mathbb{Z}$ embeds as a subring
of $A$, $\mathbb{Z} \to A$, and furthermore that $A/\mathbb{Z} \cong
\mathbb{Z}$. Thus the underlying additive abelian group of $A$ is a
free abelian group of rank $2$. Pick a generator $p \in A/\mathbb{Z}$
so that $A/\mathbb{Z} = \mathbb{Z} p$. We have the following exact
sequence:
\begin{equation} \label{eq:ex-seq-A} 0 \longrightarrow \mathbb{Z}
   \longrightarrow A \xrightarrow{\;\;\epsilon \;\;} \mathbb{Z} p
   \longrightarrow 0,
\end{equation}
where the first map is the ring embedding and $\epsilon$ is the
obvious projection.  Choose a lift $x \in A$ of $p$, i.e. $\epsilon(x)
= p$. Then additively we have $A \cong \mathbb{Z} x \oplus
\mathbb{Z}$.  With these choices there exist $\sigma(p,x), \tau(p,x)
\in \mathbb{Z}$ such that $$x^2 = \sigma(p,x) x + \tau(p,x).$$ The
integers $\sigma(p,x), \tau(p,x)$ depend on the choices of $p$ and of
$x$.  However, a simple calculation (see~\S\ref{sbsb:discr-well-def})
shows that the following expression
\begin{equation} \label{eq:discr-A} \Delta_A := \sigma(p,x)^2 + 4
   \tau(p,x) \in \mathbb{Z}
\end{equation}
is independent of $p$ and $x$, hence is an invariant of the
isomorphism type of $A$. In fact in~\S\ref{sbsb:discr-inv-class} we
show that $\Delta_A$ determines the isomorphism type of $A$. We call
$\Delta_A$ the discriminant of $A$.

\begin{remsnonum}
   \begin{enumerate}
     \item Another description of $\Delta_A$ is the following. Write
      $A$ as $A \cong \mathbb{Z}[T] / (f(T))$, where $f(T) \in
      \mathbb{Z}[T]$ is a monic quadratic polynomial. Then $\Delta_A$
      is the discriminant of $f(T)$ (and is independent of the choice
      of $f(T)$). In particular $A_{\mathbb{C}} := A \otimes
      \mathbb{C}$ is semi-simple iff $\Delta_A \neq 0$.
     \item When $\Delta_A$ is not a square $A_{\mathbb{Q}} := A
      \otimes \mathbb{Q}$ is a quadratic number field. The
      discriminant $\Delta_A$ is related to the discriminant of
      $A_{\mathbb{Q}}$ as defined in number theory.
     \item It is easy to see from~\eqref{eq:discr-A} that the only
      values $\Delta_A (\bmod \,4)$ can assume are $0$ and $1$.
   \end{enumerate}
\end{remsnonum}

Let $L$ be a Lagrangian submanifold satisfying conditions $(1)-(3)$ of
Assumption~$\mathscr{L}$ and choose a spin structure on $L$ compatible
with its orientation. Consider $A = HF_n(L,L)$ endowed with the
Donaldson product $$*: HF_n(L,L) \otimes HF_n(L,L) \longrightarrow
HF_n(L,L), \quad a \otimes b \longmapsto a*b.$$ Recall that $A$ is a
unital ring with a unit which we denote by $e_L \in HF_n(L,L)$. The
conditions $(1)-(3)$ of Assumption~$\mathscr{L}$ ensure that $A$ is a
quadratic algebra over $\mathbb{Z}$. (In case $A$ has torsion we just
replace it by $A/T$, where $T$ is its torsion ideal.) Denote by
$\Delta_L$ the discriminant of $A$, $\Delta_L := \Delta_A$ as defined
in~\eqref{eq:discr-A}. (We suppress here the dependence on the spin
structure, as we will soon see that in our case $\Delta_L$ does not
depend on it.)

The following theorem shows that the discriminant $\Delta_L$ depends
only on the class $[L] \in QH_n(M)$ and can be computed by means of
the ambient quantum homology of $M$.
\begin{mainthm} \label{t:rel-to-discr} Let $L \subset M$ be a
   Lagrangian submanifold satisfying Assumption~$\mathscr{L}$. Let
   $\sigma_L, \tau_L \in \mathbb{Q}$ be the constants from the cubic
   equation~\eqref{eq:cubic-L} in Theorem~\ref{t:cubic-eq}. Then
   $$\Delta_L = \sigma_L^2 + 4 \tau_L.$$
\end{mainthm}
The proof appears in~\S\ref{s:lag-cubic-eq}.
\begin{remsnonum}
   \begin{enumerate}
     \item \textbf{Warning:} The pair of coefficients $\sigma_L,
      \tau_L$ and $\sigma(p,x), \tau(p,x)$ should not be confused.
      The first pair is always uniquely determined by $[L]$ and can be
      read off the ambient quantum homology of $M$ via the cubic
      equation~\eqref{eq:cubic-L}. In contrast, the second pair
      $\sigma(p,x), \tau(p,x)$ are defined via Lagrangian Floer
      homology and strongly depend on the choice of the lift $x$ of
      $p$. For example, we have seen that if $L$ is a sphere then
      $\sigma_L=0$, but as we will see later (e.g.
      in~\S\ref{s:disc-lcob}) for some (useful) choices of $x$ we have
      $\sigma(p,x) \neq 0$. Additionally, $\sigma(p,x), \tau(p,x) \in
      \mathbb{Z}$ while $\sigma_L, \tau_L \in \mathbb{Q}$. Still, the
      two pairs of coefficients are related in that $\sigma(p,x)^2 + 4
      \tau(p,x) = \sigma_L^2 + 4 \tau_L = \Delta_L$.

      As we will see in the proof of Theorem~\ref{t:rel-to-discr}, the
      coefficients $\sigma_L, \tau_L$ do occur as $\sigma(p,x_0),
      \tau(p,x_0)$ but for a special choice of $x_0$, which however
      requires working over $\mathbb{Q}$.

     \item A different version of the discriminant $\Delta_L$ was
      previously defined and studied by Biran-Cornea
      in~\cite{Bi-Co:lagtop}. In that paper the discriminant occurs as
      an invariant of a quadratic form defined on $H_{n-1}(L)$ via
      Floer theory. In the case $L$ is a $2$-dimensional Lagrangian
      torus the discriminant from~\cite{Bi-Co:lagtop} and $\Delta_L$,
      as defined above, happen to coincide due to the associativity of
      the product of $HF_n(L,L)$.  Moreover, in dimension $2$,
      $\Delta_L$ has an enumerative description in terms of counting
      holomorphic disks with boundary on $L$ which satisfy certain
      incidence conditions. This description continues to hold also
      for $2$-dimensional Lagrangian spheres with $N_L = 2$ (or more
      generally for all $2$-dimensional Lagrangian submanifolds
      satisfying Assumption~$\mathscr{L}$) and the proof is the same
      as in~\cite{Bi-Co:lagtop}.
     \item Since $\sigma_L, \tau_L$ do not depend on the spin
      structure chosen for $L$ (although $\sigma(p,x)$ and $\tau(p,x)$
      do) it follows from Theorem~\ref{t:rel-to-discr} that $\Delta_L$
      does not depend on that choice either. As for the orientation on
      $L$, if we denote $\bar{L}$ the Lagrangian $L$ with the opposite
      orientation then it follows from Theorem~\ref{t:cubic-eq} that
      $\sigma_{\bar{L}} = - \sigma_{L}$ and $\tau_{\bar{L}} = \tau_L$.
      In particular $\Delta_{\bar{L}} = \Delta_L$.
   \end{enumerate}
\end{remsnonum}

The next theorem is concerned with the behavior of the discriminant
under Lagrangian cobordism. We refer the reader to~\cite{Bi-Co:cob1}
for the definitions.

\begin{mainthm} \label{t:cob} Let $L_1, \ldots, L_r \subset M$ be
   monotone Lagrangian submanifolds, each satisfying conditions $(1)$
   -- $(3)$ of Assumption~$\mathscr{L}$. Let $V^{n+1} \subset
   \mathbb{R}^2 \times M$ be a connected monotone Lagrangian cobordism
   whose ends correspond to $L_1, \ldots, L_r$ and assume that $V$
   admits a spin structure. Denote by $N_V$ the minimal Maslov number
   of $V$ and assume that:
   \begin{enumerate}
     \item $H_{j N_V}(V,\partial V) = 0$ for every $j$.
     \item $H_{1+jN_V}(V) = 0$ for every $j$.
   \end{enumerate}
   Then $\Delta_{L_1} = \cdots = \Delta_{L_r}$. Moreover if $r \geq 3$
   then $\Delta_{L_i}$ is a perfect square for every $i$.
\end{mainthm}
The proof is given in~\S\ref{s:disc-lcob}. As a corollary we obtain:
\begin{maincor} \label{c:del_1=del_2} Let $(M, \omega)$ be a monotone
   symplectic manifold with $2 C_M \mid n$, where $C_M$ is the minimal
   Chern number of $M$. Let $L_1, L_2 \subset M$ be two Lagrangian
   spheres that intersect transversely at exactly one point. Then
   $\Delta_{L_1} = \Delta_{L_2}$ and moreover this number is a perfect
   square.
\end{maincor}
We will in fact prove a stronger result in~\S\ref{sb:lag-1-pt} (see
Corollary~\ref{c:L1-L2-1pt}).

\subsection{Examples} \label{sb:exps} \cntrsb

We begin with a topological criterion that assures that condition~(3)
in Assumption~$\mathscr{L}$ is satisfied. This provides us with examples
of Lagrangian submanifolds to which the theory applies.
\begin{mainprop} \label{p:criterion} Let $L \subset M$ be an oriented
   Lagrangian submanifold satisfying condition~$(1)$ of
   Assumption~$\mathscr{L}$. Assume in addition that:
   \begin{enumerate}
     \item $[L] \neq 0 \in H_n(M;\mathbb{Q})$ (this is satisfied e.g.
      when $\chi(L) \neq 0$).
     \item $H_{j N_L}(L) = 0$ for every $0<j<\nu$. 
   \end{enumerate}
   Then condition $(3)$ in Assumption~$\mathscr{L}$ is satisfied too.
   In particular Lagrangian spheres $L$ that satisfy condition $(1)$
   of Assumption~$\mathscr{L}$ satisfy the other three conditions in
   Assumption~$\mathscr{L}$.
\end{mainprop}
The proof appears in~\S\ref{sb:prf-p-criterion}.

We now provide a sample of examples. More details will be given
in~\S\ref{s:examples}

\subsubsection{Lagrangian spheres in blow-ups of $\mathbb{C}P^2$}
\label{sbsb:intro-exp-blcp2}

Let $(M_k, \omega_k)$ be the monotone symplectic blow-up of
${\mathbb{C}}P^2$ at $2 \leq k \leq 6$ points. We normalize $\omega_k$
so that it is cohomologous to $c_1$. Denote by $H \in H_2(M_k)$ the
homology class of a line not passing through the blown up points and
by $E_1, \ldots, E_k \in H_2(M_k)$ the homology classes of the
exceptional divisors over the blown up points. With this notation the
Poincar\'{e} dual of the cohomology class of the symplectic form
$[\omega_k] \in H^2(M_k)$ satisfies $$PD [\omega_k] = PD(c_1) = 3H -
E_1 - \cdots - E_k.$$

The Lagrangian spheres $L \subset M_k$ lie in the following homology
classes (see~\S\ref{sb:lag-spheres-blow-ups} for more details):
\begin{enumerate}
  \item For $k=2$: $\pm (E_1 - E_2)$.
  \item For $2 \leq k \leq 5$: $\pm(E_i - E_j)$, $i < j$, and $\pm(H -
   E_i - E_j - E_l)$ with $i<j<l$.
  \item For $k=6$ we have the same homology classes as in~$(2)$ and in
   addition the class $\pm(2H - E_1 - \cdots - E_6)$.
\end{enumerate}
Note that all these Lagrangian spheres satisfy
Assumption~$\mathscr{L}$ since $N_L = 2$.

The discriminants of these Lagrangian spheres are gathered in
Table~\ref{tb:classes-intro}, the detailed computations being
postponed to~\S\ref{s:examples}. The column under $\lambda_L$ will be
explained in~\S\ref{sb:eigneval}.

\begin{table}[h]
   \begin{center}
      \begin{tabular}{l | l | r | r}
         & $[L]$ & $\Delta_L$ & $\lambda_L$ \\
         \hline
         $M_2$ & $\pm(E_1 - E_2)$ & 5 & -1 \\ 
         \hline
         $M_3$ & $\pm(E_i - E_j)$ & 4 & -2 \\ 
         & $\pm(H - E_1 - E_2 - E_3)$ & -3 & -3 \\ 
         \hline 
         $M_4$ & $\pm(E_i - E_j)$ & 1 & -3 \\ 
         & $\pm(H - E_i - E_j - E_l)$ & 1 & -3 \\ 
         \hline
         $M_5$ & $\pm(E_i - E_j)$ & 0 & -4 \\ 
         & $\pm(H - E_i - E_j - E_l)$ & 0 & -4 \\   
         \hline
         $M_6$ & $\pm(E_i - E_j)$ & 0 & -6 \\ 
         & $\pm(H - E_i - E_j - E_l)$ & 0 & -6 \\ 
         & $\pm(2H - E_1 - \ldots - E_6)$ & 0 & -6
      \end{tabular}
      \vspace{4mm}
      \caption{Classes representing Lagrangian spheres and their
        discriminants.}
      \label{tb:classes-intro}
   \end{center}
\end{table}

The Lagrangian spheres in the three homology classes $E_i - E_j$,
$i<j$, of $M_3$ all have the same discriminant. This can also be seen
by noting that one can choose three Lagrangian spheres $L_1, L_2,
L_3$, one in each of these homology classes so that every pair of them
intersects transversely at exactly one point. The equality of their
discriminants as well (as the fact that they are perfect squares)
follows then by Corollary~\ref{c:del_1=del_2}. We elaborate more on
these examples in~\S\ref{s:examples}.

\subsubsection{Lagrangian spheres in hypersurfaces of
  ${\mathbb{C}}P^{n+1}$} \label{sbsb:exp-intro-hypsurf} Let $M^{2n}
\subset \mathbb{C}P^{n+1}$ be a hypersurface of degree $d \leq n+1$
endowed with the induced symplectic form. By the assumption on $d$,
$M$ is monotone (in fact Fano) and the minimal Chern number is $C_M =
n+2-d$. Note that when $d \geq 2$, $M$ contains Lagrangian spheres.
Assume further that $n\geq 3$, and $d \geq 3$.  Let $L \subset M$ be a
Lagrangian sphere, hence $[L]$ belongs to the primitive homology of
$M$ (see~\cite{Gr-Ha:alg-geom, Voi:hodge-book-I}). Using the
description of the quantum homology of $M$
from~\cite{Co-Ji:QH-hpersurf, Gi:equivariant} we obtain $[L]^{*3}=0$.

Whenever $n$ is a multiple of $2C_M = 2(n+2-d)$ the Lagrangian spheres
$L \subset M$ satisfy Assumption~$\mathscr{L}$, hence the discriminant
is defined and we obtain $\Delta_L = 0$.

Consider now the case $d=2$, i.e. $M$ is the quadric of complex
dimension $n$, and let $S \subset M$ be a Lagrangian sphere. We have
$C_M = n$, so case~\eqref{i:div} of Theorem~\ref{t:cubic-eq-sphere}
applies. If $n=$~odd, then $H_n(M)=0$, hence $[S]=0$. If $n=$~even,
then from the quantum product in the quadric we obtain:
\[ [S]^{*3} = (-1)^{\frac{n(n-1)}{2}+1}4 [S] q^n.
\]

More details on all the above calculations are given
in~\S\ref{s:examples}.

\subsection*{Acknowledgments}
We would like to thank Jean-Yves Welschinger for a discussion
convincing us that all Lagrangian tori in symplectic $4$-manifolds
with $b_2^+=1$ are null-homologous.


%% file: organization-ar.tex
\subsection*{Organization of the paper}
The rest of the paper is organized as follows.
In~\S\ref{s:floer-setting} we briefly recall the necessary ingredients
from Lagrangian Floer and quantum homologies used in the sequel.
In~\S\ref{sb:more-discr} we also give more details on the
discriminant. \S\ref{s:lag-cubic-eq} is devoted to the Lagrangian
cubic equation. We prove in that section more general versions of
Theorems~\ref{t:cubic-eq} and~\ref{t:rel-to-discr}. Then
in~\S\ref{sb:prf-cubic-eq-S} we prove Theorem~\ref{t:cubic-eq-sphere}.
We also prove in~\S\ref{sb:further-res} additional corollaries derived
from these theorems. In~\S\ref{s:disc-lcob} we study the discriminant
in the realm of Lagrangian cobordism and prove Theorem~\ref{t:cob} and
Corollary~\ref{c:del_1=del_2}.  \S\ref{s:examples} is dedicated to
examples. We briefly explain how to construct Lagrangian spheres in
various homology classes on symplectic Del Pezzo surfaces and carry
out the calculation of the discriminants of those Lagrangians. We
discuss some higher dimensional examples too.  In~\S\ref{s:coeffs} we
explain an extension of the discriminant and the Lagrangian cubic
equation over a more general ring of coefficients that takes into
account the different homology classes of the holomorphic curves that
contribute to our invariants.  In~\S\ref{sb:examples-revisited} we
recalculate some of the examples from~\S\ref{s:examples} over this
ring. In~\S\ref{s:enum} we discuss the relation of the discriminant to
the enumerative geometry of holomorphic disks.  Finally,
in~\S\ref{s:non-monotone} we consider the non-monotone case and state
a version of Theorem~\ref{t:cubic-eq-sphere} for not necessarily
monotone Lagrangian 2-spheres.

%% file: floer-setting.tex
\section{Lagrangian Floer theory} \label{s:floer-setting} \cntrs

Here we briefly recall some ingredients from Floer theory that are
relevant for this paper. These include Lagrangian Floer homology and
especially its realization as Lagrangian quantum homology (a.k.a pearl
homology). The reader is referred to~\cite{Oh:HF1, Oh:spectral,
  FO3:book-vol1, FO3:book-vol2, Bi-Co:rigidity, Bi-Co:lagtop} for more
details.

\subsection{Monotone symplectic manifolds and Lagrangians} \cntrsb
\label{sb:monotone}

Let $(M, \omega)$ be a symplectic manifold. Denote by $c_1 \in H^2(M)$
the first Chern class of the tangent bundle $T(M)$ of $M$. Denote by
$H_2^S(M)$ the image of the Hurewicz homomorphism $\pi_2(M)
\longrightarrow H_2(M)$. We call $(M, \omega)$ {\em monotone} if there
exists a constant $\vartheta>0$ such that
$$A_{\omega} = \vartheta I_{c_1},$$ where
$A_{\omega}: H_2^S(M) \longrightarrow \mathbb{R}$ is the homomorphism
defined by integrating $\omega$ over spherical classes and $I_{c_1}$
is viewed as a homomorphism $H_2^S(M) \longrightarrow \mathbb{Z}$. We
denote by $C_M$ the positive generator of the subgroup
$\textnormal{image\,} I_{c_1} \subset \mathbb{Z}$ so that
$\textnormal{image\,} I_{c_1} = C_M \mathbb{Z}$. If
$\textnormal{image\,} I_{c_1} = 0$ we set $C_M = \infty$.

$L \subset M$ a Lagrangian submanifold. Denote by $H_2^D(M,L)$ the
image of the Hurewicz homomorphism $\pi_2(M,L) \longrightarrow
H_2(M,L)$.  We say that $L$ is {\em monotone} if there exists a
constant $\rho >0$ such that
$$A_{\omega} = \rho \mu,$$ where $A_{\omega}:H_2^D(M,L) \longrightarrow
\mathbb{R}$ is the homomorphism defined by integrating $\omega$ over
homology classes and $\mu: H_2^D(M,L) \longrightarrow \mathbb{Z}$ is
the Maslov index homomorphism. We denote by $N_L$ the positive
generator of the subgroup $\textnormal{image\,} \mu \subset
\mathbb{Z}$ so that $\textnormal{image\,}\mu = N_L \mathbb{Z}$.

Finally, denote by $j: H_2^S(M) \longrightarrow H_2^D(M,L)$ the
obvious homomorphism. Then we have $\mu(j(A)) = 2I_{c_1}(A)$ for every
$A \in H_2^S(M)$. Therefore, if $L$ is a monotone Lagrangian and
$I_{c_1} \neq 0$ then $(M, \omega)$ is also monotone and we have $N_L
\mid 2C_M$. When $\pi_1(L) = \{1\}$ we actually have $N_L = 2C_M$.

\subsection{Floer homology and Lagrangian quantum homology}
\label{sb:HF} \cntrsb

Let $L \subset M$ be a closed monotone Lagrangian submanifold with $2
\leq N_L \leq \infty$. Under the additional assumptions that $L$ is
spin one can define the self Floer homology $HF(L,L)$ with
coefficients in $\mathbb{Z}$. This group is cyclically graded, with
grading in $\mathbb{Z} / N_L \mathbb{Z}$.

From the point of view of the present paper it is more natural to work
with Lagrangian quantum homology $QH(L)$ rather than with the Floer
homology $HF(L,L)$. This is justified by the fact that for an
appropriate choice of coefficients we have an isomorphism of rings
$QH(L) \cong HF(L,L)$. The advantage of $QH(L)$ in our context is that
it bears a simple and explicit relation to the singular homology
$H(L)$ of $L$. For example, under certain circumstances (relevant for
our considerations) and with the right coefficient ring, $QH(L)$ can
be viewed as a deformation of the singular homology ring $H(L)$
endowed with the intersection product.

We will now summarize the most basic properties of Lagrangian quantum
homology. The reader is referred to~\cite{Bi-Co:rigidity,
  Bi-Co:lagtop} for the foundations of the theory.

Denote by $\Lambda = \mathbb{Z}[t^{-1}, t]$ the ring of Laurent
polynomials over $\mathbb{Z}$ graded so that the degree of $t$ is
$|t|=-N_L$. We denote by $QH^{\#}(L)$ the Lagrangian quantum homology
of $L$ with coefficients in $\mathbb{Z}$ and by $QH(L; \Lambda)$ the
one with coefficients in $\Lambda$. Thus $QH^{\#}(L)$ is cyclically
graded modulo $N_L$ and $QH(L;\Lambda)$ is $\mathbb{Z}$-graded and
$N_L$-periodic, i.e.  $QH_i(L;\Lambda) \cong QH_{i-N_L}(L;\Lambda)$,
the isomorphism being given by multiplication by $t$. And we have
$QH_i(L; \Lambda) \cong QH^{\#}_{i \pmod{N_L}}(L)$, hence the grading
on $QH(L;\Lambda)$ is an unwrapping of the cyclic grading of
$QH^{\#}(L)$. Sometimes, when the context is clear we will write
$QH(L)$ for $QH(L;\Lambda)$.

The Lagrangian quantum homology has the following algebraic
structures. There exists a quantum product $$QH_i(L;\Lambda) \otimes
QH_j(L;\Lambda) \longrightarrow QH_{i+j-n}(L;\Lambda), \quad \alpha
\otimes \beta \longmapsto \alpha * \beta,$$ which turns
$QH(L;\Lambda)$ into a unital associative ring with unity $e_L \in
QH_n(L;\Lambda)$.

We now briefly recall relations between the Lagrangian and ambient
quantum homologies. Denote by $R = \mathbb{Z}[q^{-1}, q]$ the ring of
Laurent polynomials in the variable $q$, whose degree we set to be
$|q|=-2$. Denote by $QH(M;R)$ the quantum homology of $M$ with
coefficients in $R$, endowed with the quantum product $*$. The
Lagrangian quantum homology $QH(L; \Lambda)$ is a module over the
subring $QH(M; \Lambda) \subset QH(M;R)$, where $\Lambda$ is embedded
in $R$ by $t \mapsto q^{N_L/2}$. We denote this operation by
$$QH_i(M; \Lambda) \otimes QH_j(L; \Lambda)
\longrightarrow QH_{i+j-2n}(L; \Lambda), \quad a \otimes \alpha
\longmapsto a*\alpha.$$ The reason for using the same notation $*$ as
for the quantum product on $L$ is that the module operation is
compatible with the latter in the following sense:
\begin{equation} \label{eq:alg-identity} c*(\alpha*\beta) =
   (c*\alpha)*\beta = (-1)^{(2n-|c|)(n-|\alpha|)}\alpha*(c*\beta),
\end{equation}
for every $c \in QH(M;\Lambda)$, $\alpha, \beta \in QH(L;\Lambda)$.
\pbred{Note that the sign conventions in~\eqref{eq:alg-identity} are
  compatible with the standard sign conventions for the intersection
  product in singular homology.}

The proof of identity~\eqref{eq:alg-identity} has been carried out
in~\cite{Bi-Co:qrel-long, Bi-Co:rigidity} over $\mathbb{Z}_2$ (hence
without taking signs into account), and the same proof carries over in
a straightforward way over $\mathbb{Z}$ using~\cite{Bi-Co:lagtop}.
Thus $QH(L;\Lambda)$ is an algebra (in the graded sense) over
$QH(M;\Lambda)$.

There is also a quantum inclusion map
$$i_L: QH_i(L; \Lambda) \longrightarrow QH_i(M; \Lambda),$$ which
is linear over the ring $QH(M; \Lambda)$, i.e. $i_L(c*\alpha) =
c*i_L(\alpha)$ for every $c \in QH(M;\Lambda)$ and $\alpha \in
QH(L;\Lambda)$. An important property of $i_L$ is that $i_L(e_L) =
[L]$, see~\cite{Bi-Co:lagtop}.

Next there is an augmentation morphism
$$\epsilon_L: QH(L;\Lambda) \longrightarrow \Lambda,$$ which is
induced from a chain level extension of the classical augmentation.
The augmentation satisfies the following identity,
see~\cite{Bi-Co:rigidity}:
\begin{equation} \label{eq:kron-aug} \langle PD(h), i_L(\alpha)
   \rangle = \epsilon_L(h * \alpha), \quad \forall h \in H_*(M), \;
   \alpha \in QH(L; \Lambda),
\end{equation}
where $PD$ stands for Poincar\'{e} duality and $\langle \cdot, \cdot
\rangle$ denotes the Kronecker pairing extended over $\Lambda$ in an
obvious way. Sometimes it will be more convenient to view the
augmentation as a map
$$\widetilde{\epsilon}_L: QH(L;\Lambda) \longrightarrow H_0(L;\Lambda)
= \Lambda [\textnormal{point}].$$ This augmentations $\epsilon_L$ and
$\widetilde{\epsilon}_L$ descend also to $QH^{\#}(L)$ and by slight
abuse of notation we denote them the same:
$$\epsilon_L: QH^{\#}(L) \longrightarrow \mathbb{Z}, \quad
\widetilde{\epsilon}_L: QH^{\#}(L) \longrightarrow H_0(L).$$

As mentioned earlier we will not really use Floer homology in this
paper, but Lagrangian quantum homology instead. The justification for
replacing $HF(L,L)$ by $QH^{\#}(L)$ is due to the PSS isomorphism
$$PSS: HF_*(L,L) \longrightarrow QH^{\#}_*(L).$$ This is a ring
isomorphism which intertwines the Donaldson product and the quantum
product on $QH^{\#}(L)$.  A version of $PSS$ works with coefficients
in $\Lambda$ too. For more details on the PSS isomorphism
see~\cite{Alb:PSS, Bar-Cor:NATO, Cor-La:Cluster-1, Bi-Co:rigidity}.
See also~\cite{Hu-La:Seidel-morph, Hu-La-Le:monodromy} for the
extension to $\mathbb{Z}$-coefficients.

Finally, we remark that everything mentioned above in this section
continues to hold (with obvious modifications) also with other choices
of base rings, replacing $\mathbb{Z}$ by $\mathbb{Q}$ or $\mathbb{C}$.
For $K = \mathbb{Q}$ or $\mathbb{C}$ we write $\Lambda_K =
K[t^{-1},t]$, $R_{K} = K[q^{-1}, q]$ for the associated rings of
Laurent polynomials and by $HF(L,L; \Lambda_K)$, $QH(L;\Lambda_K)$ and
$QH(M; R_K)$ the corresponding homologies.  Sometimes it will be
useful to drop the Laurent polynomial rings $\Lambda_K$ and $R_K$ and
simply work with $HF(L,L;K)$, $QH(L;K)$ and $QH(M;K)$. Another
variation that will be used in the sequel is to replace $\Lambda_K$
and $R_K$ by polynomial rings (rather than Laurent polynomials), i.e.
work with coefficients in $\Lambda^+_K = K[t]$ and $R^+_K = K[q]$.
See~\cite{Bi-Co:rigidity, Bi-Co:Yasha-fest, Bi-Co:lagtop} for a
detailed account on this choice of coefficients. When the base ring
$K$ is obvious we will abbreviate $Q^+H(L) := QH(L; \Lambda^+_K)$ and
similarly for $Q^+H(M)$. (There has been only one exception to this
notation. In the introduction~\S\ref{s:intro} we denoted by $QH(M)$
the quantum homology $QH(M; R^+)$ in order to facilitate the notation,
but henceforth we will stick to the notation we have just described.)
The homologies of the type $Q^+H$ will be called {\em positive quantum
  homologies}. Again, everything described above continues to work for
the positive versions of quantum homologies with one important
exception: the PSS isomorphism does not hold over $\Lambda^+_K$ (at
least not for a straightforward version of Floer homology).

\subsection{Proof of Proposition~\ref{p:criterion}}
\label{sb:prf-p-criterion} \cntrsb
The proof appeals to a spectral sequence for calculating Lagrangian
quantum homology which is rather standard in symplectic topology. For
the sake of readability we have included in~\S\ref{s:app-calc} a
summary of the main ingredients of this technique.

Let $\{E^r_{p,q}, d^r\}_{_{r \geq 0}}$ be the spectral sequence
described in~\S\ref{sb:spec-seq}. By Theorem~\ref{t:spectral-seq} and
the assumptions of Proposition~\ref{p:criterion} we see that the $E^1$
terms of the sequence has the following form:
\[
\bigoplus_{p + q = n} E^1_{p,q} = (H_n(L; \mathbb{Q})\otimes P_0)
\oplus (H_0(L; \mathbb{Q})\otimes P_{n}).
\]
It now follows easily that the dimension of
$QH_n(L;\Lambda_{\mathbb{Q}})$ as a $\mathbb{Q}$-vector space is at
most $2$. We will now show that the dimension is exactly $2$.

We first claim that the unity is not trivial, $e_L \neq 0 \in
QH_n(L;\Lambda_{\mathbb{Q}})$. To see this consider the quantum
inclusion map $i_L:QH_n(L;\Lambda_{\mathbb{Q}}) \longrightarrow
QH_n(M;R_{\mathbb{Q}})$ from~\S\ref{sb:HF}. It is well
known~\cite{Bi-Co:lagtop} that $i_L(e_L) = [L]$. As $[L] \neq 0$ it
follows that $e_L \neq 0$.

By Poincar\'{e} duality there exists a class $c \in H_n(M;\mathbb{Q})$
such that $c \cdot [L] \neq 0$. Put $x: = c * e_L \in
QH_0(L;\Lambda_{\mathbb{Q}})$. From~\eqref{eq:kron-aug} we get that
$\epsilon_L(x) \neq 0$. This implies that the two elements $xt^{-\nu},
e_L \in QH_n(L; \Lambda_{\mathbb{Q}})$ are linearly independent. It
follows that $\dim QH_n(L;\Lambda_{\mathbb{Q}}) = 2$.

From the above it now follows that the rank of of $QH^{\#}_n(L)$ is
$2$.  Finally, from the PSS isomorphism we obtain that $HF_n(L,L)$ has
rank $2$. \Qed

\subsection{Eigenvalues of $c_1$ and Lagrangian submanifolds}
\label{sb:eigneval} \cntrsb

Let $L \subset M$ be a closed spin monotone Lagrangian submanifold
with $QH(L; \mathbb{C}) \neq 0$. Assume in addition that $N_L=2$. With
these assumptions one can define an invariant $\lambda_L \in
\mathbb{Z}$ which counts the number of Maslov-$2$ pseudo-holomorphic
disks $u:(D,
\partial D) \longrightarrow (M,L)$ whose boundary $u(\partial D)$ pass
through a generic point $p \in L$. The value of $\lambda_L$ turns out
to be independent of the almost complex structure as well as of the
generic point $p$.  See~\cite{Bi-Co:lagtop} for more details.  We
extend the definition of $\lambda_L$ to the case $N_L > 2$ by setting
$\lambda_L=0$.

Consider now the following operator $$P: QH(L;\Lambda_\mathbb{C})
\longrightarrow QH(L; \Lambda_{\mathbb{C}}), \quad \alpha \longmapsto
PD(c_1)*\alpha,$$ where $PD$ stands for Poincar\'{e} duality. By abuse
of notation we have denoted here by $c_1 \in H^2(M; \mathbb{C})$ the
image of the first Chern class of $T(M)$ under the change of
coefficients map $H^2(M;\mathbb{Z}) \to H^2(M;\mathbb{C})$.

The following is well known:
\begin{enumerate}
  \item If $N_L = 2$, then $P(\alpha) = \lambda_L \alpha t$ for every
   $\alpha \in QH(L;\Lambda_{\mathbb{C}})$.
  \item If $N_L > 2$, then $P \equiv 0$.
\end{enumerate}
For the proof of~$(1)$, See~\cite{Aur:t-duality} for a special case
(where the statement is attributed to folklore, in particular also to
Kontsevich and to Seidel) and~\cite{Sher:Fano} for the general case.
As for~$(2)$, it follows immediately from the fact that the
restriction of $c_1$ to $L$ vanishes, $c_1|_{L} = 0 \in
H^2(L;\mathbb{C})$, together with degree reasons.

Denote by $\mathcal{I}_L \subset QH(M;R_{\mathbb{C}})$ the image of
the quantum inclusion map $i_L: QH(L;\Lambda_{\mathbb{C}})
\longrightarrow QH(M;R_{\mathbb{C}})$. Note that $\mathcal{I}_L$ is an
ideal of the ring $QH(M; R_{\mathbb{C}})$.

\begin{prop} \label{p:I_L-lambda} $\mathcal{I}_L \neq 0$ iff $QH(L;
   \Lambda_{\mathbb{C}}) \neq 0$ and in that case $\lambda_L$ is an
   eigenvalue of the operator
   $$Q: QH(M; R_{\mathbb{C}}) \longrightarrow QH(M; R_{\mathbb{C}}),
   \quad a \longmapsto PD(c_1)*a q^{-1}.$$ Moreover, $\mathcal{I}_L$
   is a subspace of the eigenspace of $Q$ corresponding to the
   eigenvalue $\lambda_L$. In particular if $[L] \neq 0 \in
   H_n(M;\mathbb{C})$ then $[L]$ is an eigenvector of $Q$
   corresponding to $\lambda_L$.
\end{prop}

\begin{rem}
   Denote by $Q': QH(M;\mathbb{C}) \longrightarrow QH(M;\mathbb{C})$
   the same operator as $Q$ but acting on $QH(M;\mathbb{C})$ instead
   of $QH(M;\Lambda_{\mathbb{C}})$. Similarly, denote by
   $\mathcal{I}'_L \subset QH(M;\mathbb{C})$ the image of $i_L$.  The
   statement of Proposition~\ref{p:I_L-lambda} continues to hold for
   $Q'$ and $\mathcal{I}'_L$. Moreover, if $[L] \neq 0$ then
   $$\dim_{\mathbb{C}} \mathcal{I}'_L \geq 2,$$ hence the multiplicity of the
   eigenvalue $\lambda_L$ with respect to the operator $Q'$ is at
   least $2$. Indeed, $[L] = i_L(e_L) \in \mathcal{I}'_L$. Now take $c
   \in H_n(M;\mathbb{C})$ with $c \cdot [L] \neq 0$. As
   $\mathcal{I}'_L$ is an ideal we have $c*[L] \in \mathcal{I}'_L$.
   But $c*[L] = \#(c \cdot [L]) [\textnormal{point}] +
   (\textnormal{other terms})$, hence $c*[L]$ is not proportional to
   $[L]$. (Here $\#(c \cdot [L])$ stands for the intersection number
   of $c$ and $[L]$.)
\end{rem}

\begin{proof}[Proof of Proposition~\ref{p:I_L-lambda}]
   Assume that $QH(L;\Lambda_{\mathbb{C}}) \neq 0$. By duality for
   Lagrangian quantum homology there exists $x \in
   QH_0(L;\Lambda_{\mathbb{C}})$ with $\epsilon_L(x) \neq 0$.
   (See~\cite{Bi-Co:rigidity}, Proposition~4.4.1. The proof there is
   done over $\mathbb{Z}_2$ but the extension to any field is
   straightforward in view of~\cite{Bi-Co:lagtop}).

   From~\eqref{eq:kron-aug} (with $h=[M]$ and $\alpha = x$) it follows
   that $i_L(x) \neq 0$, hence $\mathcal{I}_L \neq 0$. The opposite
   assertion is obvious.

   The statement about the eigenspace of $Q$ follows immediately from
   the discussion about the operator $P$ and the fact that $i_L$ is a
   $QH(M;R_{\mathbb{C}})$-module map.

   Finally, note that $[L] \in \mathcal{I}_L$ since $[L] = i_L(e_L)$.
\end{proof}

The following observation shows that the eigenvalues corresponding to
different Lagrangians coincide under certain circumstances.
\begin{prop} \label{p:lambda-lambda'} Let $L, L' \subset M$ be two
   closed monotone spin Lagrangian submanifolds. Assume that $[L]
   \cdot [L]' \neq 0$. Then $\lambda_{L} = \lambda_{L'}$.
\end{prop}
\begin{proof}
   We view $[L], [L']$ as elements of $QH_n(M;\mathbb{C})$. We have
   $$\textnormal{PD}(c_1)*([L]*[L']) = (\textnormal{PD}(c_1)*[L])*[L'] =
   \lambda_{L}[L]*[L'].$$ At the same time, since
   $|\textnormal{PD}(c_1)| = $ even we also have
   $$\textnormal{PD}(c_1)*([L]*[L']) = [L]*(\textnormal{PD}(c_1)*[L']) =
   \lambda_{L'} [L]*[L'].$$ Since $[L]\cdot [L'] \neq 0$ we have
   $[L]*[L'] \neq 0$ and the results follows.
\end{proof}

\subsection{More on the discriminant} \label{sb:more-discr} \cntrsb

\subsubsection{Well-definedness} \label{sbsb:discr-well-def} We start
with showing that the discriminant, as defined
in~\S\ref{sb:discr-intro} is independent of the choices of $p$ and
$x$. We first fix $p$ and show independence of its lift $x$. Indeed if
$y$ is another lift of $p$ then $y = x + r$ for some $r \in
\mathbb{Z}$. A straightforward calculation shows that
$$\sigma(p, y) = \sigma(p,x) + 2r, \quad \tau(p, y) = \tau(p,x) -
\sigma(p,x)r - r^2.$$ Another direct calculation shows that
$$\sigma(p,y)^2 + 4 \tau(p, y) = \sigma(p,x)^2 + 4\tau(p,x).$$
Assume now that $p' \in A / \mathbb{Z}$ is a different generator. We
then have $p' = -p$ and so we can choose $x' = -x$ as a lift of $p'$.
It easily follows that $$\sigma(p', x') = - \sigma(p,x), \quad
\tau(p', x') = \tau(p,x),$$ hence again $\sigma(p',x')^2 + 4 \tau(p',
x') = \sigma(p,x)^2 + 4\tau(p,x)$.
\Qed

\subsubsection{The discriminant determines the isomorphism type of a
  quadratic algebra} \label{sbsb:discr-inv-class}

\begin{lem}
   Let $A$ and $B$ be two quadratic algebras over $\mathbb{Z}$. Then
   $A$ is isomorphic to $B$ if and only if $\Delta_A = \Delta_B$.
\end{lem}
\begin{proof}
   Fix group isomorphisms
   \[
   A \cong \mathbb{Z} \oplus \mathbb{Z}x, \quad B \cong \mathbb{Z}
   \oplus \mathbb{Z} x',
   \]
   where $x \in A$, $x' \in B$ and write $x^2 = \sigma x + \tau$,
   $x'^2 = \sigma' x + \tau'$ with $\sigma, \sigma', \tau, \tau' \in
   \mathbb{Z}$. Define two quadratic monic polynomials with integral
   coefficients: $$f(X) = X^2 - \sigma X -\tau, \quad g(X) = X^2 -
   \sigma' X - \tau'.$$ The map $\mathbb{Z}[X] \longrightarrow A$,
   induced by $X \longmapsto x$, descends to a map $\mathbb{Z}[X] / (
   f(X) ) \longrightarrow A$ which is easily seen to be an isomorphism
   or rings. In a similar way we obtain a ring isomorphism
   $\mathbb{Z}[X] / ( g(X) ) \cong B$. Note that $\Delta_A$ is equal
   to the discriminant of $f$ and $\Delta_B$ to the discriminant of
   $g$.

   Assume that $A \cong B$. It is easy to see that all the ring
   isomorphisms $\mathbb{Z}[X] / ( f(X) ) \cong \mathbb{Z}[X] / ( g(X)
   )$ are induced by $X \longmapsto \pm X + r$, where $r \in
   \mathbb{Z}$. It follows that $g(X) = f(\pm X+r)$ for some $r \in
   \mathbb{Z}$ and thus $f$ and $g$ have the same discriminants, hence
   $\Delta_A = \Delta_B$.

   Conversely, assume that $\Delta_A = \Delta_B$. Then
   $$\sigma^2 + 4\tau = \sigma'^2 + 4 \tau',$$ hence
   $\sigma$ and $\sigma'$ have the same parity. Set $r = (\sigma -
   \sigma') /2 \in \mathbb{Z}$ and consider the ring homomorphism
   $\varphi: \mathbb{Z}[X] \longrightarrow \mathbb{Z}[X]$ induced by
   $X \longrightarrow X - r$. A simple calculation shows that
   $\varphi(f) = g$, hence it descends to $\bar{\varphi} :
   \mathbb{Z}[X] / ( f(X) ) \longrightarrow \mathbb{Z}[X] / ( g(X) )$.
   It is easy to see that $\bar{\varphi}$ is invertible.
\end{proof}



\subsubsection{A useful extension over other rings}
\label{sbsb:discr-ext-Q}

Let $A$ be a quadratic algebra over $\mathbb{Z}$ as described
in~\S\ref{sb:discr-intro}. Let $K$ be a commutative ring which extends
$\mathbb{Z}$, i.e. we have $\mathbb{Z} \subset K$ as a subring. For
simplicity we will assume that $K$ is torsion-free. We will mainly
consider $K = \mathbb{Q}$ or $K = \mathbb{C}$.  Write $A_{K} = A
\otimes K$.

For practical purposes it will be sometimes useful to calculate
$\Delta_A$ using $A_{K}$ rather than via $A$ itself. This can be done
as follows.  From the sequence~\eqref{eq:ex-seq-A} we obtain the
following exact sequence:
\begin{equation}
   0 \longrightarrow K \longrightarrow A_{K}
   \xrightarrow{\;\;\epsilon \;\;} K p \longrightarrow 0,
\end{equation}
where as before, $\epsilon$ is the projection to the quotient and $p$
stands for a generator of $A / \mathbb{Z} \subset A_{K}/K$. Pick a
lift $x \in A_{K}$ of $p$ and define $\sigma(p,x), \tau(p,x)$ by the
same recipe as in~\S\ref{sb:discr-intro}, only that now these two
numbers belong to $K$ rather than to $\mathbb{Z}$. A simple
calculation, similar to~\S\ref{sbsb:discr-well-def} above shows that
we still have $\Delta_A = \sigma(p,x)^2 + 4\tau(p,x)$ (and of course
despite the calculation being done in $K$ we still have $\Delta_A \in
\mathbb{Z}$).

\begin{rem}
   It is essential here that the generator $p$ is integral, i.e. that
   $p \in A_{K}/K$ was chosen to come from $A/\mathbb{Z}$. If we allow
   to replace $p$ by any non-trivial element of $A_{K}/K$ then the
   corresponding discriminant will depend on that choice, but not on
   the choice of the lift $x$.  In fact, if $p' = c p$, $c \in K$ then
   the discriminants corresponding to $p'$ and $p$ are related by
   $\Delta(p') = c^2 \Delta(p)$. Therefore, when $K = \mathbb{Q}$ for
   example, the sign of the discriminant is an invariant of
   $A_{\mathbb{Q}}$. The algebraic properties of $A_{\mathbb{Q}}$
   change depending on the sign of the discriminant and whether it is
   a perfect square or not.
\end{rem}

\subsubsection{The case of $A = QH^{\#}_n(L)$} \label{sbsb:Q=QH_n} Let
$L \subset M$ be a Lagrangian submanifold satisfying
conditions~$(1)$~--~$(3)$ of Assumption~$\mathscr{L}$. Fix a spin
structure on $L$.  Denote by $e_L \in QH^{\#}_n(L)$ the unity. Without
loss of generality we may assume that $QH^{\#}_n(L)$ is torsion-free,
otherwise we just replace it by $QH^{\#}_n(L)/T$, where $T$ is the
torsion ideal. Thus $QH^{\#}_n(L)$ is a quadratic algebra over
$\mathbb{Z}$.

By duality for Lagrangian quantum homology~\cite{Bi-Co:rigidity,
  Bi-Co:lagtop}, the augmentation $\widetilde{\epsilon}_L :
QH^{\#}_0(L) \longrightarrow H_0(L;\mathbb{Z})$ is surjective. Keeping
in mind that in our case $QH^{\#}_0(L) = QH^{\#}_n(L)$ (since $N_L
\mid n$) we obtain the following exact sequence:
$$0 \longrightarrow \mathbb{Z} e_L \longrightarrow QH^{\#}_n(L)
\xrightarrow{\;\;\widetilde{\epsilon}_L \;\;} H_0(L;\mathbb{Z})
\longrightarrow 0.$$ Let $K$ be a torsion-free commutative ring that
contains $\mathbb{Z}$. Let $p = [\textnormal{point}] \in
H_0(L;\mathbb{Z})$ be the homology class of a point. Tensoring the
last sequence by $K$ we obtain:
\begin{equation} \label{eq:quad-alg-QHn-2} 0 \longrightarrow K e_L
   \longrightarrow QH^{\#}_n(L;K)
   \xrightarrow{\;\;\widetilde{\epsilon}_L \;\;} K p \longrightarrow
   0.
\end{equation}

In order to calculate $\Delta_L$, choose a lift $x \in QH^{\#}_n(L;K)$
of $p$ with respect to $\widetilde{\epsilon}_L$. Then we have
\begin{equation} \label{eq:x*x-QH} x*x = \sigma(p,x) x + \tau(p,x)e_L,
\end{equation}
with some $\sigma(p,x), \tau(p,x) \in K$. The discriminant can then be
calculated by $$\Delta_L = \sigma(p,x)^2 + 4 \tau(p,x).$$

In the following we will need to use the equality~\eqref{eq:x*x-QH}
but in $QH_n(L;\Lambda_K)$ rather than in $QH^{\#}_n(L;K)$. We have
$QH_0(L;\Lambda_K) = t^{\nu} QH_n(L;\Lambda_K)$, with $\nu = n/N_L$.
The lift $x$ of $p$ has now to be chosen in $QH_0(L;\Lambda_K)$ and
the previous equation now takes place in $QH_0(L;\Lambda_K)$ and has
the following form:
\begin{equation} \label{eq:x*x-QH-t} x*x = \sigma(p,x) x t^{\nu} +
   \tau(p,x)e_L t^{2 \nu}.
\end{equation}

Finally, we mention that sometimes it is more convenient to define the
discriminant using the positive Lagrangian quantum homology $QH(L;
\Lambda^+_K)$ rather than $QH(L; \Lambda_K)$. The resulting
discriminant is obviously the same.


%% file: lag-cubic-eq.tex
\section{The Lagrangian cubic equation} \label{s:lag-cubic-eq} \cntrs

We begin by proving the following result that generalizes
Theorems~\ref{t:cubic-eq} and~\ref{t:rel-to-discr}.
Theorem~\ref{t:cubic-eq-sphere} will be proved
in~\S\ref{sb:prf-cubic-eq-S} below.

\begin{thm}\label{t:cubic_eq-ccl}
   Let $L \subset M$ be a Lagrangian submanifold satisfying
   conditions~$(1)$~--~$(3)$ of Assumption~$\mathscr{L}$. Assume in
   addition that $[L] \neq 0 \in H_n(M;\mathbb{Q})$. Let $c \in H_n(M;
   \mathbb{Z})$ be a class satisfying $\xi := \#(c \cdot [L]) \neq 0$.
   Then there exist unique \cmred{constants $\sigma_{c,L} \in
     \tfrac{1}{\xi^2}\mathbb{Z}$, $\tau_{c,L} \in
     \tfrac{1}{\xi^3}\mathbb{Z}$} such that the following equation
   holds in $QH(M;R_{\mathbb{Q}}^+)$:
   \begin{equation} \label{eq:cubic_eq_ccL} c*c*[L] - \xi \sigma_{c,L}
      \, c*[L] q^{n/2} - \xi^2 \tau_{c,L} \, [L] q^{n} = 0.
   \end{equation}
   The coefficients $\sigma_{c,L}, \tau_{c,L}$ are related to the
   discriminant of $L$ by $\Delta_L = \sigma_{c,L}^2 + 4 \tau_{c,L}$.
   \cmred{If $\xi$ is square-free, then $\sigma_{c,L} \in
     \tfrac{1}{\xi} \mathbb{Z}$ and $\tau_{c,L} \in \tfrac{1}{\xi^2}
     \mathbb{Z}$.}  Moreover, $\sigma_{c,L}$ can be expressed in terms
   of genus $0$ Gromov-Witten invariants as follows:
   \begin{equation} \label{eq:GW-CCL} \sigma_{c,L} = \frac{1}{\xi^2}
      \sum_A GW_{A,3}(c,c,[L]),
   \end{equation}
   where the sum is taken over all classes $A \in H_2(M)$ with
   $\langle c_1, A \rangle = n/2$.
\end{thm}
As we will see soon, Theorem~\ref{t:cubic-eq} follows immediately from
Theorem~\ref{t:cubic_eq-ccl} by taking $c=[L]$ and in the notation of
Theorem~\ref{t:cubic-eq} we have $\sigma_L = \sigma_{[L],L}$, $\tau_L
= \tau_{[L], L}$. Recall also from Corollary~\ref{c:sig=0} that if $L$
is a Lagrangian sphere then $\sigma_L = 0$ (see also
Theorem~\ref{t:cubic-eq-sphere}, case~\eqref{i:div}). We remark that
in contrast to $\sigma_L$, the constants $\sigma_{c,L}$ {\em might not
  vanish} for general $c \neq [L]$. See for example~\S\ref{sbsb:M_2},
for an explicit calculation of the constants $\sigma_{c,L},
\tau_{c,L}$ (for all possible $c$'s) for Lagrangian spheres in the
blow-up of ${\mathbb{C}}P^2$ at two points.
\begin{proof}[Proof of Theorem~\ref{t:cubic_eq-ccl}]
   Fix a spin structure on $L$. In view of~\S\ref{sb:HF} we replace
   $HF_n(L,L; \mathbb{Q})$ by $QH_n(L; \Lambda_{\mathbb{Q}})$. By
   assumption, this is a $2$-dimensional vector space over
   $\mathbb{Q}$. Recall also that $QH_0(L; \Lambda_{\mathbb{Q}}) \cong
   QH_n(L; \Lambda_{\mathbb{Q}})$. Put $$x := \tfrac{1}{\xi} c * e_L
   \in QH_0(L;\Lambda_{\mathbb{Q}}),$$ where $c$ is viewed here as an
   element of $QH_n(M;R_{\mathbb{Q}})$ and $*$ is the module operation
   mentioned in~\S\ref{sb:HF}. Let $p = [\textnormal{point}] \in
   H_0(L;\mathbb{Q})$ be the class of a point. We have
   $$\widetilde{\epsilon}_L(x) = \tfrac{1}{\xi} \#(c \cdot [L])p = p.$$
   It follows that $\{x, e_L t^{\nu}\}$ is a basis for
   $QH_0(L;\Lambda_{\mathbb{Q}})$. Following the recipe
   in~\S\ref{sbsb:Q=QH_n} and formula~\eqref{eq:x*x-QH-t} there exist
   $\sigma_{c,L}, \tau_{c,L} \in \mathbb{Q}$ such that
   \begin{equation} \label{eq:x*x-sig_cL} x*x = \sigma_{c,L} x t^{\nu}
      + \tau_{c,L} e_L t^{2\nu},
   \end{equation}
   where $*$ stands here for the Lagrangian quantum product on
   $QH(L)$. 

   We now apply the quantum inclusion map $i_L$ (see~\S\ref{sb:HF}) to
   both sides of~\eqref{eq:x*x-sig_cL}. We have
   $$i_L(x*x) = \tfrac{1}{\xi^2} i_L((c*e_L)*(c*e_L)) = 
   \tfrac{1}{\xi^2} c*c*i_L(e_L) = \tfrac{1}{\xi^2} c*c*[L].$$ Here we
   have used properties of the operations described in~\S\ref{sb:HF},
   and in particular identity~\eqref{eq:alg-identity}. We also have
   $$i_L(x) = \tfrac{1}{\xi} c*i_L(e_L) = \tfrac{1}{\xi} c*[L].$$
   Recall also that we can view $\Lambda$ as a subring of $R =
   \mathbb{Z}[q, q^{-1}]$ via the embedding $t \longmapsto q^{N_L/2}$,
   so that under this embedding we have $t^{\nu} \longmapsto q^{n/2}$.
   Therefore by applying $i_L$ to~\eqref{eq:x*x-sig_cL} we immediately
   obtain the equation claimed by the theorem. The statement on
   $\Delta_L$ follows at once from~\S\ref{sbsb:Q=QH_n}.

   \cmred{Next we claim that $\xi^2 \sigma_{c,L}, \xi^3 \tau_{c, L}
     \in \mathbb{Z}$ and moreover, if $\xi$ is square-free, then in
     fact $\xi \sigma_{c,L}, \xi^2 \tau_{c, L} \in \mathbb{Z}$. To
     this end we will denote $\Lambda$ by $\Lambda_{\mathbb{Z}}$ to
     emphasize that the ground ring is $\mathbb{Z}$. To prove the
     claim, set $y := \xi x$ and note that $y \in QH_0(L;
     \Lambda_{\mathbb{Z}})$. For $y$ we obtain the resulting equation
     in $QH_{-n}(L; \Lambda_{\mathbb{Z}})$ using~\eqref{eq:x*x-sig_cL}
     \begin{equation} \label{eq:y*y-sig_cL} y * y = \xi \sigma_{c,L} y
        t^{\nu} + \xi^2 \tau_{c,L} e_L t^{2\nu}.
   \end{equation}
   We apply the augmentation morphism $\epsilon_L :
   QH(L;\Lambda_{\mathbb{Z}}) \longrightarrow \Lambda_{\mathbb{Z}}$
   and obtain
   $$\epsilon_L(y * y) = 
   \xi \sigma_{c,L} \epsilon_L (y) t^{\nu} = \xi^2 \sigma_{c,L}
   t^{\nu}. $$ Since the left-hand side lies in $\Lambda_{\mathbb{Z}}$
   it follows that $\xi^2 \sigma_{c,L} \in \mathbb{Z}$. Multiplying
   equation~\eqref{eq:y*y-sig_cL} with $\xi$ we see that $\xi^3
   \tau_{c,L} \in \mathbb{Z}$. We now write $\sigma_{c,L} = u/\xi^2$
   and $\tau_{c,L} = v/ \xi^3$ with $u, v \in \mathbb{Z}$. The
   discriminant is then
    $$ \Delta_L = \frac{u^2}{\xi^4} + 4 \frac{v}{\xi^3} \in \mathbb{Z} $$
    and thus we have $\xi^4 \Delta_L = u^2 + 4 \xi v$. Since $\xi
    \,|\, (u^2 + 4 \xi v)$ it follows that $\xi \,|\, u^2$. If $\xi$
    is square-free then $\xi \,|\, u$ and hence $\xi \sigma_{c,L} = u/
    \xi \in \mathbb{Z}$. Now using equation~\eqref{eq:y*y-sig_cL} we
    see that $y*y - \xi \sigma_{c,L} y t^{\nu} \in QH_{-n}(L;
    \Lambda_{\mathbb{Z}})$ and therefore $\xi^2 \tau_{c,L} \in
    \mathbb{Z}$. }
 

  It remains to prove the statement on the relation between
  $\sigma_{c,L}$ and the Gromov-Witten invariants. For this purpose we
  will need the following Lemma.  We denote by $p_M \in H_0(M)$ the
  class of a point.
   \begin{lem} \label{l:p-class} Let $a, b \in H_*(M)$ be two {\em
        classical} elements of pure degree.  Then
      $$\widetilde{\epsilon}_M(a*b) = \widetilde{\epsilon}_M(a \cdot
      b),$$ where $\cdot$ is the classical intersection product.  In
      particular, the class $p_M$ appears as a summand in $a*b$ if and
      only if $|a|+|b| = 2n$ and $a \cdot b \neq 0$.
   \end{lem}
   We postpone the proof of the Lemma and proceed with the proof of
   the theorem.

   Denote by $k = C_M$ the minimal Chern number of $M$
   (see~\S\ref{sb:monotone}). Write $$c*[L] = c\cdot [L] + \sum_{j\geq
     1} \alpha_{2jk}q^{jk},$$ with $\alpha_{2jk} \in H_{2jk}(M)$. (The
   choice of the sub-indices was made to reflect the degree in
   homology.) Then we have
   $$c*c*[L] = \#(c \cdot [L]) c*p_M + \sum_{j \geq 1} 
   c*\alpha_{2jk} q^{jk},$$ which together
   with~\eqref{eq:cubic_eq_ccL} give:
   \begin{equation} \label{eq:cubic-alpha} \xi \sigma_{c,L}
      c*[L]q^{n/2} + \xi^2 \tau_{c,L} [L]q^n = \#(c \cdot [L]) c*p_M +
      \sum_{j \geq 1} c*\alpha_{2jk} q^{jk}.
   \end{equation}

   Applying $\widetilde{\epsilon}_M$ to~\eqref{eq:cubic-alpha} we
   obtain using Lemma~\ref{l:p-class} that
   \begin{equation} \label{eq:xi-2=c-alpha} \xi^2 \sigma_{c,L} p_M
      q^{n/2} = \widetilde{\epsilon}_M(c \cdot \alpha_n) q^{n/2} = \#
      (c \cdot \alpha_n) p_M q^{n/2}.
   \end{equation}
   By the definition of the quantum product we have:
   $$\#(c \cdot \alpha_n) = \sum_{A} GW^M_{A,3}(c,c,[L]),$$ where the sum 
   goes over $A \in H_2(M)$ with $\langle c_1, A \rangle = n/2$. (Note
   that since $n$=even the order of the classes $(c,c,[L])$ in the
   Gromov-Witten invariant does not make a difference.) Substituting
   this in~\eqref{eq:xi-2=c-alpha} yields the desired identity.

   Note that we have carried the proof above for the quantum homology
   $QH(M; R)$ with coefficients in the ring $R = \mathbb{Z}[q^{-1},
   q]$ but since $(M, \omega)$ is monotone, it is easy to see that
   equation~\eqref{eq:cubic_eq_ccL} involves only positive powers of
   $q$ hence it holds in fact in $QH(M;R^+)$, where $R^+ =
   \mathbb{Z}[q]$.

   To complete the proof of the theorem we still need the following.
   \begin{proof}[Proof of Lemma~\ref{l:p-class}]
      Write $$a*b = a\cdot b + \sum_{j \geq 1} \gamma_j q^{jk},$$
      where $a \cdot b \in H_{|a|+|b|-2n}(M)$ is the classical
      intersection product of $a$ and $b$, $k$ is the minimal Chern
      number, and $\gamma_j \in H_{|a|+|b|-2n+2jk}(M)$. In order to
      prove the lemma we need to show that $\gamma_{j_0} = 0$, where
      $2j_0 k = 2n-|a|-|b|$.

      Suppose by contradiction that $\gamma_{j_0} \neq 0$. Then there
      exists $A \in H_2(M)$ with $$2\langle c_1, A \rangle = 2j_0 k =
      2n-|a|-|b|$$ such that $GW_{A,3}(a,b,[M]) \neq 0$, where $[M]
      \in H_{2n}(M)$ is the fundamental class. Since $[M]$ poses no
      additional incidence conditions on $GW$-invariants, this implies
      that for a generic almost complex structure there exists a
      pseudo-holomorphic rational curve passing through generic
      representatives of the classes $a$ and $b$. More precisely
      denote by $\mathcal{M}_{0,2}(A,J)$ the space of simple rational
      $J$-holomorphic curves with $2$ marked points in the class $A$.
      Denote by $ev: \mathcal{M}_{0,2}(A,J) \longrightarrow M \times
      M$ the evaluation map. Since $GW_{A,3}(a,b,[M]) \neq 0$, then
      for a generic choice of (pseudo) cycles $D_a, D_b$ representing
      $a, b$ and for a generic choice of $J$ the map $ev$ is
      transverse to $D_a \times D_b$ and moreover $ev^{-1}(D_a \times
      D_b) \neq \emptyset$. However this is impossible because
      \begin{align*}
           & \dim \mathcal{M}_{0,2}(A,J) + \dim (D_a \times D_b) = \\
           & \bigl(2n + 2 \langle c_1, A \rangle -2\bigr) + |a|+|b| = 4n - 2
           < \dim (M \times M).
        \end{align*}
   \end{proof}
   
   The proof of Theorem~\ref{t:cubic_eq-ccl} is now complete.

\end{proof}

\subsection{Proof of Theorems~\ref{t:cubic-eq}
  and~\ref{t:rel-to-discr}} \label{sb:prf-tmain-cubic} \cntrsb

The proof follows immediately from Theorem~\ref{t:cubic_eq-ccl}.
Indeed, since $\#([L] \cdot [L]) = \varepsilon \chi \neq 0$ we can
take $c = [L]$, $\xi = \varepsilon \chi$ in
Theorem~\ref{t:cubic_eq-ccl}.  The constants $\sigma_L, \tau_L$ from
Theorem~\ref{t:cubic-eq} are now $\sigma_{[L],L}, \tau_{[L],L}$
respectively, and we have $\Delta_L = \sigma_{[L],L}^2 +
4\tau_{[L],L}$.  \Qed

\subsection{Proof of Theorem~\ref{t:cubic-eq-sphere}} \cntrsb
\label{sb:prf-cubic-eq-S}
We will prove here the following more general result, \cmred{from
  which Theorem~\ref{t:cubic-eq-sphere} follows directly.}  We call an
element $a \in QH_*(M)$ \emph{classical}, if it lies in the image of
the canonical inclusion $H_*(M) \subset QH_*(M)$.

\begin{thm} \label{t:cubic-eq-sphere-gnrl} Let $S \subset M$ be a
   monotone Lagrangian sphere in closed $2n$-dimensional symplectic
   manifold $M$.
   \begin{enumerate}
     \item \label{ig:n-odd} If $n=$~odd then $[S]*[S]=0$. More
      generally, when $n=$~odd, for all $a \in H_n(M)$ with $a \cdot
      [S]=0$ we have $a*[S]=0$.
     \item \label{ig:n-even} Assume $n=$~even. Then:
      \begin{enumerate}[(i)]
        \item \label{ig:div} If $C_M | n$ then there exists a unique
         $\gamma_S \in \mathbb{Z}$ such that $[S]^{*3} = \gamma_S [S]
         q^{n}$. If we assume in addition that $2C_M \centernot| n$,
         then $\gamma_S$ is divisible by $4$. Moreover for every (not
         necessarily classical) element $b \in QH_0(M)$ there exists a
         unique $\eta_b \in \mathbb{Z}$ such that we have $b*[S] =
         \eta_b [S]q^n$.
        \item \label{ig:ndiv} If $C_M \centernot| n$ then for every
         (not necessarily classical) element $b \in QH_0(M)$ we have
         $b*[S]=0$. In particular, by taking $b = [S]*[S]$ we obtain
         $[S]^{*3}=0$.
      \end{enumerate}
   \end{enumerate}
\end{thm}

\begin{proof}
   Fix once and for all a spin structure on $S$. Denote by $e_S \in
   QH_n(S;\Lambda)$ the unity.

   Note that the case $C_M = \infty$ (i.e. $\omega|_{\pi_2(M)} = 0$)
   is trivial. Indeed under such assumptions we have $QH_*(M) \cong
   H_*(M)$ via an isomorphism that intertwines the quantum and the
   classical intersection products. The statement in~\eqref{ig:n-odd}
   follows immediately. The statements
   in~\eqref{ig:div},~\eqref{ig:ndiv} follow from the fact that for $b
   \in QH_0(M)$ the degree of $b*[S]$ is negative. Thus, from now one
   we assume that $C_M < \infty$.

   We will also assume throughout the proof that $n > 1$, for
   otherwise the statement is again obvious (if $n=1$, then either $M
   = S^2$ and $S=$~equator, or $\omega|_{\pi_2(M)} = 0$).  Thus we
   assume from now that $\pi_1(S)= 1$ hence $N_S = 2C_M$.
   
   We now appeal to the spectral sequence described
   in~\S\ref{sb:spec-seq}. From Theorem~\ref{t:spectral-seq} it
   follows that
   \begin{equation} \label{eq:QH_i-0-n-mod-2CM} QH_i(S; \Lambda)=0 \quad
      \forall \, i \centernot\equiv 0, n (\bmod \,2C_M).
   \end{equation}
   Moreover, if $2C_M \centernot| n$ then:
   \begin{enumerate}
     \item either $QH_0(S;\Lambda)=0$, or the augmentation
      $\widetilde{\epsilon}_S: QH_0(S;\Lambda) \longrightarrow
      H_0(S;\Lambda)$ is an isomorphism.
     \item $QH_n(S;\Lambda) = \mathbb{Z} e_S$ (and $e_S$ is not a torsion
      element).
   \end{enumerate}

   \cmred{We prove statement~\eqref{ig:n-odd} of the theorem, i.e.
     when $n=$~odd.}  Let $a \in H_n(M)$ be an element with $a \cdot
   [S]=0$. Consider $$y = a * e_S \in QH_0(S;\Lambda).$$ We claim that
   $y=0$.  Indeed, either $QH_0(S;\Lambda)=0$ in which case $y=0$, or
   $\widetilde{\epsilon}_S: QH_0(S; \Lambda) \longrightarrow H_0(S)$
   is an isomorphism and then $\widetilde{\epsilon}_S(y) = a \cdot
   [S]=0$, hence $y=0$ again.

   On the other hand $i_S(y) = a * i_L(e_S) = a * [S]$, which implies
   $a*[S]=0$. Note that $[S] \cdot [S] =0$. Therefore, if we take $a =
   [S]$ we obtain $[S]*[S]=0$. This completes the proof for the case
   $n=$~odd.
   
   We now turn to statement~\eqref{ig:n-even} of the theorem, hence
   assume that $n=$~even. We first deal with the case~\eqref{ig:ndiv},
   i.e.  assume that $C_M \centernot| n$. Let $b \in QH_0(M)$. Put $u
   = b*e_S \in QH_{-n}(S;\Lambda)$.  By~\eqref{eq:QH_i-0-n-mod-2CM} we
   have $QH_{-n}(S;\Lambda)=0$, hence $u=0$. On the other hand $i_S(u)
   = b*i_S(e_S) = b*[S]$. This proves the case~\eqref{ig:ndiv}.

   To prove~\eqref{ig:div}, assume that $C_M | n$. We will first
   assume that $2C_M \centernot| n$. Let $b \in QH_0(M)$ and put $w =
   b * e_S \in QH_{-n}(S;\Lambda)$. By the discussion above we have
   $$QH_{-n}(S;\Lambda) = QH_n(S;\Lambda) t^{n/C_M} = \mathbb{Z} e_S
   t^{n/C_M}.$$ It follows that $w = \eta_b e_S t^{n/C_M}$ for some
   $\eta_b \in \mathbb{Z}$. Applying $i_S$ to $w$ we get $$\eta_b [S]
   q^{n} = b*i_S(e_S) = b*[S].$$ As before we can take $b = [S]*[S]$
   and obtain $[S]^{*3} = \gamma_S [S]q^n$, where $\gamma_S =
   \eta_{\scriptscriptstyle [S]*[S]} \in \mathbb{Z}$.

   To complete the proof of point~\eqref{ig:div} of the theorem in the
   case $2C_M \centernot| n$, it remains to show that $4 | \gamma_S$.
   To this end put $z = [S]*e_S \in QH_0(S; \Lambda)$.  Note that
   $\widetilde{\epsilon}_S(z) = \#([S]\cdot[S])p = \pm 2p$, where $p \in
   QH_0(S)$ is the class of a point. Since $\widetilde{\epsilon}_S$ is
   an isomorphism it follows that $z$ is divisible by $2$ in
   $QH_0(S;\Lambda)$ (this does not necessarily hold if $2 C_M | n$). 
   In particular $z*z \in QH_{-n}(S;\Lambda)$ is
   divisible by $4$. At the same time by the theory recalled
   in~\S\ref{sb:HF} we also have
   $$z*z = ([S]*e_S)*([S]*e_S) = ([S]*([S]*e_S))*e_S = ([S]*[S])*e_S,$$
   hence $i_S(z*z) = [S]^{*3}$. It follows that $[S]^{*3}$ is
   divisible by $4$. But $[S]^{*3} = \gamma_S [S]q^n$ and $[S]$ is
   neither torsion nor divisible by any integer $\geq 2$.
   Consequently, $\gamma_S$ is divisible by $4$.  This completes the
   proof of point~\eqref{ig:div} of the theorem under the assumption
   that $2C_M \centernot| n$.

   Finally, it remains to treat the other case at point~\eqref{ig:div}
   of the theorem, i.e.  $n=$~even and $2C_M | n$.  It is easy to see
   that $S$ satisfies condition~$\mathscr{L}$ (e.g. by using
   Proposition~\ref{p:criterion}). Therefore this case is completely
   covered by Theorem~\ref{t:cubic-eq} (which has already been proved)
   together with Corollary~\ref{c:sig=0} and the short discussion
   after its statement.
\end{proof}

\subsection{Further results} \cntrsb \label{sb:further-res}

We present here a few other results that follow from the same ideas as
in the proofs of Theorems~\ref{t:cubic_eq-ccl}
and~\ref{t:cubic-eq-sphere-gnrl}.

\begin{thm} \label{t:L_1L_2} Let $L_1, L_2 \subset M$ be two
   Lagrangian submanifolds satisfying conditions~$(1)$~--~$(3)$ of
   Assumption~$\mathscr{L}$ (possibly with different minimal Maslov
   numbers). Assume that $[L_1] \cdot [L_2] = 0$. Then one of the
   following two (non exclusive) possibilities occur:
   \begin{enumerate}
     \item either $[L_1]$ and $[L_2]$ are proportional in $H_n(M;
      \mathbb{Q})$ and moreover we have the relation $[L_1]*[L_1] =
      \kappa [L_1]q^{n/2}$ in $QH(M; R^+_{\mathbb{Q}})$ for some
      $\kappa \in \mathbb{Z}$;
     \item or $[L_1]*[L_2] = 0$.
   \end{enumerate}
\end{thm}
\begin{remnonum}
   Note that if possibly~(1) occurs in the theorem and moreover
   $N_{L_1} = N_{L_2} = 2$, then $\lambda_{L_1} = \lambda_{L_2}$. This
   is so because by the theorem $[L_1]$ and $[L_2]$ are proportional
   and $[L_i]$ is an eigenvector of the operator $P$ with eigenvalue
   $\lambda_{L_i}$ (see~\S\ref{sb:eigneval}).
\end{remnonum}

Here is a simple example of Lagrangians $L_1, L_2$ satisfying the
conditions of Theorem~\ref{t:L_1L_2}. We take $M$ to be the monotone
blow-up of ${\mathbb{C}}P^2$ at $3$ points and $L_1, L_2$ Lagrangian
spheres in the classes $[L_1] = H - E_1-E_2-E_3$, $[L_2] = E_2-E_3$
(using the notation of~\S\ref{sbsb:intro-exp-blcp2}).
See~\S\ref{sb:lag-spheres-blow-ups} for more details on how to
actually construct these spheres. Clearly $[L_1]\cdot [L_2]=0$, hence
the theorem implies that $[L_1]*[L_2]=0$ (which can of course be
confirmed also by direct calculation). One can construct many other
examples of this type in monotone blow-ups of ${\mathbb{C}}P^2$ at $3
\leq k \leq 8$ points.

On the other hand, if $L \subset M$ is a Lagrangian satisfying
conditions~$(1)$~--~$(3)$ of Assumption~$\mathscr{L}$ and we assume in
addition that $\chi(L) = 0$ then we can take $L = L_1 = L_2$.
Theorem~\ref{t:L_1L_2} then implies that $[L]*[L] = [L]\kappa q^{n/2}$
for some $\kappa \in \mathbb{Z}$. The simplest example should be when
$L$ is a $2$-torus, however we are not aware of any example of a
monotone Lagrangian $2$-torus satisfying conditions~$(1)$~--~$(3)$ of
Assumption~$\mathscr{L}$ and with $[L] \neq 0$. An easy (algebraic)
argument shows that such tori cannot exist in a symplectic
$4$-manifold with $b_2^+ = 1$ (e.g. in blow-ups of ${\mathbb{C}}P^2$).
It would be interesting to know if this holds in greater generality.

Finally, we remark that if one replaces the condition $[L_1] \cdot
[L_2] = 0$ by the stronger assumption that $L_1 \cap L_2 = \emptyset$,
and drops conditions~$(3)$,~$(4)$ of Assumption~$\mathscr{L}$, then it
still follows that $[L_1]*[L_2] = 0$. This is proved
in~\cite{Bi-Co:rigidity}-Theorem~2.4.1 (see also~\S8
in~\cite{Bi-Co:Yasha-fest}).

\begin{proof}[Proof of Theorem~\ref{t:L_1L_2}]
   Without loss of generality we may assume that both $[L_1]$ and
   $[L_2]$ are non-trivial in $H_n(M;\mathbb{Q})$, for otherwise
   possibility~(2) obviously holds.

   Define $y_1 = [L_2]*e_{L_1} \in QH_0(L_1; \Lambda^1_{\mathbb{Q}})$
   and $y_2 = [L_1]*e_{L_2} \in QH_0(L_2; \Lambda^2_{\mathbb{Q}})$.
   Here we have denoted $\Lambda^1_{\mathbb{Q}} = \mathbb{Q}[t_1^{-1},
   t_1]$ with $|t_1| = -N_{L_1}$ and $\Lambda^2_{\mathbb{Q}} =
   \mathbb{Q}[t_2^{-1}, t_2]$ with $|t_2|= -N_{L_2}$ since we have to
   distinguish between the coefficient rings of $L_1$ and $L_2$. Note
   that under the embeddings of $\Lambda^1_{\mathbb{Q}}$ and
   $\Lambda^2_{\mathbb{Q}}$ into $R_{\mathbb{Q}} =
   \mathbb{Q}[q^{-1},q]$ we have $t_1^{\nu_1} = q^{n/2} =
   t_2^{\nu_2}$. (See~\S\ref{sb:HF}.)

   Since $[L_1] \cdot [L_2] = 0$ and due to condition~$(3)$ of
   Assumption~$\mathscr{L}$, we have
   $$y_1 = \kappa_1 e_{L_1}t_1^{\nu_1}, 
   \quad y_2=\kappa_2 e_{L_2} t_2^{\nu_2},$$ for some $\kappa_1,
   \kappa_2 \in \mathbb{Z}$ and where $\nu_1 = n/N_{L_1}$, $\nu_2 = n/
   N_{L_2}$. At the same time we also have
   $$i_{L_1}(y_1) = i_{L_2}(y_2) = [L_1]*[L_2].$$ Here we have used the
   fact that $n$ must be even, hence $[L_1]*[L_2] = [L_2]*[L_1]$.

   It follows that $\kappa_1 [L_1] q^{n/2} = [L_1]*[L_2] = \kappa_2
   [L_2] q^{n/2}$ and the result follows. (As in the proof of
   Theorem~\ref{t:cubic_eq-ccl}, note that here too, the identities
   proved involve only positive powers of $q$ hence they hold in
   $QH(M;R^+)$ too.)
\end{proof}

The next result is concerned with Lagrangian spheres that do not
satisfy Assumption~$\mathscr{L}$, but rather~(2i-b) on
page~\pageref{i:CM-n} (after Theorem~\ref{t:cubic-eq-sphere}).

\begin{thm} \label{t:lag-spheres-no_L} Let $L_1, L_2 \subset M$ be
   oriented Lagrangian spheres in a closed monotone symplectic
   manifold $M$ of dimension $2n$. Assume that $n=$~even and $C_M | n$
   but $2C_M \centernot| n$.
   \begin{enumerate}
     \item \label{if:L1L2} If $[L_1] \cdot [L_2] = 0$ then
      $[L_1]*[L_2]=0$.
     \item \label{if:k-L1L2} If $k:= \#([L_1] \cdot [L_2]) \neq 0$
      then $$[L_1]^{*2} = [L_2]^{*2} = \tfrac{2 \varepsilon}{k}
      [L_1]*[L_2],$$ where $\varepsilon = (-1)^{n(n-1)/2}$.
      Furthermore, either $[L_1]^{*3} = [L_2]^{*3} = 0$ or $[L_1] =
      \pm [L_2]$ (the two possibilities not being exclusive).
   \end{enumerate}
\end{thm}

\begin{rem} \label{r:lag-spheres-no_L} Recall from
   Theorem~\ref{t:cubic-eq-sphere} that each of the Lagrangians $L_i$,
   $i=1,2$, satisfies a cubic equation of the type: $[L_i]^{*3} =
   \gamma_i [L_i]q^n$. In general, it seems that the coefficients
   $\gamma_1$ and $\gamma_2$ might differ one from the other, however
   in case~\eqref{if:k-L1L2} of the theorem it is easy to see that
   $\gamma_1 = \gamma_2$.
\end{rem}

\begin{proof}[Proof of Theorem~\ref{t:lag-spheres-no_L}]
   By standard arguments there exist canonical isomorphisms $QH_*(L_i)
   \to H_*(L_i;\Lambda)$, $i=1,2$. Thus 
   $$QH_0(L_i) = \mathbb{Z} p_i, \quad QH_n(L_i) = \mathbb{Z} e_{L_i},$$
   where $p_i$ is the class of a point in $L_i$ and $e_{L_i}$ is the
   fundamental class of $L_i$.

   Assume first that $[L_1] \cdot [L_2] = 0$. In view of the
   isomorphism just mentioned we have $[L_1] * e_{L_2} = 0$. Applying
   $i_{L_2}$ to the last equality we obtain $[L_1]*[L_2] = 0$.

   Assume now that $k: = \#([L_1] \cdot [L_2]) \neq 0$. Due to our
   assumptions we have:
   \begin{enumerate}[(i)]
     \item \label{i:L2eL1} $[L_2]*e_{L_1} = k p_1$.
     \item \label{i:L1eL2} $[L_1]*e_{L_2} = k p_2$.
     \item \label{i:L1eL1} $[L_1]*e_{L_1} = 2\varepsilon p_1$.
     \item \label{i:L2eL2} $[L_2]*e_{L_2} = 2 \varepsilon p_2$.
   \end{enumerate}
   From~\eqref{i:L2eL1} and~\eqref{i:L1eL2} it follows that
   $$i_{L_1}(p_1) = i_{L_2}(p_2) = \tfrac{1}{k} [L_1]*[L_2].$$
   From~\eqref{i:L1eL1} and~\eqref{i:L2eL2} we obtain:
   $$i_{L_1}(p_1) = \tfrac{\varepsilon}{2} [L_1]*[L_1], \quad 
   i_{L_2}(p_2) = \tfrac{\varepsilon}{2} [L_2]*[L_2].$$ This implies
   the first result of point~\eqref{if:k-L1L2} of the
   theorem.

   To prove the other statements, we use point~\eqref{i:div} of
   Theorem~\ref{t:cubic-eq-sphere}. By that theorem there exist
   $\gamma_1, \gamma_2 \in \mathbb{Z}$ such that
   $$[L_1]^{*3} = \gamma_1 [L_1]q^n, \quad [L_2]^{*3} = \gamma_2 [L_2] q^n.$$
   It follows that $$\gamma_1 [L_1]q^n = [L_1]^{*3} = [L_2]^{*2}*[L_1]
   = \tfrac{k \varepsilon}{2} [L_2]^{*3} = \tfrac{k \varepsilon}{2}
   \gamma_2 [L_2]q^n,$$ hence $\gamma_1[L_1] = \tfrac{k
     \varepsilon}{2}\gamma_2[L_2]$. It follows that $\gamma_1=0$ if
   and only if $\gamma_2 = 0$. Now, if $\gamma_1 \neq 0$ then
   $$\gamma_1 [L_1]\cdot [L_2] = \tfrac{k \varepsilon}{2} \gamma_2 [L_2]\cdot [L_2] = 
   \tfrac{k \varepsilon}{2} \gamma_2 2 \varepsilon p,$$ where $p \in
   H_0(M)$ is the class of a point. At the same time we have
   $[L_1]\cdot [L_2] = k p$ and so $k \gamma_1 = k\gamma_2$. It
   follows that $\gamma_1 = \gamma_2$ and $[L_1] = \tfrac{k
     \varepsilon}{2}[L_2]$. Squaring the last equality with respect to
   the (classical) intersection product we obtain: $2 \varepsilon =
   \tfrac{k^2}{4} 2 \varepsilon$, hence $k = \pm 2$. This shows that
   $[L_1] = \pm [L_2]$.
\end{proof}


%% file: lag-cob.tex

\section{The discriminant and Lagrangian cobordisms} \cntrs
\label{s:disc-lcob} 

This section provides the proofs of Theorem~\ref{t:cob} and a
generalization of Corollary~\ref{c:del_1=del_2}. 

\pbred{In what follows Lagrangian cobordisms $V$ will be generally
  assumed to be connected. In contrast, their boundaries $\partial V$
  are allowed to have several connected components.}

We begin with:
\begin{proof}[Proof of Theorem~\ref{t:cob}]
   Before going into the details of the proof, here is the rationale
   behind it. To the Lagrangian cobordism $V$ we can associate a
   (relative) quantum homology $QH(V, \partial V)$ which has a quantum
   product. The quantum product on $QH(V, \partial V)$ is related to
   the quantum products for the ends of $V$ via a quantum connectant
   $\delta: QH(V, \partial V) \longrightarrow QH(\partial V) =
   \oplus_{i=1}^r QH(L_i)$. This makes it possible to find relations
   between the products on the quantum homologies $QH(L_i)$ of
   different ends of $V$ and the quantum product on $QH(V, \partial
   V)$. In particular this gives the desired relation between the
   discriminants of the different ends.

   We now turn to the details of the proof. We will use here several
   versions of the pearl complex and its homology (also called
   Lagrangian quantum homology) both for Lagrangian cobordisms as well
   as for their ends.  We refer the reader to~\cite{Bi-Co:Yasha-fest,
     Bi-Co:rigidity, Bi-Co:lagtop} for the foundations of the theory
   in the case of closed Lagrangians and to~\S5 of~\cite{Bi-Co:cob1}
   in the case of cobordisms.

   Throughout this proof we will work with $\mathbb{Q}$ as the base
   field and with $\Lambda = \mathbb{Q}[t^{-1}, t]$ or $\Lambda^{+} =
   \mathbb{Q}[t]$ as coefficient rings. We denote by $\mathcal{C}$ and
   $\mathcal{C}^+$ the pearl complexes with coefficients in $\Lambda$
   and $\Lambda^+$ respectively, and by $QH$ and $Q^+H$ their
   homologies. The latter is sometimes called the positive Lagrangian
   quantum homology.

   Before we go on, a small remark regarding the coefficients is in
   order. Throughout this proof we grade the variable $t \in \Lambda$
   as $|t|=-N_V$. This is the standard grading for $QH(V)$ and $QH(V,
   \partial V)$ and their positive versions. We use the same
   coefficient rings (and grading) also for $QH(L_i)$ and its positive
   version. This is possible since $N_V | N_{L_i}$, hence our ring
   $\Lambda^+$ is an extension of the corresponding ring in which the
   degree of $t$ is $-N_{L_i}$.

   Recall that (for any Lagrangian submanifold) the positive quantum
   homology $Q^+H$ admits a natural map $Q^+H \longrightarrow QH$
   induced by the inclusion $\mathcal{C}^+ \longrightarrow
   \mathcal{C}$. Again, for degree reasons the induced map in homology
   is an isomorphism in degree $0$ and surjective in degree $1$:
   \begin{equation} \label{eq:Q^+H-QH} Q^+H_0 \xrightarrow{\; \; \cong
        \; \;} QH_0, \quad Q^+H_1 \rrightarrow QH_1.
   \end{equation}
   In fact, the last map is an isomorphism whenever the minimal Maslov
   number is $> 2$. We also have $Q^+H_n(K) \cong H_n(K)$ for every
   $n$-dimensional Lagrangian submanifold $K$.

   Coming back to the proof of the theorem, we first claim there is a
   commutative diagram
   \begin{equation} \label{eq:diag-QH-H}
      \begin{CD}
         Q^+H_1(V) @>j_Q>>Q^+H_1(V, \partial V) @ > \delta >>
         Q^+H_0(\partial V) @>i_Q>> Q^+H_0(V) \\
         @V s VV @V s  VV @V s VV @V s VV \\
         H_1(V) @>j>> H_1(V, \partial V) @ > \partial >> H_0(\partial
         V) @>i>> H_0(V)\\
         @. @VVV @VVV @. \\
         @. 0 @. 0 @.
      \end{CD}
   \end{equation}
   with exact rows and columns. The second row of the diagram is the
   classical homology sequence for the pair $(V, \partial V)$ with
   $\partial$ being the connecting homomorphism (we use $\mathbb{Q}$
   coefficients here). The first row is its quantum homology analogue,
   and we remark that the quantum connectant $\delta$ is
   multiplicative with respect to the quantum product (see~\S5
   of~\cite{Bi-Co:cob1} and~\cite{BSing:msc}). The vertical maps $s$
   come from the following general exact sequence of chain complexes:
   \begin{equation} \label{eq:ses-s} 0 \longrightarrow t \mathcal{C}^+
      \xrightarrow{\;\, \iota \;\,} \mathcal{C}^+ \xrightarrow{\;\, s
          \; \, } CM \longrightarrow 0,
   \end{equation}
   where $CM$ stand for the Morse complex (defined using the same
   Morse function and metric as used for the pearl complex, but with
   coefficient in $\mathbb{Q}$ rather than $\Lambda^+$). The second
   map in this exact sequence, $s: \mathcal{C}^+ \longrightarrow CM$,
   is induced by $t \mapsto 0$ (i.e. it sends a pearly chain to its
   classical part, omitting the $t$'s), and $\iota$ stand for the
   inclusion. We now explain why the two middle $s$ maps
   in~\eqref{eq:diag-QH-H} are surjective. We start with the third $s$
   map (i.e. the one before the rightmost $s$). We have:
   \begin{equation} \label{eq:H_0(del-V)}
      H_0(\partial V) = \bigoplus_{i=1}^r H_0(L_i), \quad
      Q^+H_0(\partial V) = \bigoplus_{i=1}^r Q^+H_0(L_i).
   \end{equation}
   Next, note that the composition of $s: Q^+H_0(L_i) \longrightarrow
   H_0(L_i)$ with the inclusion $H_0(L_i) \subset H_0(L_i; \Lambda^+)$
   coincides with the augmentation $\widetilde{\epsilon}_{L_i} :
   Q^+H_0(L_i) \longrightarrow H_0(L; \Lambda^+)$. The fact that $s$
   is surjective now follows easily from~\S\ref{sbsb:Q=QH_n}
   and~\eqref{eq:Q^+H-QH}.

   The surjectivity of the second to the left $s$ map requires a
   different argument.  Consider the chain complex $\mathcal{D}_* = (t
   \mathcal{C}^+)_*$, viewed as a subcomplex of $\mathcal{C}^+$. In
   view of the exact sequence~\eqref{eq:ses-s} the surjectivity of the
   second to the left $s$ map in~\eqref{eq:diag-QH-H} would follow if
   we show that $H_0(\mathcal{D})=0$. To this end consider the
   following filtration $\mathcal{F}_{\bullet}\mathcal{D}$ of
   $\mathcal{D}$ by subcomplexes, defined by:
   \begin{align*}
      & \mathcal{F}_m \mathcal{D} := t^{-m} \mathcal{D} = t^{-m+1}
      \mathcal{C}^+ \quad \forall \, m \leq 0, \\
      & \mathcal{F}_k \mathcal{D} := \mathcal{D} \quad \forall \, k
      \geq 0.
   \end{align*}
   Note that this filtration is very similar to the one described
   in~\S\ref{sb:spec-seq} only that here it is applied to the complex
   $\mathcal{D}$ rather than to $\mathcal{C}$.

   A simple calculation (similar to the one in~\S\ref{sb:spec-seq})
   shows that the first page of the spectral sequence associated to
   this filtration satisfies:
   \begin{align*}
      & E^1_{p,q} \cong t^{-p+1} H_{{\scriptscriptstyle
          p+q+N_V-pN_V}}(V, \partial V) \quad
      \forall \, p \leq 0, \\
      & E^1_{p,q} = 0 \quad \forall \, p \geq 1.
   \end{align*}
   It follows from the assumption of the theorem that for all $p,q$
   with $p+q=0$ we have $E^1_{p,q}=0$, hence also
   $E^{\infty}_{p,q}=0$. Since this spectral sequence converges to
   $H_*(\mathcal{D})$ this implies that $H_0(\mathcal{D})=0$. This
   completes the proof of the surjectivity of the second to the left
   $s$ map in~\eqref{eq:diag-QH-H}.

   We proceed now with the proof of the theorem, based on the
   diagram~\eqref{eq:diag-QH-H} and its properties. We first remark
   that due to the assumptions of the theorem the number of ends of
   $V$ must be $r \geq 2$. Indeed, by the results of~\cite{Bi-Co:cob1}
   if a Lagrangian submanifold $L_1$ is Lagrangian null-cobordant
   (i.e. there exists a monotone Lagrangian cobordism $V$ with only
   one end being $L_1$) then $HF(L_1,L_1)=0$, in contrast with the
   assumption that $L_1$ satisfies condition $(3)$ of
   Assumption~$\mathscr{L}$. We therefore assume from now on that $r
   \geq 2$.

   Denote by $p_i \in H_0(L_i) \subset H_0(\partial V)$ the class
   corresponding to a point in $L_i$. Let $\alpha_2, \ldots, \alpha_r
   \in H_1(V, \partial V)$ be classes with $\partial \alpha_i = p_1 -
   p_i$. Choose lifts $\overline{p}_i \in Q^+H_0(\partial V)$ of the
   $p_i$'s under the map $s$ as well as lifts $\overline{\alpha}_2,
   \ldots, \overline{\alpha}_r \in Q^+H_1(V,\partial V)$ of $\alpha_2,
   \ldots, \alpha_r$. Denote by $e_V \in Q^+H_{n+1}(V, \partial V)$
   the unity and by $e_{L_i} \in Q^+H_n(L_i)$ the unities
   corresponding to the $L_i$'s. Note that $\delta(e_V) = e_{L_1} +
   \cdots + e_{L_r}$. Finally, put $\nu = n / N_V$. (Recall that
   $N_{L_i} | n$ by assumption, and since $N_V | N_{L_i}$ we have $N_V
   | n$.) Since the Lagrangians $L_i$ satisfy conditions~(1)~--~(3) of
   Assumption~$\mathscr{L}$ and in view of~\S\ref{sbsb:Q=QH_n}, we
   have:
   $$Q^+H_0(\partial V)
   \cong QH_0(\partial V) = \mathbb{Q} \overline{p}_1 \oplus \cdots
   \oplus \mathbb{Q} \overline{p}_r \oplus \mathbb{Q} e_{L_1} t^{\nu}
   \oplus \cdots \oplus \mathbb{Q} e_{L_r} t^{\nu}.$$

   \begin{prop} \label{p:basis-delta} $\dim_{\mathbb{Q}}
      (\textnormal{image\,} \delta) = r$. Moreover, for every choice
      of $\alpha_i$'s and $\overline{\alpha}_i$'s the elements
      $$\delta(\overline{\alpha}_2), \ldots, \delta(\overline{\alpha}_r), 
      (e_{L_1} + \cdots + e_{L_r}) t^{\nu}$$ form a basis (over
      $\mathbb{Q}$) of the vector space $\textnormal{image\,}{\delta}
      \subset QH_0(\partial V)$.
   \end{prop}
   We defer the proof of the lemma and continue with the proof of our
   theorem.
   
   Denote by $\mathcal{B} \subset Q^+H_1(V, \partial V)$ the kernel of
   $\delta : Q^+H_1(V, \partial V) \longrightarrow Q^+H_0(\partial
   V)$. By Proposition~\ref{p:basis-delta} the elements
   $$\overline{\alpha}_2, \ldots, \overline{\alpha}_r, e_V t^{\nu}$$
   induce a basis for the vector space $Q^+H_1(V, \partial V) /
   \mathcal{B}$.

   We now continue by proving that $\Delta_{L_1} = \Delta_{L_2}$. The
   other equalities follow by the same recipe. Using the preceding
   basis we can write:
   \begin{equation} \label{eq:a2*a2-delta-1}
      \begin{aligned} & \overline{\alpha}_2 *
         \overline{\alpha}_2 = \sum_{j=2}^r \xi_j \overline{\alpha}_j
         t^{\nu} + Bt^{\nu} + \rho e_V t^{2 \nu}, \\
         & \delta(\overline{\alpha}_2) = \overline{p}_1 - \overline{p}_2 +
         \sum_{k=1}^r a_k e_{L_k} t^{\nu},
      \end{aligned}
   \end{equation}
   for some $\xi_j, a_k, \rho \in \mathbb{Q}$ and $B \in \mathcal{B}$.
   For the first equality we have used the fact that
   $\overline{\alpha}_2 * \overline{\alpha}_2 \in Q^+H_{1-n}(V,
   \partial V) \cong t^{\nu} Q^+H_1(V, \partial V)$.

   We will also need a similar equality to the second one
   in~\eqref{eq:a2*a2-delta-1}, but for $\delta(\overline{\alpha}_i)$:
   \begin{equation} \label{eq:delta-alpha-i}
      \delta(\overline{\alpha}_i) = \overline{p}_1 - \overline{p}_i +
      \sum_{k=1}^r a_k^{(i)} e_{L_k} t^{\nu}, \quad \forall \, 2 \leq
      i \leq r,
   \end{equation}
   where $a_k^{(i)} \in \mathbb{Q}$. (Note that according to our
   notation $a_k = a_k^{(2)}$.)

   At this point we need to separate the arguments to the cases $r
   \geq 3$ and $r=2$. (As we have already remarked, $r=1$ is
   impossible under the assumptions of the theorem.) We assume first
   that $r \geq 3$. The case $r=2$ will be treated after that.
  
   We now perform a little change in the basis and the choice of the
   lift $\overline{p}_i$ as follows:
   \begin{align*}
      & \overline{\alpha}_2 \longrightarrow \overline{\alpha}_2 - a_3
      e_V t^{\nu}, \quad \overline{\alpha}_i \longrightarrow
      \overline{\alpha}_i \quad
      \forall i \geq 3, \\
      & \overline{p}_1 \longrightarrow \overline{p}_1 +
      (a_1-a_3)e_{L_1} t^{\nu}, \quad \overline{p}_2 \longrightarrow
      \overline{p}_2 - (a_2-a_3) e_{L_2} t^{\nu}, \quad \overline{p}_i
      \longrightarrow \overline{p}_i \quad \forall i \geq 3.
   \end{align*}
   To simplify notation we continue to denote the new basis elements
   by $\overline{\alpha}_i$ and similarly for the $\overline{p}_i$'s.
   By abuse of notation we also continue to denote the new
   coefficients $a_k$, $a_k^{(i)}$, $\xi_j$ and $\rho$ resulting from
   the basis change by the same symbols, and similarly for the term $B
   \in \mathcal{B}$. The outcome of the basis change is that now the
   second equality in~\eqref{eq:a2*a2-delta-1} becomes:
   \begin{equation} \label{eq:delta-alpha-2-new}
      \delta(\overline{\alpha}_2) = \overline{p}_1 - \overline{p}_2 +
      \sum_{k=4}^r a_k e_{L_k} t^{\nu}.
   \end{equation}
   (Of course, if $r=3$ then the third term in the last equation is
   void.)  We now use the fact that $\delta$ is multiplicative
   (see~\cite{Bi-Co:cob1}):
   \begin{equation} \label{eq:delta-alpha-2*alpha-2}
      \delta(\overline{\alpha}_2 * \overline{\alpha}_2) =
      \delta(\overline{\alpha}_2) * \delta(\overline{\alpha}_2) =
      \overline{p}_1^{*2} + \overline{p}_2^{*2} + \sum_{k=4}^r a_k^2
      e_{L_k} t^{2 \nu}.
   \end{equation}
   We now express $\overline{p}_1^{*2} \in Q^+H_{-n}(L_1) \cong
   t^{\nu}Q^+H_{0}(L_1)$ in terms of the basis $\{ \overline{p}_1
   t^\nu, e_{L_1}t^{2\nu}\}$ and similarly for $\overline{p}_2^{*2}$:
   $$\overline{p}_1^{*2} = \sigma_1 \overline{p}_1 t^{\nu} +
   \tau_1 e_{L_1}t^{2\nu}, \qquad \overline{p}_2^{*2} = \sigma_2
   \overline{p}_2 t^{\nu} + \tau_2 e_{L_2}t^{2\nu},$$ where $\sigma_1,
   \sigma_2 \in \mathbb{Q}$ and $\tau_1, \tau_2 \in \mathbb{Q}$. (In
   fact, by choosing the $\alpha_i$'s, $\overline{\alpha}_i$'s and
   $\overline{p}_i$'s carefully, over $\mathbb{Z}$, the coefficients
   $\sigma_1, \sigma_2, \tau_1, \tau_2$ will in fact be in
   $\mathbb{Z}$, but we will not need that.)  Substituting this
   into~\eqref{eq:delta-alpha-2*alpha-2} we obtain:
   \begin{equation} \label{eq:delta-alpha-2*alpha-2-II}
      \delta(\overline{\alpha}_2 * \overline{\alpha}_2) = \sigma_1
      \overline{p}_1 t^{\nu} + \sigma_2 \overline{p}_2 t^{\nu} +
      \tau_1 e_{L_1}t^{2\nu} + \tau_2 e_{L_2}t^{2\nu} + \sum_{k=4}^r
      a_k^2 e_{L_k} t^{2 \nu}.
   \end{equation}
   Applying $\delta$ to the first equality in~\eqref{eq:a2*a2-delta-1}
   and using~\eqref{eq:delta-alpha-2-new}
   and~\eqref{eq:delta-alpha-2*alpha-2-II} we obtain:
   \begin{align*}
      & \xi_2\Bigl(\overline{p}_1 - \overline{p}_2 + \sum_{k=4}^r a_k
      e_{L_k} t^{\nu}\Bigr) t^{\nu} + \sum_{i=3}^r \xi_i
      \Bigl(\overline{p}_1 - \overline{p}_i + \sum_{q=1}^r a_q^{(i)}
      e_{L_q} t^{\nu}\Bigr) t^{\nu} + \rho (e_{L_1} + \cdots +
      e_{L_r})t^{2 \nu} \\
      & = \sigma_1 \overline{p}_1 t^{\nu} + \sigma_2 \overline{p}_2
      t^{\nu} + \tau_1 e_{L_1}t^{2\nu} + \tau_2 e_{L_2}t^{2\nu} +
      \sum_{k=4}^r a_k^2 e_{L_k} t^{2 \nu}.
   \end{align*}
   Comparing the coefficients of $\overline{p}_3, \ldots,
   \overline{p}_r$ we deduce that $\xi_3 = \cdots = \xi_r = 0$. The
   last equation thus becomes:
   \begin{equation} \label{eq:coeffs-xi=0}
      \begin{aligned}
         & \xi_2\Bigl(\overline{p}_1 - \overline{p}_2 + \sum_{k=4}^r
         a_k e_{L_k} t^{\nu}\Bigr) t^{\nu} + \rho (e_{L_1} + \cdots +
         e_{L_r})t^{2 \nu} \\
         & = \sigma_1 \overline{p}_1 t^{\nu} + \sigma_2 \overline{p}_2
         t^{\nu} + \tau_1 e_{L_1}t^{2\nu} + \tau_2 e_{L_2}t^{2\nu} +
         \sum_{k=4}^r a_k^2 e_{L_k} t^{2 \nu}.
      \end{aligned}
   \end{equation}
   Comparing the coefficients of $\overline{e}_3$ on both sides
   of~\eqref{eq:coeffs-xi=0} (recall that $r \geq 3$) we deduce that
   $\rho=0$. It easily follows now that $\tau_1 = \tau_2 = 0$ and that
   $\sigma_1 = \xi_2 = -\sigma_2$. By the definition of the
   discriminant it follows that
   $$\Delta_{L_1} = \sigma_1^2 = \sigma_2^2 = \Delta_{L_2}.$$
   Note that the relation between our $\sigma_i$'s and $\tau_i$'s and
   the notation used in~\S\ref{sb:discr-intro} and
   in~\S\ref{sbsb:Q=QH_n} is $\sigma_1 = \sigma_1(p_1,
   \overline{p}_1)$, $\sigma_2 = \sigma_2(p_2, \overline{p}_2)$ and
   similarly for $\tau_1, \tau_2$.  Finally we remark that since
   $\Delta_{L_1} = \sigma_1^2 \in \mathbb{Z}$ we must have $\sigma_1
   \in \mathbb{Z}$, hence $\Delta_{L_1}$ is a perfect square.

   We now turn to the case $r=2$. In that case we can
   write~\eqref{eq:a2*a2-delta-1} as
   \begin{equation} \label{eq:a2*a2-delta-2}
      \begin{aligned} & \overline{\alpha}_2 *
         \overline{\alpha}_2 = \xi \overline{\alpha}_2
         t^{\nu} + Bt^{\nu} + \rho e_V t^{2 \nu}, \\
         & \delta(\overline{\alpha}_2) = \overline{p}_1 - \overline{p}_2 +
         a_1 e_{L_1} t^{\nu} + a_2 e_{L_2} t^{\nu},
      \end{aligned}
   \end{equation}
   By an obvious basis change (among $\overline{p}_1$,
   $\overline{p}_2$) we may assume that $a_1 = a_2 = 0$. Then the
   identity $\delta(\overline{\alpha}_2 * \overline{\alpha}_2) =
   \delta(\overline{\alpha}_2) * \delta(\overline{\alpha}_2)$ becomes:
   $$\xi(\overline{p}_1 - \overline{p}_2)t^{\nu} + 
   \rho(e_{L_1} + e_{L_2})t^{2 \nu} = \sigma_1 \overline{p}_1 t^{\nu}
   + \sigma_2 \overline{p}_2 t^{\nu} + \tau_1 e_{L_1} t^{2\nu} +
   \tau_2 e_{L_2} t^{2\nu}.$$ It follows immediately that
   $\sigma_1=-\sigma_2$ and $\tau_1 = \tau_2$. Consequently
   $\Delta_{L_1} = \Delta_{L_2}$.

   To complete the proof of the theorem it remains to prove
   Proposition~\ref{p:basis-delta}. For this purpose we will need the
   following Lemma.
   \begin{lem} \label{l:div-t} Let $j \geq 0$ and consider the
      connecting homomorphism $$\delta: Q^+H_{1+jN_V}(V, \partial V)
      \longrightarrow Q^+H_{jN_V}(\partial V).$$ Let $\eta \in
      Q^+H_{1+jN_V}(V,
      \partial V)$ and assume that $\delta(\eta)$ is divisible by $t$.
      Then $\eta$ is also divisible by $t$.
   \end{lem}
   \begin{proof}[Proof of the lemma]
      The connecting homomorphism $\delta$ is part of the following diagram:
      \begin{equation} \label{eq:diag-QH-H-2}
         \begin{CD}
            @. Q^+H_{1+jN_V}(V, \partial V) @ > \delta >>
            Q^+H_{jN_V}(\partial V) \\
            @. @V s  VV @V s VV \\
            H_{1+jN_V}(V) @>j>> H_{1+jN_V}(V, \partial V) @ > \partial >>
            H_{jN_V}(\partial
            V) \\
         \end{CD}
      \end{equation}
      where the vertical $s$-maps are induced by~\eqref{eq:ses-s}.
      Since $\delta(\eta)$ is divisible by $t$ we have $s
      (\delta(\eta))=0$ hence $\partial (s(\eta))=0$. By assumption
      $H_{1+jN_V}(V)=0$ hence the bottom map $\partial$ is injective,
      and therefore we have $s(\eta)=0$. Looking again
      at~\eqref{eq:ses-s} it follows that $$\eta \in
      \textnormal{image\,} \Bigl( H_{1+jN_V}(t\mathcal{C}^+)
      \xrightarrow{\;\; \iota_* \;\;} Q^+H_{1+jN_V}(V,
      \partial V) \Bigr),$$ where $\mathcal{C}^+$ stands for the
      positive pearl complex of $(V,
      \partial V)$. But $$H_{1+jN_V}(t \mathcal{C}^+) \cong
      tQ^+H_{1+(j+1)N_V}(V,\partial V)$$ via an isomorphism for which
      $\iota_*$ becomes the inclusion $$tQ^+H_{1+(j+1)N_V}(V,\partial
      V) \subset Q^+H_{1+jN_V}(V, \partial V).$$ This proves that
      $\eta$ is divisible by $t$.
   \end{proof}
   
   We are finally in position to prove the preceding proposition.
   \begin{proof}[Proof of Proposition~\ref{p:basis-delta}]
      Note that $$\{\overline{p}_1, \delta(\overline{\alpha}_2),
      \ldots, \delta(\overline{\alpha}_r), \delta(e_V)t^{\nu}, e_{L_2}
      t^{\nu}, \ldots, e_{L_r} t^{\nu}\}$$ is a basis for
      $Q^+H_0(\partial V)$ (recall that $\delta(e_V) = e_{L_1} +
      \cdots + e_{L_r}$).  Therefore it is enough to show that the
      subspace of $Q^+H_0(\partial V)$ generated by $\overline{p}_1,
      e_{L_2} t^{\nu}, \ldots, e_{L_r} t^{\nu}$ has trivial
      intersection with $\textnormal{image\,}(\delta)$.

      Let $\gamma = c \overline{p}_1 + \sum_{j=2}^r b_j e_{L_j}
      t^{\nu}$, where $c, b_j \in \mathbb{Q}$ and assume that $\gamma
      = \delta(\beta)$ for some $\beta \in Q^+H_1(V, \partial V)$.  We
      have $s(\gamma) = c p_1$, where the map $s$ is the third
      vertical map from diagram~\eqref{eq:diag-QH-H}. It follows from
      that diagram that $\partial (s(\beta))= c p_1$. But this is
      possible only if $c=0$ since $p_1 \not\in
      \textnormal{image\,}(\partial)$.

      Thus $\gamma = \sum_{j=2}^r b_j e_{L_j} t^{\nu}$ and we have to
      show that $\gamma=0$.  Recall that $\gamma = \delta(\beta)$. We
      claim that $\beta$ is divisible by $t^{\nu}$, i.e. there exists
      $\beta' \in Q^+H_{n+1}(V,\partial V)$ such that $\beta = t^{\nu}
      \beta'$. To prove this we first note that $\gamma$ is divisible
      by $t$. By Lemma~\ref{l:div-t}, $\beta$ is also divisible by
      $t$. Thus there exists $\beta_1 \in Q^+H_{1+N_V}(V, \partial V)$
      with $\beta = t \beta_1$. In particular $\delta(\beta_1) =
      \sum_{j=2}^r b_j e_{L_j} t^{\nu-1}$.  Continuing by induction, using
      Lemma~\ref{l:div-t} repeatedly, we obtain elements $\beta_j \in
      Q^+H_{1+jN_V}(V, \partial V)$ with $t \beta_{j+1} = \beta_j$ for
      every $1 \leq j \leq \nu-1$. Take $\beta' = \beta_{\nu}$.

      It follows that $t^{\nu} \delta(\beta') = \sum_{j=2}^r b_j
      e_{L_j} t^{\nu}$ for some $\beta' \in Q^+H_{n+1}(V, \partial
      V)$. As $Q^+H_{n+1}(V, \partial V) = \mathbb{Q} e_V$ we have
      $\beta' = a e_V$ for some $a \in \mathbb{Q}$. But $\delta(e_V) =
      e_{L_1} + \cdots +e_{L_r}$ hence $a(e_{L_1} + \cdots +
      e_{L_r})t^{\nu} = (\sum_{j=2}^r b_j e_{L_j}) t^{\nu}$. Since by
      condition~(3) of Assumption~$\mathscr{L}$ the element $e_{L_1}
      \in Q^+H_n(\partial V)$ is not torsion (over $\Lambda^+$), it
      follows that $a=0$.  Consequently $b_2 = \cdots = b_r =0$ and so
      $\gamma = 0$. This concludes the proof of
      Proposition~\ref{p:basis-delta}.
   \end{proof}

   Having proved Proposition~\ref{p:basis-delta}, the proof of
   Theorem~\ref{t:cob} is now complete.
\end{proof}

\subsection{Lagrangians intersecting at one point} \label{sb:lag-1-pt}
\cntrsb

We start with a stronger version of Corollary~\ref{c:del_1=del_2}
from~\S\ref{sb:discr-intro}.

\begin{cor} \label{c:L1-L2-1pt} Let $(M, \omega)$ be a monotone
   symplectic manifold. Let $L_1, L_2 \subset M$ be two Lagrangian
   submanifolds that satisfy conditions $(1)$ -- $(3)$ of Assumption
   $\mathscr{L}$ and such that $N_{L_1} = N_{L_2}$. Denote by $N =
   N_{L_i}$ their mutual minimal Maslov number and assume further
   that:
   \begin{enumerate}
     \item $H_{1+jN}(L_1)=H_{1+jN}(L_2)=0$ for every $j$;
     \item $H_{jN-1}(L_1) = H_{jN-1}(L_2)=0$ for every $j$;
     \item either $\pi_1(L_1 \cup L_2) \to \pi_1(M)$ is injective, or
      $\pi_1(L_i) \to \pi_1(M)$ is trivial for $i=1,2$.
   \end{enumerate}
   Finally, suppose that $L_1$ and $L_2$ intersect transversely at
   exactly one point. Then $$\Delta_{L_1} = \Delta_{L_2}$$ and
   moreover this number is a perfect square.
\end{cor}
Note that if $L_1$, $L_2$ are even dimensional Lagrangian spheres then
conditions~(1)~--~(3) of Corollary~\ref{c:L1-L2-1pt} are obviously
satisfied, hence Corollary~\ref{c:del_1=del_2} follows from
Corollary~\ref{c:L1-L2-1pt}.

We now turn to the proof of Corollary~\ref{c:L1-L2-1pt}. We will need
the following Proposition.

\begin{prop} \label{p:surgery} Let $L_1, L_2 \subset (M, \omega)$ be
   two Lagrangian submanifolds intersecting transversely at one point.
   Then there exists a Lagrangian cobordism $V \subset \mathbb{R}^2
   \times M$ with three ends, corresponding to $L_1$, $L_2$ and $L_1
   \# L_2$ and such that $V$ has the homotopy type of $L_1 \vee L_2$.
   If $L_1$ and $L_2$ are monotone with the same minimal Maslov number
   $N$ and they satisfy assumption (3) from
   Corollary~\ref{c:L1-L2-1pt} then $V$ is also monotone with minimal
   Maslov number $N_V = N$.  Moreover, if $L_1$ and $L_2$ are spin
   then $V$ admits a spin structure that extends those of $L_1$ and
   $L_2$.
\end{prop}

Before proving this proposition we show how to deduce
Corollary~\ref{c:L1-L2-1pt} from it.

\begin{proof}[Proof of Corollary~\ref{c:L1-L2-1pt}]
   Consider the Lagrangian cobordism provided by
   Proposition~\ref{p:surgery}. Since $V$ is homotopy equivalent to
   $L_1 \vee L_2$ and $L_i$ satisfy assumptions (1) and~(2) of
   Corollary~\ref{c:L1-L2-1pt} a simple calculation shows that
   $$H_{jN}(V, \partial V) = 0, \quad H_{1+jN}(V) = 0, \quad \forall j.$$
   The result now follows immediately from Theorem~\ref{t:cob}.
\end{proof}

We now turn to the proof of the Proposition.

\begin{proof}[Proof of Proposition~\ref{p:surgery}]
   The proof is based on a version of the Polterovich Lagrangian
   surgery~\cite{Po:surgery} adapted to the case of
   cobordisms~\cite{Bi-Co:cob1}.  We briefly outline those parts of
   the construction that are relevant here. More details can be found
   in~\cite{Bi-Co:cob1}.

   Consider two plane curves $\gamma_1, \gamma_2$ as in
   Figure~\ref{f:gamma-12}.
   \begin{figure}[htbp]
      \begin{center} 
         \epsfig{file=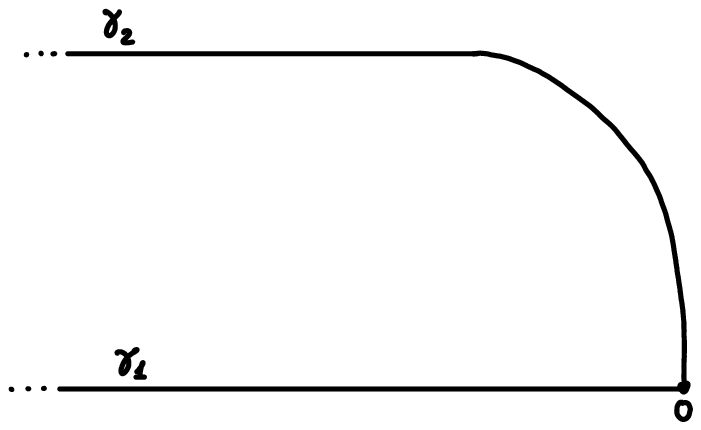, width=0.5\linewidth}
      \end{center}
      \caption{\label{f:gamma-12}}
   \end{figure}
   Consider the Lagrangian submanifolds $\gamma_1 \times L_1, \gamma_2
   \times L_2 \subset \mathbb{R}^2 \times M$. The surgery construction
   from~\cite{Bi-Co:cob1} produces a Lagrangian cobordism $V \subset
   \mathbb{R} \times M$ with two negative ends which coincide with
   negative ends of $\gamma_i \times L_i$ and with whose positive end
   looks like the positive end of $\gamma_3 \times (L_1 \# L_2)$,
   where the curve $\gamma_3$ is depicted in Figure~\ref{f:pi-V} and
   $L_1 \# L_2$ stands for the Polterovich surgery (in $M$) of $L_1$
   and $L_2$ (which coincides with the connected sum of the $L_i$'s
   because they intersect transversely at exactly one point). The
   projection of $V$ to $\mathbb{R}^2$ is depicted in
   Figure~\ref{f:pi-V}.
   \begin{figure}[htbp] 
      \begin{center} 
         \epsfig{file=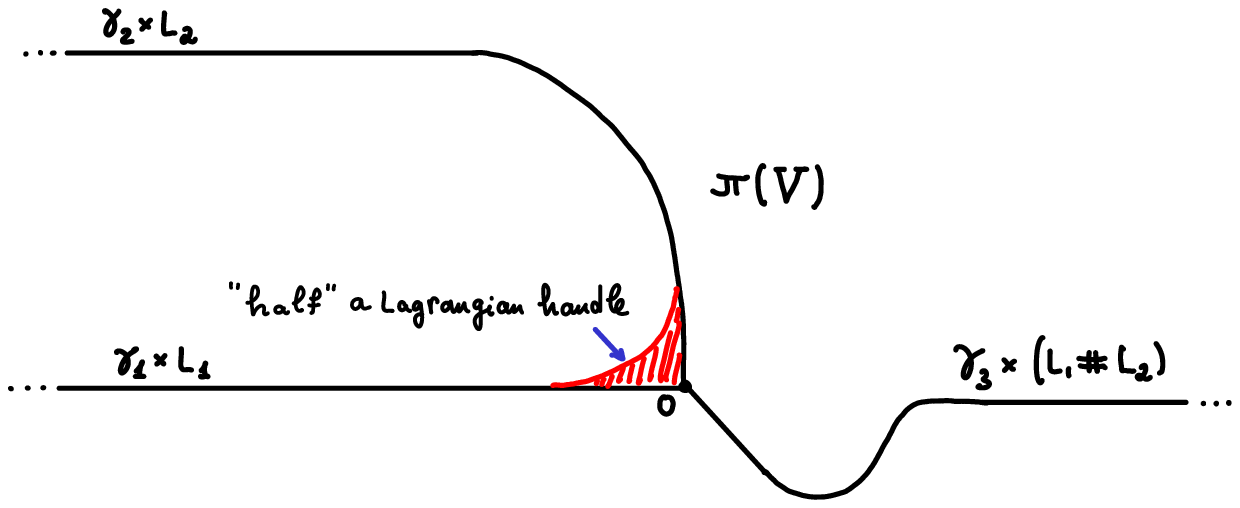, width=0.85\linewidth}
      \end{center}
      \caption{\label{f:pi-V}}
   \end{figure}

   Next we determine the topology of $V$.  Consider the curves
   $\widetilde{\gamma}_1, \widetilde{\gamma}_2$ (which are extensions
   of the $\gamma_i$'s to curves with positive ends as in
   Figure~\ref{f:gamma-tilde-12}.)
   \begin{figure}[htbp]
      \begin{center} 
         \epsfig{file=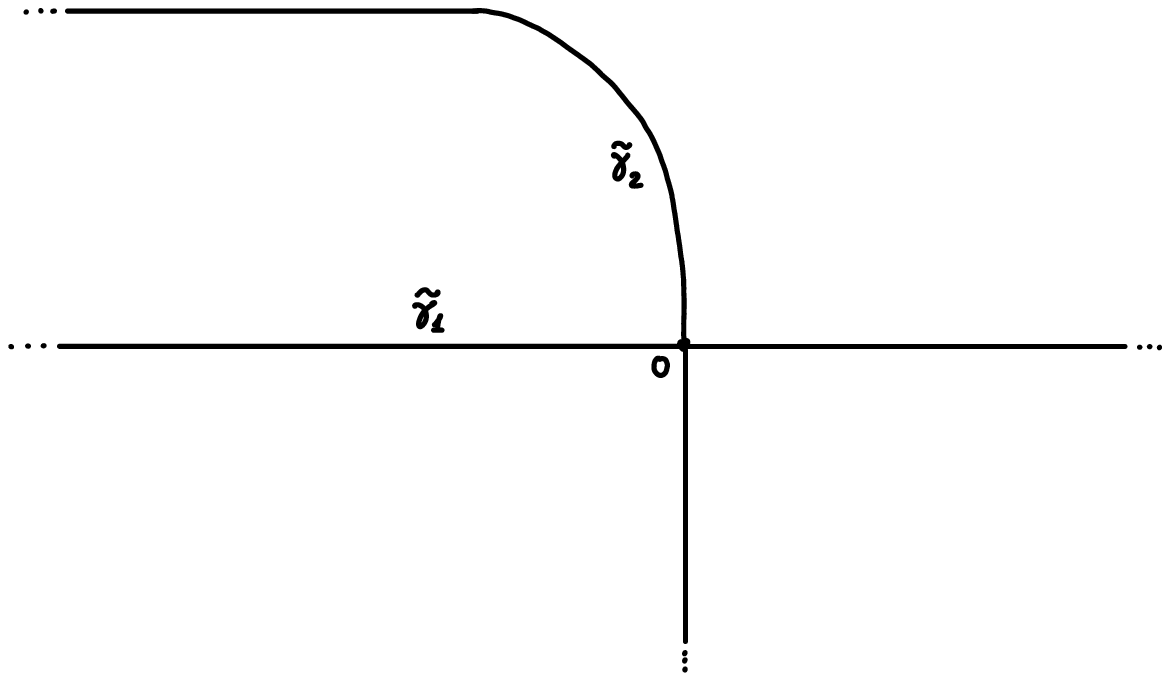, width=0.7\linewidth}
      \end{center}
      \caption{\label{f:gamma-tilde-12}}
   \end{figure}
   Consider the Polterovich surgery $W = (\widetilde{\gamma}_1 \times
   L_1) \# (\widetilde{\gamma}_2 \times L_2) \subset \mathbb{R}^2
   \times M$ (note that the latter two Lagrangians also intersect
   transversely at a single point). See Figure~\ref{f:cob-W}.
   \begin{figure}[htbp]
      \begin{center} 
         \epsfig{file=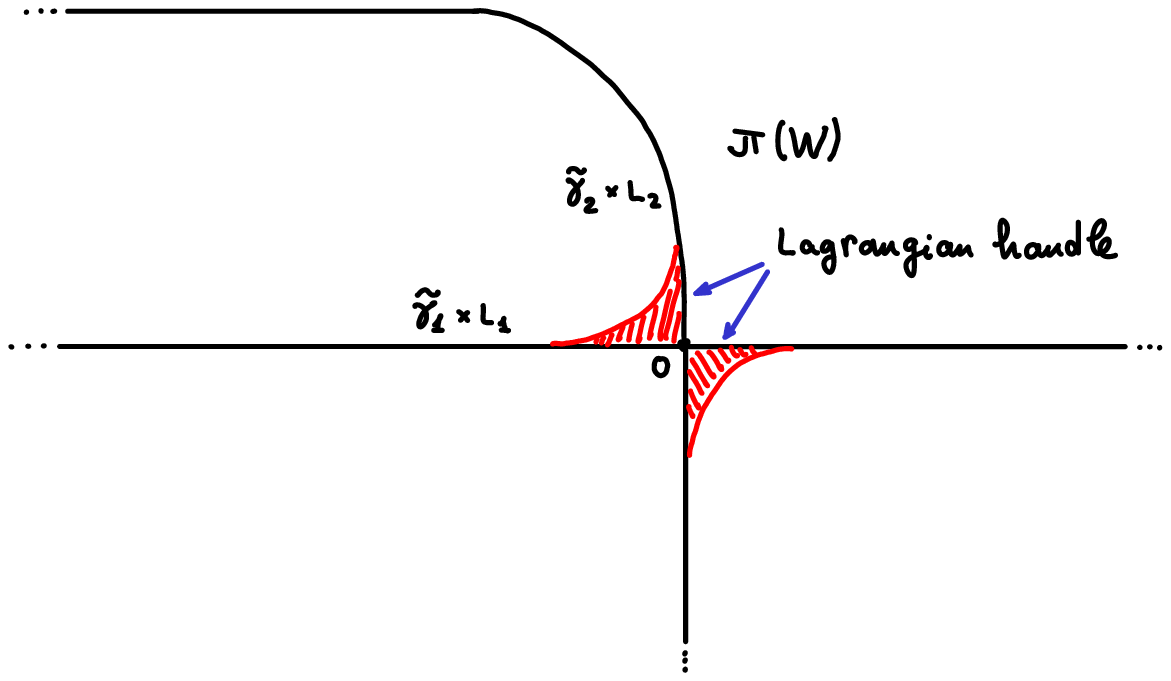, width=0.85\linewidth}
      \end{center}
      \caption{\label{f:cob-W}}
   \end{figure}
   \begin{figure}[htbp]
      \begin{center} 
         \epsfig{file=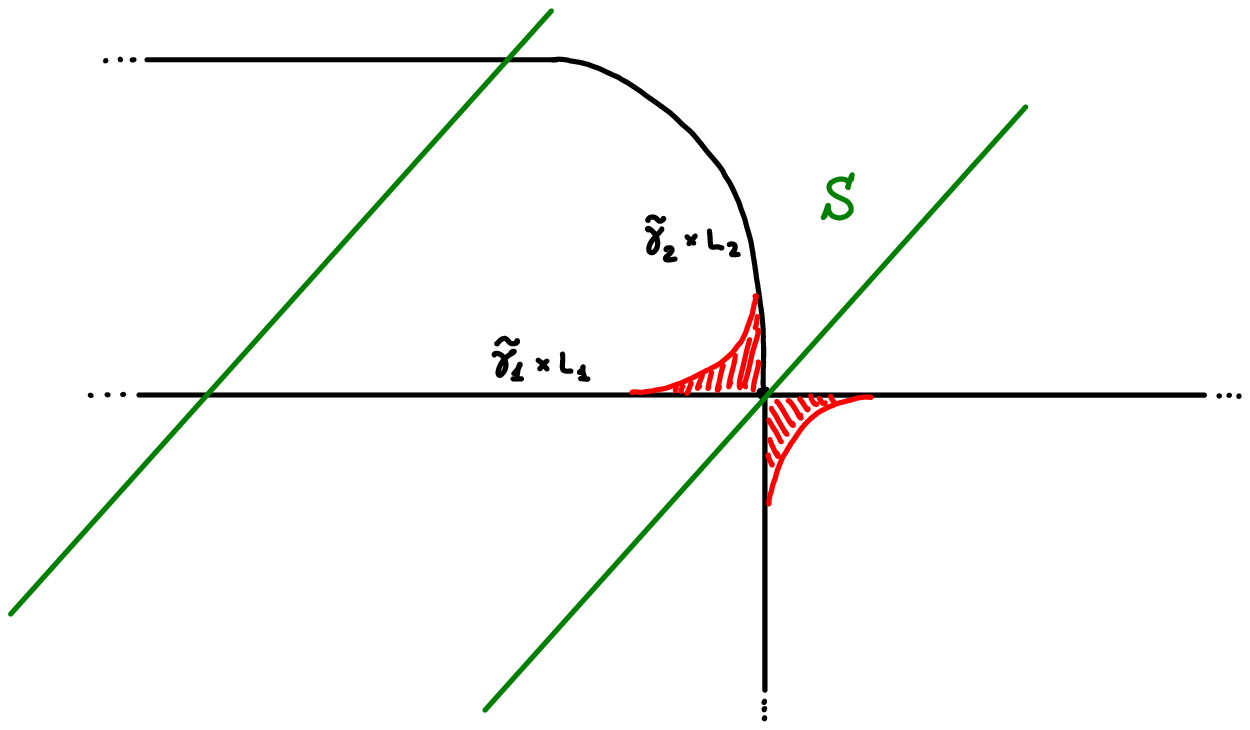, width=0.85\linewidth}
      \end{center}
      \caption{\label{f:V_0-strip-S}}
   \end{figure}

   Denote by $\pi: \mathbb{R}^2 \times M \longrightarrow \mathbb{R}^2$
   the projection, and by $S \subset \mathbb{R}^2$ the strip depicted
   in Figure~\ref{f:V_0-strip-S}. Put $V_0 = W \cap \pi^{-1}(S)$.
   According to~\cite{Bi-Co:cob1}, $V_0$ is a manifold with boundary,
   with two obvious boundary components corresponding to the $L_i$'s
   and a third boundary component which is $W \cap \pi^{-1}(0)$. The
   latter is exactly the Polterovich surgery $L_1 \# L_2$.  Moreover
   $V_0$ is homotopy equivalent to $V$ (in fact $V_0 \subset V$ and is
   a deformation retract of $V$). A straightforward calculation shows
   that there is an embedding $L_1 \vee L_2 \subset V_0$ and moreover
   that $L_1 \vee L_2$ is a deformation retract of $V_0$. (In fact,
   one can show that $V_0$ is diffeomorphic to the boundary connected
   sum of $[0,1] \times L_1$ and $[0,1] \times L_2$, where the
   connected sum occurs among the boundary components $\{1\} \times
   L_i$, $i=1,2$.)
   
   The statement on monotonicity follows from the Seifert - Van~Kampen
   theorem (see also~\cite{Bi-Co:cob1}).

   Assume now that $L_1, L_2$ are spin. Then $\widetilde{\gamma}_1
   \times L_1$ and $\widetilde{\gamma}_2 \times L_2$ are also spin,
   with a spin structure extending those of the ends. Recall that the
   connected sum of spin manifolds is also
   spin~\cite{La-Mi:spin-geom}. Thus $W = (\widetilde{\gamma}_1 \times
   L_1) \# (\widetilde{\gamma}_2 \times L_2)$ is spin too and by
   standard arguments it follows that the spin structure on $W$ can be
   chosen so that it extends those given on the ends. By restriction
   we obtain a spin structure on $V_0 \subset W$ and consequently also
   the desired one on $V$.
\end{proof}


%% file: examples.tex

\section{Examples} \label{s:examples} \cntrs 

This section is a continuation of~\S\ref{sb:exps} in which we provide
more details to the examples. We will work here with the following
setting. $(M, \omega)$ will be a monotone symplectic manifold with
minimal Chern number $C_M$. To keep the notation short we will denote
here by $QH(M)$ the quantum homology of $M$ with coefficients in the
ring $R=\mathbb{Z}[q^{-1}, q]$ (with $|q|=-2$), instead of writing
$QH(M;R)$.

\subsection{Lagrangian spheres in symplectic blow-ups of
  $\mathbb{C}P^2$} \label{sb:lag-spheres-blow-ups} \cntrsb

Denote as in~\S\ref{sbsb:intro-exp-blcp2} by $M_k$ the blow-up of
${\mathbb{C}}P^2$ at $k \leq 6$ points endowed with a K\"{a}hler
symplectic structure $\omega_k$ in the cohomology class of $c_1 \in
H^2(M_k)$.  Note that $-K_{M_k}$ is ample hence $c_1$ represents a
K\"{a}hler class. Note that $C_{M_k}=1$. As will be seen
in~\S\ref{s:non-monotone} some of our results (e.g.
Theorem~\ref{t:cubic-eq-sphere}) continue to hold in dimension $4$
also for non-monotone Lagrangian spheres. In this section however we
still stick to the monotone case.

We first claim that the set of classes in $H_2(M_k)$ which are
represented by Lagrangian spheres are precisely those that appear in
Table~\ref{tb:classes-intro}. This is well known and there are many
ways to prove it (see e.g.~\cite{Se:lect-4-Dehn, Ev:lag-spheres,
  Li-Wu:lag-spheres, Shev:secondary-stw}). For the classes $A = E_i -
E_j \in H_2(M_k)$ when $k=2$ and $k=3$ it is easy to find Lagrangian
spheres in the class $A$ by an explicit construction which we outline
below (see~\cite{Ev:lag-spheres} for more details). For $k \geq 4$, as
well as $k=3$ with $A = H - E_1 - E_2 - E_3$, it seems less trivial to
perform explicit constructions and one could appeal instead to less
transparent methods such as (relative) inflation, as in
in~\cite{Li-Wu:lag-spheres, Shev:secondary-stw} (we will briefly
outline this in a special case below). Another approach which works
for some of the $k$'s is to realize $M_k$ as a fiber in a Lefschetz
pencil and obtain the Lagrangian spheres as vanishing cycles (e.g.
$M_6$ is the cubic surface in ${\mathbb{C}}P^3$ and $M_5$ is a
complete intersection of two quadrics in ${\mathbb{C}}P^4$). Yet
another approach comes from real algebraic geometry, where one can
obtain Lagrangian spheres in some of the $M_k$'s as a component of the
fixed point set of an anti-symplectic involution. This works for $k=5,
6$ and all classes $A$, and for $k=3$ with $A = E_i-E_j$.
See~\cite{Koll:real-alg-surf} for more details. Finally note that for
$2 \leq k \leq 8$, $k \neq 3$, the group of symplectomorphisms of
$M_k$ acts transitively on the set of classes that can be represented
by Lagrangian spheres~\cite{Demazure:del-pezzo-1, Li-Wu:lag-spheres},
hence it is enough to construct one Lagrangian sphere in each $M_k$.
(This also explains why the invariants in Table~\ref{tb:classes-intro}
coincide for different classes within each of the $M_k$'s with the
exception $k=3$.)

Despite the many ways to establish Lagrangian spheres in the $M_k$'s,
the shortest (albeit not the most explicit) path to this end is to
appeal to the work Li-Wu~\cite{Li-Wu:lag-spheres}. According
to~\cite{Li-Wu:lag-spheres} a homology class $A \in H_2(M_k)$ can be
represented by a Lagrangian sphere iff it satisfies the following two
conditions:
\begin{enumerate}
  \item[(LS-1)] $A$ can be represented by a smooth embedded $2$-sphere.
  \item[(LS-2)] $\langle [\omega_k], A \rangle = 0$.
  \item[(LS-3)] $A \cdot A  = -2$.
\end{enumerate}
We remark again that we have assumed that $[\omega_k] = c_1$
(otherwise one has to assume in addition that $\langle c_1, A \rangle
= 0$).

It is straightforward to see that all the classes in
Table~\ref{tb:classes-intro} satisfy conditions~(LS-2) and~(LS-3)
above. As for condition~(LS-1), note that if $C', C'' \subset M^4$ are
two {\em disjoint} embedded smooth $2$-spheres in a $4$-manifold
$M^4$, then by performing the connected sum operation one obtains a new smooth
embedded $2$-sphere in the class $[C']+[C'']$. From this it follows
that any non-trivial class of the form $\sum_{i=1}^k \epsilon_i E_i$
with $\epsilon_i \in \{-1, 0, 1\}$ can be represented by a smooth
embedded $2$-sphere. This settles the cases $\pm(E_i - E_j)$. For the
other type of classes, note that $H$ and $2H$ can both be represented
by smooth embedded $2$-spheres (e.g. a projective line and a conic
respectively) hence the same holds also for for classes of the form
$\pm(H - E_i - E_j - E_l)$ and $\pm(2H - \sum_{i=1}^6 E_i)$.

We remark that in fact there are no other classes but the ones in
Table~\ref{tb:classes-intro} that can be represented by Lagrangian
spheres in $M_k$. This can be proved by elementary means using
conditions~(LS-2) and~(LS-3) above.

\subsubsection{Construction of Lagrangian spheres in $M_2$ and $M_3$}
\label{sbsb:constructions}
We now outline a more explicit way to construct Lagrangian spheres in
some of the $M_k$'s (c.f.~\cite{Ev:lag-spheres}). Consider $Q =
\mathbb{C}P^1 \times \mathbb{C}P^1$ endowed with the symplectic form
$\omega = 2\omega_{\mathbb{C}P^1} \oplus 2\omega_{\mathbb{C}P^1}$,
where $\omega_{\mathbb{C}P^1}$ is the standard K\"{a}hler form on
$\mathbb{C}P^1$ normalized so that $\mathbb{C}P^1$ has area $1$. Note
that the first Chern class of $Q$ satisfies $c_1 = [\omega]$.  The
symplectic manifold $Q$ contains a Lagrangian sphere
$\overline{\Delta}$ in the class $[\mathbb{C}P^1 \times
\textnormal{pt}] - [\textnormal{pt} \times \mathbb{C}P^1]$ (i.e. the
class of the anti-diagonal). For example, we can write
$\overline{\Delta}$ as the graph of the antipodal map, given in
homogeneous coordinates by
\[
\mathbb{C}P^1 \longrightarrow \mathbb{C}P^1, \qquad [z_0 : z_1]
\longmapsto [-\overline{z_1} : \overline{z_0}].
\]
Next, we claim that $Q$ admits a symplectic embedding of two disjoint
closed balls $B_1, B_2$ of capacity $1$ whose images are disjoint from
$\overline{\Delta}$. This can be easily seen from the toric picture.
Indeed the image of the moment map of $Q$ is the square $[0,2] \times
[0,2]$ and the image of $\overline{\Delta}$ under that map is given by
the anti-diagonal $\{(x, y) \mid x,y \in [0,2], x+y=2\}$. By standard
arguments in toric geometry we can symplectically embed in $Q$ a ball
$B_1$ of capacity $1$ whose image under the moment map is $\{(x,y)
\mid x,y \in [0,2], x+y \leq 1\}$. Similarly we can embed another ball
$B_2$ whose image is $\{(x,y) \mid x,y \in [0,2], x+y \geq 3\}$.
Clearly $B_1$, $B_2$ and $\overline{\Delta}$ are mutually disjoint.
Denote by $\widetilde{Q}_1$ the blow-up of $Q$ with respect to $B_1$
and by $\widetilde{Q}_2$ the blow-up of $Q$ with respect to both balls
$B_1$ and $B_2$. It is well known that $\widetilde{Q}_1$ is
symplectomorphic to $M_2$ via a symplectomorphism that sends the class
$\overline{\Delta}$ to $E_1 - E_2$. And $\widetilde{Q}_2$ is
symplectomorphic to $M_3$ by a similar symplectomorphism.  It follows
that $E_1 - E_2$ represents Lagrangian spheres both in $M_2$ and in
$M_3$. Construction of Lagrangian spheres in the other classes of the
type $E_i - E_j$ in $M_3$ can be done in a similar way.

\subsubsection*{Lagrangian spheres in the class $H-E_1-E_2-E_3$ in $M_3$}
We start with the complex blow-up of ${\mathbb{C}}P^2$ at three points
that {\em lie on the same projective line}.  Denote by $E_i$ the
exceptional divisors over the blown-up points. The result of the blow
up is a complex algebraic surface $X$ which contains an embedded
holomorphic rational curve $\Sigma$ in the class $H - E_1 - E_2 -
E_3$. Note also that there are three embedded holomorphic curves $C_i
\subset X$, $i=1,2,3$, in the classes $[C_i] = H - E_i$.  Since $[C_i]
\cdot [\Sigma] = 0$ the curves $C_i$ are disjoint from $\Sigma$. Pick
a K\"{a}hler symplectic structure $\omega_0$ on $X$.  After a suitable
normalization we can write $[\omega_0] = h - \lambda_1 e_1 - \lambda_2
e_2 - \lambda_3 e_3$, where $h, e_1, e_2, e_3$ are the Poincar\'{e}
duals to $H, E_1, E_2, E_3$ respectively.  It is easy to check that
$\lambda_i \geq 0$ and that $\lambda_1 + \lambda_2 + \lambda_3 < 1$.
We now change $\omega_0$ to a new symplectic form $\omega'$ such that:
\begin{enumerate}
  \item $\omega'$ coincides with $\omega_0$ outside a small
   neighborhood $\mathcal{U}$ of $\Sigma$, where $\mathcal{U}$ is
   disjoint from the curves $C_1, C_2, C_3$.
  \item $\omega'|_{T(\Sigma)} \equiv 0$, i.e. $\Sigma$ becomes a
   Lagrangian sphere with respect to $\omega'$.
  \item $\omega'$ and $\omega$ are in the same deformation class of
   symplectic forms on $X$ (i.e. they can be connected by a path of
   symplectic forms).
\end{enumerate}
This can be achieved for example using the {\em deflation}
procedure~\cite{Shev:secondary-stw} (see also~\cite{Li-Ush:neg-sq}).
Alternatively, one can construct $\omega'$ using Gompf fiber-sum
surgery~\cite{Go:fiber-sum-surgery} with respect to $\Sigma \subset X$
and the diagonal in $\mathbb{C}P^1 \times \mathbb{C}P^1$: $$(Y,
\omega'') = (X, \omega_0) \, _\Sigma\#_{\textnormal{diag}} \,
(\mathbb{C}P^1 \times \mathbb{C}P^1, a \omega_{\mathbb{C}P^1} \oplus a
\omega_{\mathbb{C}P^1}),$$ where $a = \tfrac{1}{2} \int_{\Sigma}
\omega_0$, and $S^2$ is symplectically embedded in $X$ as $\Sigma$ and
in $\mathbb{C}P^1 \times \mathbb{C}P^1$ as the diagonal. Since the
anti-diagonal $\overline{\Delta}$ is a Lagrangian sphere in
$\mathbb{C}P^1 \times \mathbb{C}P^1$ which is disjoint from the
diagonal it gives rise to a Lagrangian sphere $L'' \subset Y$.
Finally observe that the surgery has not changed the diffeomorphism
type of $X$, namely there exists a diffeomorphism $\phi: Y
\longrightarrow X$ and moreover $\phi$ can be chosen in such a way
that $\phi(L'')=\Sigma$. Take now $\omega' = \phi_* \omega''$. To
obtain a symplectic deformation between $\omega'$ and $\omega_0$ one
can perform the preceding surgery in a suitable one-parametric family,
where the symplectic form on $\mathbb{C}P^1 \times \mathbb{C}P^1$ is
rescaled so that the area of one of the factors becomes smaller and
smaller and the area of the other increases so that the area of the
diagonal stays constant.

Having replaced the form $\omega_0$ by $\omega'$ we have a Lagrangian
sphere in the desired homology class $H-E_1-E_2-E_3$ but the form
$\omega'$ might not be in the cohomology class of $c_1$. We will now
correct that using inflation.

After a normalization we can assume that $[\omega'] = h - \lambda'_1
e_1 - \lambda'_2 e_2 - \lambda'_3 e_3$. Since $\Sigma$ is Lagrangian
with respect to $\omega'$ we have $\lambda'_1 + \lambda'_2 +
\lambda'_3 = 1$. Recall also that the surfaces $C_1, C_2, C_3$ are
symplectic with respect to $\omega'$, hence $\lambda'_i < 1$ for
every $i$. Moreover, by construction, the surfaces $C_1, C_2, C_3$ can
be made simultaneously $J$-holomorphic for some $\omega'$-compatible
almost complex structure $J$. Since the $C_i$'s are disjoint from
$\Sigma$ we can find neighborhoods $U_i$ of $C_i$ such that the
$U_i$'s are disjoint from $\Sigma$. We now perform inflation
simultaneously along the three surfaces $C_1, C_2, C_3$. More
specifically, by the results of~\cite{Bi:connected-pack,
  Bi:connected-pack-arxiv} there exist closed $2$-forms $\rho_i$
supported in $U_i$, representing the Poincar\'{e} dual of $[C_i]$
(i.e. $[\rho_i] = h-e_i$) and such that the $2$-form
$$\omega_{t_1, t_2, t_3} = \omega' + t_1
\rho_1 + t_2 \rho_2 + t_3 \rho_3$$ is symplectic for every $t_1, t_2,
t_3 \geq 0$. See Lemma~2.1 in~\cite{Bi:connected-pack} and
Proposition~4.3 in~\cite{Bi:connected-pack-arxiv} (see
also~\cite{La:isotopy-balls, LM:classif-1, LM:classif-2, McD:From,
  McD-Op:sing-inf}.) The cohomology class of $\omega'_t$ is:
$$[\omega'_t] = (1 + t_1 + t_2 + t_3)h - (\lambda'_1+t_1)e_1 - 
(\lambda'_2+t_2)e_2 - (\lambda'_3 + t_3)e_3.$$ Choosing $t^0_i =
1-\lambda'_i$ we have $t^0_i>0$ and $1+t_1^0 + t_2^0 + t_3^0 = 4 -
(\lambda'_1 + \lambda'_2 + \lambda'_3) = 3$, hence:
$$[\omega'_{t^0_1, t_2^0, t_3^0}] = 3h - e_1 - e_2 - e_3 = c_1.$$
Due to the support of the forms $\rho_i$ the surface $\Sigma$ remains
Lagrangian for $\omega'_{t^0_1, t^0_2, t^0_3}$. Finally note that
$\omega'_{t^0_1, t^0_2, t^0_3}$ is in the same symplectic deformation
class of $\omega_0$ hence by standard results $(X, \omega'_{t^0_1,
  t^0_2, t^0_3})$ is symplectomorphic to $M_3$.

\subsubsection{Calculation of the discriminant for $M_k$, $2 \leq k
  \leq 6$}

We now give more details on the calculation of the discriminant
$\Delta_L$ for each of the examples in Table~\ref{tb:classes-intro}.
In what follows, for a symplectic manifold $M$, we denote by $p \in
H_0(M)$ the homology class of a point. As before we write $QH(M)$ for
the quantum homology ring of $M$ with coefficients in $R =
\mathbb{Z}[q^{-1}, q]$ where $|q| = -2$. The calculations below make
use of the ``multiplication table'' of the quantum homology of the
$M_k$'s which can be found in~\cite{Cra-Mir:QH}.

Recall that for $M_k$ with $4 \leq k \leq 6$ the group of
symplectomorphisms of $M_k$ acts transitively on the set of classes
that can be represented by Lagrangian
spheres~\cite{Demazure:del-pezzo-1, Li-Wu:lag-spheres}. Therefore, for
$k \geq 4$ we will perform explicit calculations only for Lagrangians
in the class $E_1 - E_2$.

Before we go on we remark that all the calculations for the $M_k$'s
below extend without any change in case we endow $M_k$ with a
non-monotone symplectic structure (provided that a Lagrangian sphere
in the respective class still exists). This is special to dimension
$4$ and is explained in detail in~\S\ref{s:non-monotone}.

\subsubsection{2-point blow-up of $\mathbb{C}P^2$} \label{sbsb:M_2}
\cntrsb $QH(M_2)$ has the following ring structure:
\begin{align*}
& p * p  =  H  q^3 + [M_2]q^4 \\
& p * H  =  (H - E_1) q^2 + (H - E_2)  q^2 + [M_2] q^3 \\
& p * E_i  =  (H - E_i) q^2  \\
& H * H  =  p + (H - E_1 - E_2) q + 2[M_2] q^2 \\
& H * E_i  =  (H - E_1 - E_2) q + [M_2] q^2  \\
& E_1 * E_2  =  (H - E_1 - E_2) q   \\
& E_1 * E_1  =  -p + (H - E_2) q + [M_2] q^2 \\
& E_2 * E_2  =  -p + (H - E_1) q + [M_2] q^2.
\end{align*}
Consider Lagrangian spheres $L \subset M_2$ in the class $E_1-E_2$. A
straightforward calculation shows that:
\[
(E_1 - E_2)^{*3} - 5 (E_1 - E_2) q^2 = 0,
\]
and thus we obtain $\Delta_L = 5$. Multiplication of $c_1$ with $[L]$
gives: $c_1 * (E_1 - E_2) = (-1)(E_1 - E_2)q$, hence $\lambda_L = -1$.
The associated ideal (see~\S\ref{sb:eigneval}) $\mathcal{I}_L \subset
QH_*(M_2)$ is:
\[
\mathcal{I}(E_1 - E_2) = R (-2p + (E_1 + E_2)q + 2 [M_2] q^2) \oplus R
(E_1 - E_2).
\]

We now turn to Theorem~\ref{t:cubic_eq-ccl} and calculate explicitly
the coefficients $\sigma_{c,L}$, $\tau_{c,L}$ from
equation~\eqref{eq:cubic_eq_ccL}. Consider a general element $c = dH -
m_1 E_1 - m_2 E_2\in H_2(M_2)$, where $d, m_1, m_2 \in \mathbb{Z}$.
Then $\xi := c \cdot [L] = m_1 - m_2$ and we assume that $m_1 \neq
m_2$. A straightforward calculation gives:
$$\sigma_{c,L} = - \frac{m_1+m_2}{m_1-m_2}, \quad
\tau_{c,L} = \frac{m_1^2 - 3m_1 m_2 + m_2^2}{(m_1-m_2)^2}.$$ One can
easily check that $\sigma_{c,L}^2 + 4 \tau_{c,L} = 5$.

\subsubsection{3-point blow-up of $\mathbb{C}P^2$}
$QH(M_3)$ has the following ring structure:
\begin{align*}
   & p * p  =  (3H - E_1 - E_2 - E_3)q^3 + 3 [M_3]q^4 \\
   & p * H  =  (3H - E_1 - E_2 - E_3)q^2 + 3 [M_3] q^3 \\
   & p * E_i  =  (H - E_i)q^2 + [M_3]q^3 \\
   & H * H  =  p + (3H - 2E_1 - 2E_2 - 2E_3)q + 3 [M_3]q^2 \\
   & H * E_i  =  (2H - 2E_i - E_j - E_k)q + [M_3]q^2,
   \quad  i \neq j \neq k \neq i \\
   & E_i * E_i  =  -p + (2H - E_1 - E_2 - E_3)q + [M_3]q^2 \\
   & E_i * E_j = (H - E_i - E_j )q, \quad i \neq j.
\end{align*}
Consider Lagrangians $L, L' \subset M_3$ in the classes $[L] = E_i -
E_j$ and $[L'] = H - E_1 - E_2 - E_3$. The corresponding Lagrangian
cubic equations are given by:
\begin{align*}
   & (E_i - E_j)^{*3} - 4 (E_i - E_j) q^2  =  0, \\
   & (H - E_1 - E_2 - E_3)^{*3} + 3 (H - E_1 - E_2 - E_3) q^2 =  0,
\end{align*}
and thus obtain $\Delta_L = 4$ and $\Delta_{L'} = -3$.  Multiplication
with $c_1$ gives:
\begin{align*}
   & c_1 * (E_i - E_j)  =  (-2)(E_i - E_j)t, \\
   & c_1 * (H - E_1 - E_2 - E_3) = (-3) (H - E_1 - E_2 - E_3)t,
\end{align*}
hence $\lambda_L = -2$ and $\lambda_{L'} = -3$. The associated ideals
in $QH(M_3)$ are:
\begin{align*}
   & \mathcal{I}_L =
   R (-2p + 2(H - E_3)t + 2[M_3]q^2) \oplus R (E_1 - E_2), \\
   & \mathcal{I}_{L'} = R(-2p + (3H - E_1 - E_2 - E_3)q + 4 [M_3]q^2 )
   \oplus R (H - E_1 - E_2 - E_3).
\end{align*}
The Lagrangian spheres in different homology classes of the type $E_i
- E_j$ in $M_3$ have the same discriminant and the same eigenvalue
$\lambda_L$. This is so because for every $i<j$ there is a
symplectomorphism $\varphi: M_3 \longrightarrow M_3$ such that
$\varphi_*(E_1-E_2) = E_i - E_j$. In contrast, note that there exists
no symplectomorphism of $M_3$ sending $E_1-E_2$ to $H - E_1 - E_2 -
E_3$.

\subsubsection{4-point blow-up of $\mathbb{C}P^2$} $QH(M_4)$ has the
following ring structure:
\begin{align*}
   & p * p  = (9H - 3E_1 - 3E_2 - 3E_3 - 3E_4)q^3 + 10  [M_4]q^4 \\
   & p * H  = (8H - 3E_1 - 3E_2 - 3E_3 - 3E_4)q^2 + 9[M_4] q^3 \\
   & p * E_i  = (3H - 2E_i - \sum\limits_{j \neq i}E_j )q^2 + 3[M_4]q^3 \\
   & H*H    = p + (6H - 3E_1 - 3E_2 - 3E_3 - 3E_4)q + 8[M_4] q^2  \\
   & H*E_i  =  (3H - 3E_i - \sum\limits_{j \neq i}E_j)q + 3[M_4] q^2 \\
   & E_i*E_i  =  -p + (3H - 2E_i - \sum\limits_{j \neq i}E_j)q + 2[M_4] q^2 \\
   & E_i*E_j  = (H - E_i - E_j)q + [M_4] q^2
\end{align*}
As explained above it is enough to calculate our invariants for
Lagrangians in the class $E_1-E_2$. A straightforward calculation
shows that: $$(E_1-E_2)^{*3} = (E_1-E_2)q^2, \quad c_1*(E_1 - E_2) =
-3 (E_1-E_2)q,$$ hence $\Delta_L = 1$ and $\lambda_L = -3$. The
associated ideals for Lagrangians $L$, $L'$ with $[L] = E_1-E_2$ and
$L' = H - E_1 - E_2 - E_3$ are:
\begin{align*}
   & \mathcal{I}_L = R (-2p + (4H - E_1 - E_2 - 2E_3 - 2E_4)q +
   2[M_4]q^2)
   \oplus R (E_1 - E_2), \\
   & \mathcal{I}_{L'} = R (-2p + (3H - E_1 - E_2 - E_3)q + 2[M_4]q^2)
   \oplus R (H - E_1 - E_2 - E_3).
\end{align*}

\subsubsection{5-point blow-up of $\mathbb{C}P^2$} $QH(M_5)$ has the
following ring structure:
\begin{align*}
   & p * p  = (36H - 12E_1 - 12E_2 - 12E_3 - 12E_4 - 12E_5)q^3 + 52[M_5]q^4  \\
   & p * H  = (25H - 9E_1 - 9E_2 - 9E_3 - 9E_4 - 9E_5)q^2 + 36[M_5]q^3  \\
   & p * E_i  = (9H - 5E_i - 3\sum\limits_{j \neq i}E_j)q^2 + 12[M_5]q^3 \\
   & H*H      = p + (18H - 8E_1 - 8E_2 - 8E_3 - 8E_4 - 8E_5)q + 25[M_5]q^2  \\
   & H*E_i  = (8H - 6E_i - 3\sum\limits_{j \neq i}E_j)q + 9[M_5]q^2 \\
   & E_i*E_i  = -p + (6H - 4E_i - 2\sum\limits_{j \neq i}E_j)q + 5[M_5]q^2 \\
   & E_i*E_j = (3H - 2E_i - 2E_j - \sum\limits_{k \neq i,j}E_k)q +
   3[M_5]q^2
\end{align*}
As before, it is enough to consider only the case $[L] = E_1 - E_2$.
A direct calculation gives:
$$(E_1-E_2)^{*3} = 0, \quad c_1*(E_1-E_2) =  -4(E_1 - E_2)q,$$
hence $\Delta_L = 0$, $\lambda_L = -4$. 

The associated ideals for Lagrangians $L$, $L'$ with $[L] = E_1 - E_2$
and $[L'] = H - E_1 - E_2 - E_3$ are:
\begin{align*}
   & \mathcal{I}_L = R (-2p + 
   (6H - 2E_1 - 2E_2 - 2E_3 - 2E_4 - 2E_5)q + 4[M_5]q^2) \oplus R (E_1 - E_2), \\
   & \mathcal{I}_{L'} = R (-2p + (6H - 2E_1 - 2E_2 - 2E_3 - 2E_4 - 2 E_5)q + 
   4[M_5]q^2) \oplus R (H - E_1 - E_2 - E_3).
\end{align*}

\subsubsection{6-point blow-up of $\mathbb{C}P^2$} $QH(M_6)$ has the
following the ring structure:
\begin{align*}
   & p * p = (252H - 84 E_1 - 84 E_2 - 84 E_3 - 84 E_4 - 84 E_5 -
   84 E_6)q^3 + 540[M_6]q^4 \\
   & p * H = (120 H - 42 E_1 - 42 E_2 - 42 E_3 - 42 E_4 - 42 E_5 -
   42 E_6)q^2 + 252 [M_6] q^3 \\
   & p * E_i  =  ( 42H - 20 E_i - 14\sum\limits_{j\neq i} E_j )q^2 + 84 [M_6]q^3 \\
   & H*H = p + (63H - 25 E_1 - 25 E_2 - 25 E_3 - 25 E_4 - 25 E_5 -
   25 E_6)q + 120[M_6]q^2 \\
   & H*E_i  =  (25H - 15 E_i - 9\sum\limits_{j\neq i} E_j)q + 42 [M_6] q^2 \\
   & E_i*E_i  =  -p + (15H - 9E_i - 5\sum\limits_{j\neq i} E_j )q + 20 [M_6] q^2 \\
   & E_i*E_j = (9H - 5 E_i - 5 E_j - 3 \sum\limits_{k\neq i,j} E_j)q +
   14 [M_6] q^2
\end{align*}
Again, we may assume without loss of generality that $[L] = E_1 -
E_2$. A direct calculation gives:
$$(E_1-E_2)^{*3} = 0, \quad c_1*(E_1-E_2) =  -6(E_1 - E_2)q,$$
hence $\Delta_L = 0$, $\lambda_L = -6$.

Interestingly, the associated ideals $\mathcal{I}_L$ for Lagrangians
$L$ in any of the classes: $E_i - E_j$, $2H - E_i - E_j - E_l$, $2H -
E_1 - E_2 - E_3 - E_4 - E_5 - E_6$ all coincide:
$$\mathcal{I}_L =  
R ( -2p + (12H - 4 \sum_{j=1}^6 E_j)q + 12[M_6]q^2) \bigoplus R (2H -
\sum_{j=1}^6 E_j).$$

\begin{rem} \label{r:1-point}
   Note that all Lagrangian spheres in each of $M_4$, $M_5$ and $M_6$
   have the same discriminant and the same holds for the Lagrangian
   spheres in $M_3$ in the classes $E_1 - E_2$, $E_2-E_3$ and $E_1 -
   E_3$.  This follows of course from the fact that all these classes
   belong to the same orbit of the action of the symplectomorphism
   group (on each of the $M_k$'s). However, here is a different
   potential explanation which might give more insight. Consider for
   example the classes $E_1 - E_2$ and $E_2 - E_3$ in $M_3$. It seems
   reasonable to expect that there exist Lagrangian spheres $L_1, L_2
   \subset M_3$ with $[L_1] = E_1 - E_2$, $[L_2] = E_2 - E_3$ such
   that $L_1$ and $L_2$ intersect transversely at exactly one point.
   (We have not verified the details of that, but this seems plausible
   in view of the constructions outlined at the beginning
   of~\S\ref{sbsb:constructions}). The fact that $\Delta_{L_1} =
   \Delta_{L_2}$ would now follow from Corollary~\ref{c:del_1=del_2}.
   Similar arguments should apply to many other pairs of classes on
   $M_4$, $M_5$ and $M_6$. This would also explain why in all these
   cases the discriminants turn out to be perfect squares.
\end{rem}

\subsection{Lagrangian spheres in hypersurfaces of
  ${\mathbb{C}}P^{n+1}$} \cntrsb \label{sb:exp-hypsurf} 

Let $M^{2n} \subset \mathbb{C}P^{n+1}$ be a Fano hypersurface of
degree $d$, where $n \geq 3$. We endow $M$ with the symplectic
structure induced from ${\mathbb{C}}P^{n+1}$. It is easy to check that
$M$ is monotone and that the minimal Chern number is $C_M = n+2 - d$.

We view the homology $H_*(M;\mathbb{Q})$ as a ring, endowed with the
intersection product which we denote by $a \cdot b$ for $a, b \in
H_*(M;\mathbb{Q})$. Write $h \in H_{2n-2}(M;\mathbb{Q})$ for the class
of a hyperplane section. The homology $H_*(M;\mathbb{Q})$ is generated
as a ring by the class $h$ and the subspace of primitive classes,
denoted by $H_n(M; \mathbb{Q})_0$.  (Recall that the latter is by
definition the kernel of the map $H_n(M; \mathbb{Q}) \longrightarrow
H_{n-2}(M;\mathbb{Q})$, $a \longmapsto a \cdot h$).

Assume that $d \geq 2$. Then by Picard-Lefschetz theory $M$ contains
Lagrangian spheres (that can be realized as vanishing cycles of the
Lefschetz pencil associated to the embedding $M \subset
{\mathbb{C}}P^{n+1}$). 

Let $L \subset M$ be a Lagrangian sphere and assume further that $d
\geq 3$. To calculate $[L]^{*3}$ we appeal to the work of
Collino-Jinzenji~\cite{Co-Ji:QH-hpersurf} (see
also~\cite{Gi:equivariant, Beau:quant, Ti:qh-assoc} for related
results). We set $x := h + d![M]q$ if $C_M = 1$, and $x := h$, if $C_M
\geq 2$. Specifically, we will need the following:
\begin{thm}[Collino-Jinzenji~\cite{Co-Ji:QH-hpersurf}]
   \label{t:qh-hypersuf} 
   In the quantum homology ring of $M$ with coefficients in
   $\mathbb{Q}[q]$ we have the following identities:
   \begin{enumerate}
     \item $x * a = 0$ for every $a \in H_n(M; \mathbb{Q})_0$.
     \item $a * b = \frac{1}{d} \#(a \cdot b) (x^{*n} - d^d x^{*(d-2)}
      q^{n+2-d})$ for every $a, b \in H_n(M; \mathbb{Q})_0$.
   \end{enumerate}
\end{thm}

Coming back to our Lagrangian spheres $L \subset M$, we clearly have
$[L] \in H_n(M; \mathbb{Q})_0$. Therefore we obtain from
Theorem~\ref{t:qh-hypersuf}:
\begin{equation} \label{eq:L*3-hypersuf} [L]*[L]*[L] = \frac{1}{d}
   \#([L]\cdot [L]) (x^{*n}*[L] - d^d x^{*(d-2)}*[L] q^{n+2-d}) = 0,
\end{equation}
where in the last equality we have used that $d>2$ (hence
$x^{*(d-2)}*[L]=0$). 

If we also assume that $2C_M | n$, then the Lagrangian spheres $L
\subset M$ have minimal Maslov number $N_L = 2C_M$ and it is easy to
see that they satisfy Assumption~$\mathscr{L}$ (see e.g.
Proposition~\ref{p:criterion}). Therefore in this case the
discriminant $\Delta_L$ is defined and we clearly have $\Delta_L = 0$.
(Note that when $2C_M | n$ we must have $d>2$.)

Finally, we discuss the case $d=2$. A straightforward calculation
based on the quantum homology ring structure of the quadric (see
e.g.~\cite{Beau:quant}) shows that Lagrangian spheres $L \subset M$
satisfy $[L]^{*3} = (-1)^{\frac{n(n-1)}{2}+1}4 [L] q^n$
if $n=$~even and $[L]=0$ (hence
$[L]^{*2}=0$) if $n=$~odd.

\subsubsection{An example which is not a sphere}
\label{sbsb:prod-spheres} All our examples so far were for Lagrangians
that are spheres. However, our theory is more general and applies to
other topological types of Lagrangians (see e.g.
Assumption~$\mathscr{L}$, Proposition~\ref{p:criterion} and
Theorem~\ref{t:cubic-eq}). Here is such an example with $L \approx
S^{m} \times S^{m}$.

Let $Q \subset \mathbb{C}P^{m+1}$ be the complex $n$-dimensional
quadric $Q = \{ [z_0: \ldots : z_{m+1}] \,|\, -z_0^2 + \ldots +
z_{m+1}^2 = 0 \}$ endowed with the symplectic structure induced from
${\mathbb{C}}P^{m+1}$. Then $S := \{ [z_0: \ldots : z_{m+1}] \,|\,
-z_0^2 + \ldots + z_{m+1}^2 = 0, z_i \in \mathbb{R} \}$ is a
Lagrangian sphere. The first Chern class $c_1$ of $Q$ equals the
Poincar\'e dual of $m h$, where $h$ is a hyperplane section of $Q$
associated to the projective embedding $Q \subset
{\mathbb{C}}P^{m+1}$.  The minimal Chern number is $C_Q = m$ and $S$
has minimal Maslov number $N_S = 2m$. Note that $S$ does not satisfy
Assumption~$\mathscr{L}$ (since $N_S$ does not divide $m$). {\em
  Henceforth we will assume that $m=$ even.}

Put $M = Q \times Q$ endowed with the split symplectic structure
induced from both factors and consider the Lagrangian submanifold $L
\subset M$ which is the product of two copies of $S$:
$$L := S \times S \subset Q \times Q.$$
Put $2n = \dim_{\mathbb{R}} M$ so that $\dim L = n = 2m$. 

The symplectic manifold $Q \times Q$ has minimal Chern number $C_M =
m$ and the minimal Maslov number of $L$ is $N_{L} = 2m = n$.  By
Proposition~\ref{p:criterion}, $L$ satisfies Assumption~$\mathscr{L}$.

For our calculations the following identities in the quantum homology
ring of $Q$ will be relevant (see e.g.~\cite{Beau:quant}):
\begin{enumerate}
  \item $h * [S] = 0$.
  \item $a * b = \frac{1}{2} \#(a \cdot b) (h^{*m} - 4 [Q]q^m)$ for
   every $a, b \in H_m(Q;\mathbb{Q})_0$.
\end{enumerate}
To calculate $\Delta_L$ we compute $[L]^{*3}$ in $QH(Q \times Q)$.  By
the K\"{u}nneth formula in quantum homology~\cite{McD-Sa:jhol} we have
$QH(Q \times Q; \mathbb{Z}[q]) \cong QH(Q; \mathbb{Z}[q])
\otimes_{\mathbb{Z}[q]} QH(Q; \mathbb{Z}[q])$. Together with the
previous identities (with $a=b=[S]$) this gives:
$$[L]*[L] = ([S]*[S]) \otimes ([S]*[S]) =(h^{*m} - 4[Q]q^m)
\otimes (h^{*m} - 4[Q]q^m),$$ and therefore
$$[L]^{*3} = (h^{*m}*[S] - 4[S]q^m) \otimes (h^{*m}*[S] - 4[S]q^m) = 
16 [S] \otimes [S] q^{2m} = 16[L]q^{2m}.$$ It follows that $\sigma_L =
0$ and $\tau_L = 1$ (in the notation of Theorem~\ref{t:cubic-eq}),
hence $\Delta_L = 4\tau_L = 4$.


%% file: examples-full-ring-ar.tex

\section{Finer invariants over the positive group ring} \cntrs
\label{s:coeffs}

Much of the theory developed in the previous sections can be enriched
so that the discriminant $\Delta_L$ and the cubic equation take into
account the homology classes of the holomorphic curves involved in
their definition. The result is clearly a finer invariant.

We now briefly explain this generalization. Let $L \subset (M,
\omega)$ be a monotone Lagrangian submanifold. Denote by $H_2^D(M,L)
\subset H_2(M,L;\mathbb{Z})$ the image of the Hurewicz homomorphism
$\pi_2(M,L) \longrightarrow H_2(M,L;\mathbb{Z})$. We abbreviate $H_2^D
= H_2^D(M,L)$ when $L$ is clear from the discussion.

We will use here the ring $\widetilde{\Lambda}^+$, introduced
in~\cite{Bi-Co:rigidity}, which is the most general ring of
coefficients for Lagrangian quantum homology. It can be viewed as a
positive version (with respect to $\mu$) of the group ring over
$H_2^D$. Specifically, denote by $\widetilde{\Lambda}^+$ the following
ring:
\begin{equation} \label{eq:lambda-tilde+}
   \widetilde{\Lambda}^+ = \biggl \{ p(T) \mid
   p(T) = c_0 + \sum_{\substack{A \in H_2^D \\ \mu(A)>0}} c_A T^{A},
   \quad c_0, c_A \in \mathbb{Z}\biggr \}.
\end{equation}
We grade $\widetilde{\Lambda}^+$ by assigning to the monomial $T^A$
degree $|T^A| = -\mu(A)$. Note that the degree-$0$ component of
$\widetilde{\Lambda}^+$ is just $\mathbb{Z}$ (not linear combinations
of $T^{A}$ with $\mu(A)=0$). As explained in~\cite{Bi-Co:rigidity} we
can define $QH(L;\widetilde{\Lambda}^+)$, and in fact $QH(L;
\mathcal{R})$ for rings $\mathcal{R}$ which are
$\widetilde{\Lambda}^+$-algebras.

Similarly to $\widetilde{\Lambda}^+$ we associate to the ambient
manifold the ring $\widetilde{\Gamma}^+$. This ring is defined in the
same way as $\widetilde{\Lambda}^+$ but with $H_2^D$ replaced by
$H_2^S := \textnormal{image\,} (\pi_2(M) \longrightarrow
H_2(M;\mathbb{Z}))$ and with $\mu(A)>0$ replaced by $\langle c_1, A
\rangle > 0$ in~\eqref{eq:lambda-tilde+}. To avoid confusion we will
denote the formal variable in $\widetilde{\Gamma}^+$ with $S$ and we
grade $|S^A| = -2\langle c_1, A \rangle$. Similarly to $QH(L;
\widetilde{\Lambda}^+)$ we can define the ambient quantum homology
$QH(M; \widetilde{\Gamma}^+)$ with coefficients in
$\widetilde{\Gamma}^+$ and in fact with coefficients in any ring
$\mathcal{A}$ which is a $\widetilde{\Gamma}^+$-algebra.  In
particular, since the map $H_2^S \longrightarrow H_2^D$ gives
$\widetilde{\Lambda}^+$ the structure of an
$\widetilde{\Gamma}^+$-algebra and we can define $QH(M;
\widetilde{\Lambda}^+) = QH(M; \widetilde{\Gamma}^+)
\otimes_{\widetilde{\Gamma}^+} \widetilde{\Lambda}^+$.

Assume for simplicity that $L$ satisfies the assumptions of
Proposition~\ref{p:criterion}. Then the conclusion of
Proposition~\ref{p:criterion} holds with $HF(L,L)$ replaced by
$QH(L;\widetilde{\Lambda}^+)$ in the sense that
\pbred{$\textnormal{rank}_{\mathbb{Z}} \, QH_0(L;
  \widetilde{\Lambda}^+)/\widetilde{\Lambda}^+_{-n} e_L = 1$, where
  $\widetilde{\Lambda}^+_{-n} \subset \widetilde{\Lambda}^+$ stands
  for the subgroup generated by the homogeneous elements of degree
  $-n$.}  Assume further that $L$ is oriented and spinable. Again, the
main example satisfying all these assumptions is $L$ being a
Lagrangian sphere in a monotone symplectic manifold $M$ with $2 C_M |
\dim L$.

The definition of the discriminant $\Delta_L$ carries over to this
setting as follows. Pick an element $x \in QH_0(L;
\widetilde{\Lambda}^+)$ which lifts $[\textnormal{point}] \in H_0(L)$
as in~\S\ref{sbsb:Q=QH_n}. Write $$x*x = \widetilde{\sigma} x +
\widetilde{\tau} e_L,$$ where $\widetilde{\sigma}, \widetilde{\tau}
\in \widetilde{\Lambda}^+$ are elements of degrees
$|\widetilde{\sigma}| = - n$ and $|\widetilde{\tau}| = -2n$
respectively.  As before, the elements $\widetilde{\sigma}$ and
$\widetilde{\tau}$ depend on $x$.  Define
$$\widetilde{\Delta}_L = \widetilde{\sigma}^2 + 4 \widetilde{\tau} \in
\widetilde{\Lambda}^+.$$ The same arguments as
in~\S\ref{sb:more-discr} show that $\widetilde{\Delta}_L$ is
independent of the choice of $x$.

Theorems~\ref{t:cubic-eq-sphere},~\ref{t:cubic-eq} continue to hold
but the cubic equation~\eqref{eq:cubic-L} now has the form:
\begin{equation} \label{eq:cubic-L-full-ring} [L]^{*3} - \varepsilon
   \chi \widetilde{\sigma}_L [L]^{*2} - \chi^2 \widetilde{\tau}_L [L]
   = 0,
\end{equation}
where $\widetilde{\sigma}_L \in \tfrac{1}{\chi^2}
\widetilde{\Lambda}^+$, $\widetilde{\tau}_L \in \tfrac{1}{\chi^3}
\widetilde{\Lambda}^+$ are uniquely determined. (Note that
in~\eqref{eq:cubic-L-full-ring} we do not have the variable $q$
anymore since the elements $\chi^2 \widetilde{\sigma}_L, \chi^3
\widetilde{\tau}_L$ are assumed in advance to be in the ring
$\widetilde{\Lambda}^+$.) As for identity~\eqref{eq:GW-LLL}, it now
becomes:
\begin{equation} \label{eq:GW-LLL-full-ring} \widetilde{\sigma}_L =
   \frac{1}{\chi^2} \sum_A GW_{A,3}([L],[L],[L])T^{j(A)},
\end{equation}
where $j: H_2^S \longrightarrow H_2^D$ is the map induced by
inclusion.

Analogous versions of Theorem~\ref{t:cubic_eq-ccl} hold over
$\widetilde{\Lambda}^+$ too.

Denoting by $\bar{L}$ the Lagrangian $L$ with the opposite orientation,
it is easy to check that
\begin{equation} \label{eq:phi-action}
   \widetilde{\sigma}_{\bar{L}} =
   -\widetilde{\sigma}_L, \quad \widetilde{\tau}_{\bar{L}} =
   \widetilde{\tau}_L, \quad \widetilde{\Delta}_{\bar{L}} =
   \widetilde{\Delta}_L.
\end{equation}

We now discuss the action of symplectic diffeomorphisms on these
invariants. Let $\varphi: M \longrightarrow M$ be a symplectomorphism.
The action $\varphi^M_*: H_2^S \longrightarrow H_2^S$ of $\varphi$ on
homology induces an isomorphism of rings $\varphi_{\Gamma}:
\widetilde{\Gamma}^+ \longrightarrow \widetilde{\Gamma}^+$. Put $L' =
\varphi(L)$. Instead of the preceding ring $\widetilde{\Lambda}^+$ we
now have two rings $\widetilde{\Lambda}_L^+$ and
$\widetilde{\Lambda}_{L'}^+$ associated to $L$ and to $L'$
respectively. The action $\varphi^{(M,L)}_*: H_2^D(M,L) \longrightarrow
H_2^D(M,L')$ of $\varphi$ on homology induces an isomorphism of rings
$\varphi_{\Lambda}: \widetilde{\Lambda}_L^+ \longrightarrow
\widetilde{\Lambda}_{L'}^+$.  Moreover, writing an
$\mathcal{R}$-algebra $\mathcal{A}$ as $_{\mathcal{R}}\mathcal{A}$,
the pair of maps $(\varphi_{\Lambda}, \varphi_{\Gamma})$ gives rise to
an isomorphism of algebras
$_{\widetilde{\Gamma}^+}\widetilde{\Lambda}_L^+ \longrightarrow \,
_{\widetilde{\Gamma}^+} \widetilde{\Lambda}_{L'}^+$.

Turning to quantum homologies, standard arguments together with the
previous discussion yield two ring isomorphisms (both denoted
$\varphi_Q$ by abuse of notation): $$\varphi_Q: QH(L;
\widetilde{\Lambda}^+_{L}) \longrightarrow QH(L';
\widetilde{\Lambda}^+_{L'}), \quad \varphi_Q: QH(M;
\widetilde{\Lambda}^+_{L}) \longrightarrow QH(M;
\widetilde{\Lambda}^+_{L'}),$$ which are linear over
$\widetilde{\Gamma}^+$ via $\varphi_{\Gamma}$ and also
$(\widetilde{\Lambda}_L^+, \widetilde{\Lambda}_{L'}^+)$ linear via
$\varphi_{\Lambda}$. Most of the theory from~\S\ref{sb:HF} extends,
with suitable modifications, to the present setting.

The following follows immediately from the preceding discussion
and~\eqref{eq:phi-action} above:
\begin{thm} \label{t:phi-action}
   Let $\varphi: M \longrightarrow M$ be a symplectomorphism. Then:
   $$\widetilde{\sigma}_{\varphi(L)} =
   \varphi_{\Lambda}(\widetilde{\sigma}_L), \quad
   \widetilde{\tau}_{\varphi(L)} =
   \varphi_{\Lambda}(\widetilde{\tau}_L), \quad
   \widetilde{\Delta}_{\varphi(L)} =
   \varphi_{\Lambda}(\widetilde{\Delta}_L).$$ In particular
   $\widetilde{\tau}_L$ and $\widetilde{\Delta}_L$ are invariant under
   the action of the group $\textnormal{Symp}(M,L)$ of
   symplectomorphisms $\varphi: (M, L) \longrightarrow (M,L)$ and
   $\widetilde{\sigma}_L$ is invariant under the action of the
   subgroup $\textnormal{Symp}^+(M,L) \subset \textnormal{Symp}(M,L)$
   of those $\varphi$'s that preserve the orientation on $L$. If
   $\varphi \in \textnormal{Symp}(M,L)$ reverses orientation on $L$
   then $\varphi_{\Lambda}(\widetilde{\sigma}_L) =
   -\widetilde{\sigma}_L$.
\end{thm}

Next we have the following analogue of Corollary~\ref{c:sig=0}:
\begin{cor} \label{c:sig_L-sphere} Let $L \subset M$ be a Lagrangian
   sphere, where $M$ is a monotone symplectic manifold with $2C_M |
   \dim L$. Then $\widetilde{\sigma}_L = 0$. In particular,
   $\widetilde{\Delta}_L = 4 \widetilde{\tau}_L$.
\end{cor}
\begin{proof}
   Denote by $\varphi : M \longrightarrow M$ the Dehn-twist associated
   to the Lagrangian sphere $L$. Since $n=\dim L=$ even, the
   restriction $\varphi|_L$ reverses orientation on $L$. By
   Theorem~\ref{t:phi-action},
   $\varphi_{\Lambda}(\widetilde{\sigma}_L) = -\widetilde{\sigma}_L$.
   Thus the corollary would follow if we show that $\varphi_{\Lambda}
   = \textnormal{id}$. To show the latter we need to prove that the
   map induced by $\varphi$ on homology $\varphi^{(M,L)}_*: H_2(M, L)
   \longrightarrow H_2(M, L)$ is the identity.

   Assume first that $n>2$. Then the map induced by inclusion $H_2(M)
   \to H_2(M,L)$ is an isomorphism. Moreover, for every $A \in H_2(M)$
   we can find a a cycle $C$ representing $A$ which lies in the
   complement of the support of $\varphi$. This shows that
   $\varphi^{M}_*(A)=A$ hence $\varphi^{(M,L)}_* = \textnormal{id}$.

   Assume now that $n=2$. Then we have $H_2(M,L) \cong
   H_2(M)/\mathbb{Z}[L]$. By the Picard-Lefschetz formula, the action
   of $\varphi^{M}_*$ on $H_2(M)$ is given by:
   $$\varphi^{M}_*(A) = A + \#(A \cdot [L])[L].$$ It immediately follows that
   $\varphi^{(M,L)}_*: H_2(M,L) \longrightarrow H_2(M,L)$ is trivial.
\end{proof}

\subsection{Other rings of interest} \cntrsb \label{sb:other-rings}
The results in this section continue to hold if we replace the ring
$\widetilde{\Lambda}^+$ by any $\widetilde{\Lambda}^+$-algebra
$\mathcal{R}$ (graded or not). See Section 2.1.2
of~\cite{Bi-Co:rigidity} for the precise definitions (in the graded
case). Such a structure is defined e.g. by specifying a ring
homomorphism $\eta: \widetilde{\Lambda}^+ \longrightarrow
\mathcal{R}$. The most natural examples are:
\begin{enumerate}
  \item $\mathcal{R} = \mathbb{Z}$, $\mathbb{Q}$ or $\mathbb{C}$,
   where $\eta(T^A) = 1$.
  \item $\mathcal{R} = \mathbb{Z}[t^{-1}, t]$, where $\eta(T^{A}) =
   t^{\mu(A)/N_L}$.
  \item $\mathcal{R} = \mathbb{C}$, with $\eta(T^A) = \rho(A)$, where
   $\rho: H_2^D \longrightarrow \mathbb{C}^*$ is a given group
   homomorphism. This is sometime referred to as {\em twisted
     coefficients}.
  \item $\mathcal{R} = \Lambda_{\textnormal{Nov}}$ is the Novikov
   ring (say in the variable $u$), and $\eta(T^A) = u^{\omega(A)}$.
  \item Combinations of~(3) with any of the other possibilities.
  \item \label{i:rings-quotient} $\mathcal{R}$ is defined similarly to
   $\widetilde{\Lambda}^+$ but instead of taking powers $T^A$ of with
   $A \in H_2^D$ we take $A \in H_2^D/K$, where $K \subset \ker \mu$.
   See Remark~\ref{r:cob-full-ring} for such an example.  (Of course
   we can take quotients by a subgroup $K \subset H_2^D$ with $\mu|_K
   \neq 0$. Then we can still define an
   $\widetilde{\Lambda}^+$-algebra $\mathcal{R}$ by taking all linear
   combinations of $T^A$ with $A \in H_2^D/K$.)
\end{enumerate}
In all cases the Lagrangian cubic equation will hold with coefficients
in $\mathcal{R}$ and the coefficients $\sigma^{\mathcal{R}}_L,
\tau^{\mathcal{R}}_L$ and discriminant $\Delta^{\mathcal{R}}_L$ will
now be elements of $\mathcal{R}$.  Moreover if $\eta:
\widetilde{\Lambda}^+ \to \mathcal{R}$ is the ring homomorphism
defining the $\widetilde{\Lambda}^+$-algebra structure on
$\mathcal{R}$ then $\eta$ induces ring homomorphisms $QH(L;
\widetilde{\Lambda}^+) \longrightarrow QH(L; \mathcal{R})$ and
$\eta_{Q}: QH(M; \widetilde{\Lambda}^+) \longrightarrow QH(M;
\mathcal{R})$. Applying $\eta_Q$ to the cubic
equation~\eqref{eq:cubic-L-full-ring} we obtain the cubic equation
over $\mathcal{R}$. Similarly
$$\eta(\widetilde{\sigma}_L) =
\sigma^{\mathcal{R}}_L,\quad \eta(\widetilde{\tau}_L) =
\tau^{\mathcal{R}}_L,\quad \eta (\widetilde{\Delta}_L) =
\Delta^{\mathcal{R}}_L.$$ Of course if we take $\mathcal{R} =
\mathbb{Z}$ or $\mathbb{Q}$ with $\eta(T^A) = 1$ then $\eta_Q$ sends
equation~\eqref{eq:cubic-L-full-ring} to the original cubic
equation~\eqref{eq:cubic-L} with $q = 1$ and
$\eta(\widetilde{\sigma}_L) = \sigma_L$, $\eta(\widetilde{\tau}_L) =
\tau_L$, $\eta(\widetilde{\Delta}_L) = \Delta_L$.

\begin{rem} \label{r:cob-full-ring} Analogues of Theorem~\ref{t:cob}
   and Corollary~\ref{c:del_1=del_2} should carry over to the present
   setting if we replace $\widetilde{\Lambda}^+$ by the
   $\widetilde{\Lambda}^+$-algebra $\mathcal{R}$ defined as in
   point~\eqref{i:rings-quotient} of the above list where we quotient
   $H_2^D$ by the subgroup $K = \ker \bigl( H^D_2(M,\partial V)
   \longrightarrow H_2(\mathbb{R}^2 \times M, V) \bigr)$.
\end{rem}

\subsection{Examples revisited} \cntrsb \label{sb:examples-revisited}

Here we briefly present the outcome of the calculation of our
invariants $\widetilde{\tau_L}$ and $\widetilde{\Delta}_L$ for
Lagrangian spheres on blow-ups of ${\mathbb{C}}P^2$ at $2 \leq k \leq
6$ points. (As for $\widetilde{\sigma}_L$, recall that it vanishes
when $L$ is a sphere.) We use similar notation as
in~\S\ref{sbsb:intro-exp-blcp2}.  For simplicity we denote by $u \in
H_4(M_k)$ the fundamental class viewed as the unity of $QH(M_k)$. As
before we appeal to~\cite{Cra-Mir:QH} for the calculation of the
quantum homology of the ambient manifolds. Since the explicit
calculations in $QH(M_k)$ turn out to be very lengthy we often omit
the details and present only the end results (full details can be
found in~\cite{Mem:PhD-thesis}).  We recall again that in $QH(M;
\widetilde{\Gamma}^+)$ the quantum variables are denoted now by
$S^{A}$ where $A \in H_2^S$.
\subsubsection{2-point blow-up of $\mathbb{C}P^2$}
$QH(M_2; \widetilde{\Gamma}^+)$ has the following ring structure:
\begin{align*}
   & p * p  =  H S^H + uS^{2H - E_1- E_2} \\
   & p * H  =  (H - E_1)S^{H-E_1} + (H - E_2)S^{H - E_2} + u S^H \\
   & p * E_1  =  (H - E_1) S^{H - E_1} \\
   & p * E_2 = (H - E_2) S^{H - E_2} \\
   & H*H = p + (H - E_1 - E_2) S^{H - E_1 - E_2} +
   u(S^{H - E_1} + S^{H - E_2} ) \\
   & H * E_1 =  (H - E_1 - E_2) S^{H - E_1 - E_2} + u S^{H - E_1} \\
   & H * E_2  =  (H - E_1 - E_2) S^{H - E_1 - E_2} + u S^{H - E_2} \\
   & E_1 * E_1 = -p + (H - E_1 - E_2) S^{H - E_1 - E_2} + E_1
   S^{E_1} + u S^{H - E_1} \\
   & E_2 * E_2 = -p + (H - E_1 - E_2) S^{H - E_1 - E_2} + E_2
   S^{E_2} + u S^{H - E_2} \\
   & E_1 * E_2 = (H - E_1 - E_2) S^{H - E_1 - E_2}.
\end{align*}

Let $L \subset M_2$ be a Lagrangian sphere in the class $[L]=E_1-E_2$.
Then $H_2^D = H_2(M,L) \cong H_2(M) / H_2(L)$ and as a basis for
$H_2^D$ we can choose $\{H, E\}$, where $E$ stands for the image of
both $E_1$ and $E_2$ in $H_2(M) / H_2(L)$. (Thus in
$\widetilde{\Lambda}^+$ we have $S^{E_1} = S^{E_2} = T^{E}$.)

A straightforward calculation gives:
$$(E_1 - E_2)^{*3} = (T^{2E} + 4 T^{H - E}) (E_1 - E_2), \quad
\widetilde{\Delta}_L = 4\widetilde{\tau}_L = T^{2E} + 4 T^{H}.$$

\subsubsection{3-point blow-up of $\mathbb{C}P^2$}
The multiplication table for $QH(M_3; \widetilde{\Gamma}^+)$ is rather
long hence we omit it here (see~\cite{Mem:PhD-thesis} for these
details).

Consider first Lagrangian spheres $L \subset M_3$ in the class $[L] =
E_1 - E_2$.  We choose $\{H, E, E_3\}$ for a basis for $H_2^D$ where
$E$ stands for the image of both of $E_1$ and $E_2$ in $H_2^D$. A
straightforward calculation using the Lagrangian cubic equation gives
$$\widetilde{\Delta}_L = 4 \widetilde{\tau}_L =
4 T^{H - E} + T^{2E} - 2 T^{H - E_3} + T^{2H - 2E - 2 E_3}.$$

As explained in Remark~\ref{r:1-point}, we expect that there exist
Lagrangian spheres $L_1$, $L_2$ with $[L_1] =E_1 - E_2, [L_2] = E_2 -
E_3$ such that $L_1$ and $L_2$ intersect transversely at exactly one
point. By Remark~\ref{r:cob-full-ring} we should have
$$\widetilde{\Delta}_{L_1} = \widetilde{\Delta}_{L_2} =
\textnormal{perfect square},$$ if we replace the ring
$\widetilde{\Lambda}^+$ by a quotient of it where $T^{E_1}, T^{E_2},
T^{E_3}$ are all identified. The discriminant of both of $L_1$ and
$L_2$ (which now denote $\widetilde{\Delta}'$) becomes in this
setting:
$$\widetilde{\Delta}' =
2 T^{H - E} + T^{2E} + T^{2H - 4E} = (T^{H-2E} + T^E)^2,$$ where we
have written here $T^{E}$ for the $T^{E_i}$'s.  Similar calculations
should apply to the examples discussed
in~\S\ref{sbsb:M_4-full-ring}~--~\S\ref{sbsb:M_6-full-ring}.

Next we consider Lagrangian $L \subset M_3$ with $[L] = H - E_1 - E_2
- E_3$. We work with the basis $\{E_1, E_2, E_3\}$ for $H_2^D$. Direct
calculation gives
$$\widetilde{\Delta}_L = 4\widetilde{\tau}_L = T^{2 E_1} + T^{2 E_2} +
T^{2 E_3} - 2 T^{E_1 + E_2} - 2 T^{E_1 + E_3} - 2 T^{E_2 + E_3}.$$

\subsubsection{4-point blow-up of $\mathbb{C}P^2$}
\label{sbsb:M_4-full-ring} Consider Lagrangian spheres in the class
$[L] = E_1 - E_2$ and work with the basis $\{H, E, E_3, E_4\}$, where
$E = [E_1]=[E_2] \in H_2^D$. Omitting the details of a rather long
calculation we obtain:
$$\widetilde{\Delta}_L = 4 \widetilde{\tau}_L = T^{2E} + 4 T^{H - E} -
2T^{H - E_3} - 2T^{H - E_4} + T^{2H - 2E - 2E_3} +T^{2H - 2E - 2E_4} -
2T^{2H - 2E - E_3 - E_4}.$$

For Lagrangian spheres in the class $[L] = H - E_1 - E_2 - E_3$ we
obtain: $$\widetilde{\Delta}_L = 4 \widetilde{\tau}_L = T^{2E_1} +
T^{2E_2} + T^{2E_3} - 2T^{E_1+E_2} - 2T^{E_1+E_3} - 2T^{E_2+E_3} + 4
T^{E_1+E_2+E_3-E_4},$$ where we have worked here with the basis
$\{E_1, E_2, E_3, E_4\}$ for $H_2^D$.

\subsubsection{5-point blow-up of $\mathbb{C}P^2$}
Consider Lagrangian spheres in the class $[L] = E_1 - E_2$ and work
with the basis $\{H, E, E_3, E_4, E_5\}$, where $E = [E_1]=[E_2] \in
H_2^D$. Omitting the details of a rather long calculation we obtain:

\begin{align*}
   \widetilde{\Delta}_L = 4 \widetilde{\tau}_L = \, & T^{2E} + 4 T^{H
     - E} -
   2T^{H - E_3} - 2T^{H - E_4} - 2T^{H - E_5} \\
   & + T^{2H - 2E - 2E_3} + T^{2H - 2E - 2E_4} + T^{2H - 2E - 2E_5}  \\
   & - 2T^{2H - 2E - E_3 - E_4} - 2T^{2H - 2E - E_3 - E_5} -
   2T^{2H - 2E - E_4 - E_5} \\
   & + 4 T^{2H - E - E_3 - E_4 - E_5}.
\end{align*}

Consider now a Lagrangian sphere in the class $[L] = H - E_1 - E_2 -
E_3$. We work with the basis $\{E_1, E_2, E_3, E_4, E_5\}$ for
$H_2^D$. We obtain:
\begin{align*}
   \widetilde{\Delta}_L = 4 \widetilde{\tau}_L = \, & T^{2E_1} +
   T^{2E_2} + T^{2E_3} -
   2T^{E_1+E_2} - 2T^{E_1+E_3} - 2T^{E_2+E_3} \\
   & + 4T^{E_1 + E_2 + E_3 - E_4} + 4T^{E_1 + E_2 + E_3 - E_5} +
   T^{2(E_1 + E_2 + E_3 - E_4 - E_5)} \\
   & - 2 T^{2E_1 + E_2 + E_3 - E_4 - E_5} - 2 T^{E_1 + 2E_2 + E_3 -
     E_4 - E_5} - 2 T^{E_1 + E_2 + 2E_3 - E_4 - E_5}.
\end{align*}

\subsubsection{6-point blow-up of $\mathbb{C}P^2$}
\label{sbsb:M_6-full-ring} Due to the complexity of the calculation we
restrict here to Lagrangians in the class $[L] = E_1 - E_2$. We work
with the basis $\{H, E, E_3, E_4, E_5, E_5, E_6\}$ for $H_2^D$, where
$E = [E_1] = [E_2]$.

\begin{align*}
   \widetilde{\Delta}_L = \, & T^{2E} + 4 T^{H - E} - 2T^{H - E_3} -
   2T^{H - E_4} - 2T^{H - E_5} - 2T^{H - E_6} \\
   & + T^{2H - 2E - 2E_3} +T^{2H - 2E - 2E_4} + T^{2H - 2E - 2E_5} +
   T^{2H - 2E - 2E_6} \\
   & - 2T^{2H - 2E - E_3 - E_4} - 2T^{2H - 2E - E_3 - E_5} -
   2T^{2H - 2E - E_3 - E_6} \\
   & - 2T^{2H - 2E - E_4 - E_5} - 2T^{2H - 2E - E_4 - E_6} -
   2T^{2H - 2E - E_5 - E_6} \\
   & -2 T^{2H - E_3 - E_4 - E_5 - E_6} + 4 T^{2H - E - E_3 - E_4 -
     E_5} +
   4 T^{2H - E - E_3 - E_4 - E_6}  \\
   & + 4 T^{2H - E - E_3 - E_5 - E_6} +
   4 T^{2H - E - E_4 - E_5 - E_6} \\
   & - 2 T^{3H - 2E - 2E_3 - E_4 - E_5 - E_6} -
   2 T^{3H - 2E - E_3 - 2E_4 - E_5 - E_6} \\
   & - 2 T^{3H - 2E - E_3 - E_4 - 2E_5 - E_6} -
   2 T^{3H - 2E - E_3 - E_4 - E_5 - 2E_6} \\
   & + 4 T^{3H - 3E - E_3 - E_4 - E_5 - E_6} + T^{4H -2E - 2E_3 -
     2E_4 - 2E_5 - 2E_6}
\end{align*}


%% file: enumerative-ar.tex
\section{Relations to enumerative geometry of holomorphic disks}
\cntrs \label{s:enum}

Let $L^n \subset M^{2n}$ be an $n$-dimensional oriented Lagrangian
sphere in a monotone symplectic manifold $M$ with $n=$~even and $C_M =
\tfrac{n}{2}$. Note that $L$ satisfies Assumption~$\mathscr{L}$ hence
we can define its discriminant $\Delta_L \in \mathbb{Z}$ by the recipe
in~\S\ref{sb:discr-intro} or more generally $\widetilde{\Delta}_L \in
\widetilde{\Lambda}^+$ as described in~\S\ref{s:coeffs}.

The purpose of this section is to give an interpretation of the
discriminant in terms of enumeration of holomorphic disks with
boundary on $L$. A related previous result was established
in~\cite{Bi-Co:lagtop} for $2$-dimensional Lagrangian tori and the
same arguments from that paper easily generalize to our setting.

We will use below the notation from~\S\ref{s:coeffs}. Let $A \in
H_2^D$ and $J$ an almost complex structure compatible with the
symplectic structure of $M$. Denote by $\mathcal{M}_p(A, J)$ the space
of simple $J$-holomorphic disks with boundary on $L$ in the class $A$
and with $p$ marked points on the boundary (the space is defined
modulo parametrization by the group $Aut(D) \cong PSL(2,\mathbb{R})$
of biholomorphisms of the disk $D$. See Section~A.1.11
in~\cite{Bi-Co:lagtop} for the precise definitions). Denote by $ev_i:
\mathcal{M}_p(A, J) \longrightarrow L$ the evaluation at the $i$'th
marked point, where $1 \leq i \leq p$.

Fix three points $P, Q, R \in L$. Choose an oriented smooth path
$\overrightarrow{PQ}$ in $L$ starting at $P$ and ending at $Q$.
Similarly choose another two oriented paths $\overrightarrow{QR}$ and
$\overrightarrow{RP}$.

Let $A \in H_2^D$ with $\mu(A) = n$. Define $n_P(A) \in \mathbb{Z}$ to
be the number of $J$-holomorphic disks in the class $A$ whose boundaries
pass through both the path $\overrightarrow{QR}$ and the point $P$. In
other words we count the number of disks $u: (D, \partial D)
\longrightarrow (M,L)$ in the class $A$ with two marked points $z_1,
z_2 \in \partial D$ such that $u(z_1) \in \overrightarrow{QR}$ and
$u(z_2) = P$. (The disks with marked points $(u, z_1, z_2)$ are
considered modulo parametrization by $Aut(D)$ of course.) Standard
arguments show that for a generic choice of $J$ the number $n_P(A)$ is
finite.

The count $n_P(A)$ should take into account the orientations of all
the spaces involved. To this end we will use here the orientation
conventions from~\cite{Bi-Co:lagtop} and describe $n_P(A)$ via a fiber
product.  More precisely we use the spin structure on $L$ to orient
$\mathcal{M}_2(A,J)$ and define: $$n_P(A) = \#
\bigl(\overrightarrow{QR} \times_L \mathcal{M}_2(A,J) \times_L \{P\}
\bigr),$$ where the left fiber product is defined using $ev_1$, the
right one using $ev_2$ and $\#$ stands for the total number of points
in an oriented finite set, counted with signs.

Similarly, set:
\begin{align*}
   & n_Q(A) := \# \bigl(\overrightarrow{RP} \times_L
   \mathcal{M}_2(A,J) \times_L  \{Q\} \bigr), \\
   & n_R(A) := \# \bigl(\overrightarrow{PQ} \times_L
   \mathcal{M}_2(A,J) \times_L \{R\} \bigr).
\end{align*}
Define now $$n_P := \sum n_P(A) T^A \in \widetilde{\Lambda}^+,$$ where
the sum runs over all $A \in H_2^D$ with $\mu(A)=n$. Similarly define
$n_Q, n_R \in \widetilde{\Lambda}^+$.

Next, let $B \in H_2^D$ with $\mu(B) = 2n$. We would like to count the
number of $J$-holomorphic disks in the class $B$ with boundary passing
through $P,Q,R$ (in this order!). The precise definition goes as
follows. Consider the map
$$ev_{1,2,3} = ev_1 \times ev_2 \times ev_3: \mathcal{M}_3(B,J)
\longrightarrow L \times L \times L.$$ Standard arguments imply that
for a generic choice of $J$, $(ev_{1,2,3})^{-1}(P, Q, R)$ is a finite
oriented set. Consider the number of points in that set, namely
define:
$$n_{PQR}(B) := \# (ev_{1,2,3})^{-1}(P,Q,R),$$
where the count takes orientations into account.  Finally define
$$n_{PQR} : = \sum n_{PQR}(B) T^B \in \widetilde{\Lambda}^+,$$where the
sum is taken over all classes $B \in H_2^D$ with $\mu(B) = 2n$.

We remark that the numbers $n_P(A)$ (as well as the element $n_P \in
\widetilde{\Lambda}^+$) are not invariant in the sense that they
depend on the choices of the points $P,Q,R$ and of $J$. The same
happens with $n_Q, n_R$ and presumably with $n_{PQR}$ too.

\begin{thm}[c.f. Theorem~6.2.2 in~\cite{Bi-Co:lagtop}]
   \label{t:enum}
   Let $L \subset M$ be as above. Then 
   \begin{equation} \label{eq:discr-enum} \widetilde{\Delta}_L= 4
      n_{PQR} + n_P^2 + n_Q^2 + n_R^2 - 2n_P n_Q - 2n_Q n_R - 2n_R
      n_P.
   \end{equation}
\end{thm}
We omit the proof since it is a straightforward generalization of the
proof of the analogous theorem in~\cite{Bi-Co:lagtop} (see
Section~6.2.3 in that paper).

In view of the Lagrangian cubic equation~\eqref{eq:cubic-L-full-ring}
from page~\pageref{eq:cubic-L-full-ring} and
Corollary~\ref{c:sig_L-sphere} we can calculate the right-hand side
of~\eqref{eq:discr-enum} via the ambient quantum homology of $M$.

Note that if we choose the points $P,Q,R$ in specific positions
formula~\eqref{eq:discr-enum} might become simpler. For example, if we
fix the point $P$ then for a suitable (yet generic) choice of the
points $Q$ and $R$ we can make $n_P = 0$.
The formula then becomes $\widetilde{\Delta}_L= 4 n_{PQR} + (n_R-n_Q)^2$.

\begin{rem} \label{r:enum-full-ring} In contrast to
   Theorem~\ref{t:enum} the analogous statement
   from~\cite{Bi-Co:lagtop} (Theorem~6.2.2 in that paper) for
   Lagrangian tori does not work over $\widetilde{\Lambda}^+$. The
   reason is that Lagrangian tori are often not wide over
   $\widetilde{\Lambda}^+$ in the sense that for such Lagrangians
   $QH_*(L; \widetilde{\Lambda}^+)$ might not be isomorphic to $H_*(L;
   \widetilde{\Lambda}^+)$. For this reason Theorem~6.2.2
   in~\cite{Bi-Co:lagtop} is stated over the variety of
   representations $\rho: H_2^D \to \mathbb{C}^*$ for which the
   Lagrangian quantum homology $QH_*(L;\Lambda^{\rho})$ with $\rho$-twisted
   coefficients is isomorphic to $H_*(L)$. In contrast, if $L$ is an
   even dimensional Lagrangian sphere then we always have $QH_*(L;
   \widetilde{\Lambda}^+) \cong H_*(L; \widetilde{\Lambda}^+)$ (though
   possibly not in a canonical way).
\end{rem}

%% file: non-monotone-ar.tex
\section{What happens in the non-monotone case} \cntrs
\label{s:non-monotone}

Here we briefly outline how to extend, in certain situations, part of
the results of the paper to non-monotone Lagrangians.

Let $L^n \subset M^{2n}$ be a Lagrangian submanifold, which is not
necessarily monotone. Under such general assumptions, the Lagrangian
Floer and Lagrangian quantum homologies might not be well defined, at
least not in a straightforward way. There are several problems with the
definition.  The main one has to do with transversality related to
spaces of pseudo-holomorphic disks which cannot be controlled easily
(see~\cite{FO3:book-vol1, FO3:book-vol2} for a sophisticated general
approach to deal with this problem). The other problem (which is very
much related to the first one) comes from bubbling of holomorphic
disks with non-positive Maslov index. This leads to complications in
the algebraic formalism of Lagrangian Floer theory.

Nevertheless, the theory does work sufficiently well in dimension $4$
and we can still push some of our results to this case. Henceforth we
assume that $\dim M = 2n = 4$. We denote the symplectic structure of
$M$ by $\omega$. For simplicity assume that $L$ is a Lagrangian
sphere.  We fix for the rest of the section an orientation and spin structure on
$L$.

We first introduce the coefficient ring $\widetilde{\Lambda}^+_{nov}$
which is a hybrid between the Novikov ring and
$\widetilde{\Lambda}^+$. More precisely, we define
$\widetilde{\Lambda}^+_{nov}$ to be the set of all elements $p(T)$ of
the form $$p(T) = a_0 + \sum_A a_A T^A, \quad a_0, a_A \in
\mathbb{Z},$$ satisfying the following conditions. The sum is allowed
to be infinite (in contrast to $\widetilde{\Lambda}^+$) and is taken
over all $A \in H_2^D(M,L)$ satisfying both $\mu(A)>0$ and
$\omega(A)>0$. In addition we require that for every $S \in
\mathbb{R}$ the number of non-trivial coefficients $a_A \neq 0$ in
$p(T)$ with $\omega(A) < S$ is finite. It is easy to see that
$\widetilde{\Lambda}^+_{nov}$ is a commutative ring with respect to
the usual operations. We endow $\widetilde{\Lambda}^+_{nov}$ with the
same grading as $\widetilde{\Lambda}^+$, i.e. $|T^A| = -\mu(A)$.

Similarly to the monotone case, we define the minimal Chern number
$C_M$ of $(M, \omega)$ as follows. Let $H_2^S = \textnormal{image\,}
(\pi_2(M) \longrightarrow H_2(M))$ be the image of the Hurewicz
homomorphism.  Define: $C_M = \min \bigl\{ \langle c_1, A \rangle \mid
A \in H_2^S, \, \langle c_1, A \rangle >0, \; \langle [\omega], A
\rangle > 0 \bigr\}$.

The following version of Theorem~\ref{t:cubic-eq-sphere} continues to
hold for all Lagrangian $2$-spheres, whether monotone or not, provided
we work over the ring $\widetilde{\Lambda}^+_{nov}$ \cmred{in 
$QH(M)$.}
\begin{thm} \label{t:non-monotone} Let $L^2 \subset M^4$ be a
   Lagrangian $2$-sphere (without any monotonicity assumptions).  Then
   there exists $\widetilde{\gamma}_L \in \widetilde{\Lambda}^+_{nov}$
   such that $[L]^{*3} = \widetilde{\gamma}_L [L]$. If $C_M = 2$ then
   $\widetilde{\gamma}_L$ is divisible by $4$. Moreover, all the
   calculations made in~\S\ref{sb:examples-revisited} continue to hold
   without any changes in this setting.
\end{thm}

We will now outline the main points in the proof of the theorem,
paying attention to the main difficulties in the non-monotone case.

Recall that the proof of Theorem~\ref{t:cubic-eq-sphere} made use of
both the ambient quantum homology $QH(M)$ and the Lagrangian one
$QH(L)$, as well as the relations between them, e.g. the quantum
inclusion map $i_L : QH(L) \longrightarrow QH(M)$.

The ambient quantum homology $QH(M)$ can be defined (over
$\widetilde{\Lambda}^+_{nov}$) in the semi-positive case
(see~\cite{McD-Sa:jhol}) in a very similar way as in the monotone
case. This covers our case since $4$-dimensional symplectic manifolds
are always semi-positive. As for the Lagrangian quantum homology
things are less straightforward, and we explain the difficulties next.

Denote by $\mathcal{J}$ the space of almost complex structures
compatible with $\omega$. Then for generic $J \in \mathcal{J}$ there are no
non-constant $J$-holomorphic disks $u:(D, \partial D) \to (M, L)$ with
Maslov index $\mu(u) \leq 0$. This follows from the fact that the
spaces of such disks have negative virtual dimentions, together with
standard transversality arguments from the theory of
pseudo-holomorphic curves (see~\cite{McD-Sa:jhol, Laz:discs,
  Laz:decomp, Kw-Oh:discs}). From this it follows by the theory
from~\cite{Bi-Co:qrel-long, Bi-Co:rigidity} that for a generic choice
of $J$ (and other auxiliary data) the associated pearl complex is well
defined and its homology $QH(L;\widetilde{\Lambda}^+_{nov}; J)$
satisfies all the algebraic properties described in~\S\ref{sb:HF} as
long as we work with coefficients in $\widetilde{\Lambda}^+_{nov}$.
The reason to work over $\widetilde{\Lambda}^+_{nov}$ comes from the
fact that there might be infinitely many pearly trajectories
connecting two critical points that all contribute to the differential
of the pearl complex. However, for any given $0< S \in \mathbb{R}$ the
number of such trajectories with disks of total area bounded above by
$S$ is finite, and therefore the differential of the pearl complex is
well defined over $\widetilde{\Lambda}^+_{nov}$. A detailed account on
this approach to the pearl complex in dimension $4$ has been carried
out in~\cite{Cha:uniruling-dim-4}.

Since $L$ is an even dimensional sphere, for degree reasons
$QH(L;\widetilde{\Lambda}^+_{nov}; J)$ is isomorphic (possibly in a
non-canonical way) to the singular homology $H_*(L;
\widetilde{\Lambda}^+_{nov})$. However, it is not clear whether the
continuation maps $QH(L;\widetilde{\Lambda}^+_{nov}; J_0) \longrightarrow
QH(L;\widetilde{\Lambda}^+_{nov}; J_1)$ are well defined for every two
regular $J$'s, and moreover, it is a priori not clear whether the
quantum ring structure on $QH(L;\widetilde{\Lambda}^+_{nov}; J)$ is
independent of $J$.

To understand these problems better denote by $\mathcal{J}_{\mu \leq
  0} \subset \mathcal{J}$ the subspace of all $J$'s for which there
exists either a non-constant $J$-holomorphic disk with $\mu \leq 0$ or
a $J$-holomorphic rational curve with Chern number $\leq 0$.
\cmred{Roughly speaking the space $\mathcal{J}_{\mu \leq 0}$ has
  strata of codimension $1$ in $\mathcal{J}$.}  Denote by
$\mathcal{J}_{\mu>0} = \mathcal{J} \setminus \mathcal{J}_{\mu \leq 0}$
its complement. Let $J_0, J_1 \in \mathcal{J}_{\mu>0}$ be two regular
almost complex structures. If $J_0, J_1$ happen to belong to the same
path connected component of $\mathcal{J}_{\mu>0}$ then we have a
canonical isomorphism $QH(L;\widetilde{\Lambda}^+_{nov}; J_0)
\longrightarrow QH(L;\widetilde{\Lambda}^+_{nov}; J_1)$ which is in
fact a ring isomorphism.  However, for $J_0, J_1$ lying in different
path connected components of $\mathcal{J}_{\mu>0}$ this might not be
the case. The problem is that when joining $J_0$ with $J_1$ by a path
$\{J_t\}_{t \in [0,1]}$ there will be instances of $t$ where the path
goes through $\mathcal{J}_{\mu \leq 0}$, hence the spaces of pearly
trajectories used in defining the continuation maps might not be
compact due to bubbling of holomorphic disks with Maslov index $0$.
Under such circumstances ``wall crossing'' analysis is necessary in
order to try to rectify the situation.

Despite these difficulties, Theorem~\ref{t:non-monotone} still holds.
The point is that although the Lagrangian quantum homology does depend
on the choice of $J$, the ambient quantum homology
$QH(M;\widetilde{\Lambda}^+_{nov}; J)$ is independent of that choice.
Inspecting the proof of Theorem~\ref{t:cubic-eq-sphere} one can see
that the invariance of $QH(L;\widetilde{\Lambda}^+_{nov}; J)$ under changes
of $J$ does not play any role. The only important thing is that
$QH(M;\widetilde{\Lambda}^+_{nov};J)$ is independent of $J$ and that the
quantum inclusion map $i_L: QH(L;\widetilde{\Lambda}^+_{nov}; J)
\longrightarrow QH(M;\widetilde{\Lambda}^+_{nov}; J)$ is well defined and
satisfies the algebraic properties described in~\S\ref{sb:HF}.

The rest of the arguments proving Theorem~\ref{t:cubic-eq-sphere} go
through with mild modifications and yield
Theorem~\ref{t:non-monotone}. \qed

\begin{rem} \label{r:discr-determines-prod} Assume that $C_M = 1$.
   Change the ground ring from $\mathbb{Z}$ to $\mathbb{Q}$ and define
   $\widetilde{\Lambda}^+_{nov, \mathbb{Q}}$ in the same way as
   $\widetilde{\Lambda}^+_{nov}$ but over $\mathbb{Q}$.  It is easy to
   see that the discriminant $\widetilde{\Delta}_L =
   \widetilde{\gamma}_L \in \widetilde{\Lambda}^+_{nov}$ determines
   the isomorphism type of the ring $QH(L;\widetilde{\Lambda}^+_{nov,
     \mathbb{Q}}; J)$. Since the discriminant is independent of $J$ it
   follows that the ring isomorphism type of
   $QH(L;\widetilde{\Lambda}^+_{nov, \mathbb{Q}}; J)$ is in fact
   independent of $J$ too. However, as mentioned earlier, it is not
   clear if an isomorphism between the Lagrangian quantum homologies
   corresponding to $J$'s in different components of
   $\mathcal{J}_{\mu>0}$ can be realized via continuation maps.

   If $C_M = 2$ the situation is simpler. In this case there is no
   need to work over $\mathbb{Q}$, i.e. the isomorphism type of the
   Lagrangian quantum homology with coefficients in
   $\widetilde{\Lambda}^+_{nov}$ is determined by
   $\widetilde{\gamma}_L$.
\end{rem}

%% file: app-spectral.tex
\appendix

\section{Calculations in Lagrangian quantum homology}
\label{s:app-calc} \cntrs

At several instances along the paper we have appealed to basic
techniques for calculating the Lagrangian quantum homology. The main
ingredient is a spectral sequence whose initial page is the singular
homology of a given Lagrangian and which converges to its quantum
homology. This is well known to specialists and most of the details
can be recollected from several references indicated below. For the
sake of readability we summarize below the main ingredients of this
method. We begin in~\S\ref{sb:spec-seq} with the general setup of the
spectral sequence and in~\S\ref{sb:HF-lag-spheres} specialize to the
case of Lagrangian spheres.

\subsection{A spectral sequence in Lagrangian quantum homology}
\label{sb:spec-seq} \cntrsb

This is a homological version of the spectral sequence that was
introduced in~\cite{Oh:spectral} and further elaborated
in~\cite{Bi:Nonintersections}, see
also~\cite{Bi-Co:Yasha-fest,Bi-Co:rigidity}.

Let $L \subset M$ be a monotone Lagrangian submanifold with minimal
Maslov number $N_L$ and denote $n = \dim L$. Let $K$ be a commutative
ring which will serve as the ground ring for the quantum homology. In
case the characteristic of $K$ is not $2$ we assume that $L$ is spin
(i.e.  endowed with a given spin structure).

Denote by $\Lambda_K = K[t^{-1},t]$ the ring of Laurent polynomials in
$t$, graded so that the degree of $t$ is $|t| = -N_L$. Let $f: L
\longrightarrow \mathbb{R}$ be a Morse function and fix in addition a
generic almost complex structure $J$, compatible with the symplectic
structure of $M$ and a generic Riemannian metric on $L$. With this
data fixed one can define the pearl complex $(\mathcal{C}, d)$ whose
homology $QH(L; \Lambda_K)$ is (by definition) the Lagrangian quantum
homology of $L$, and which turns out to be isomorphic to the
self-Floer homology of $L$ (See~\cite{Bi-Co:Yasha-fest,
  Bi-Co:rigidity, Bi-Co:qrel-long, Bi-Co:lagtop} for the foundations
of Lagrangian quantum homology.)

Consider now the graded free $K$-module $C$ whose basis is formed by
the critical points of $f$, where the degree $i$ part is generated by
the critical points of index $i$:
$$C_i := \bigoplus_{x \in \textnormal{Crit}_i(f)} Kx, \quad
C_* := \bigoplus_{i=0}^n C_i.$$ Morse
theory~\cite{Ba-Hu:Morse-homology-book, Au-Da:Morse-Floer-book,
  Au-Da:Morse-Floer-book-eng} gives rise to a differential
$\partial^{\textnormal{m}}: C_i \longrightarrow C_{i-1}$ on $C$ whose
homology $H_*(C,\partial^{\textnormal{m}})$ is canonically isomorphic
to the singular homology $H_*(L;K)$ of $L$.

Below it will be useful to write $\Lambda_K = \oplus_{i \in
  \mathbb{Z}} P_i$, where
\begin{equation*}
   P_i =
   \begin{cases}
      K t^{-i / N_L}, & \text{if } i \equiv
      0 \; (\mathrm{mod } \, N_L),\\
      0, & \text{otherwise}.
   \end{cases}
\end{equation*}

The pearl complex $(\mathcal{C}, d)$ is related to $C$ as follows.
Its underlying module is defined by $\mathcal{C}_* = C_* \otimes_K
\Lambda_K$, where the grading is induced from both factors in the
tensor product. Thus we have:
$$\mathcal{C}_l = \bigoplus_{k \in \mathbb{Z}} C_{l - k N_L} \otimes
P_{k N_L}, \quad \forall \, l \in \mathbb{Z}.$$

The differential $d$ can be written as a sum of $K$-linear operators
as follows:
\begin{equation} \label{eq:pearl-diff}
   d = \partial_0 \otimes 1 + \partial_1 \otimes t + \cdots +
   \partial_{\nu} \otimes t^{\nu},
\end{equation}
with $\partial_i: C_j \to C_{j + i N_L -1}$ and $\nu = \left[
   \tfrac{n+1}{N_L}\right]$.  Moreover, the first operator in this sum
coincides with the Morse differential, i.e.  $\partial_0 =
\partial^{\textnormal{m}}$. We refer the reader
to~\cite{Bi-Co:Yasha-fest, Bi-Co:rigidity, Bi-Co:qrel-long,
  Bi-Co:lagtop} for the precise definition of the operators
$\partial_i$. As far as this section is concerned, the only relevant
thing is the precise shift in grading for each $\partial_i$.

Consider now the following increasing filtration
$\mathcal{F}_{\bullet}\Lambda_K$ on $\Lambda_K$:
$$\mathcal{F}_p \Lambda_K :=
\Biggl\{ h(t) \in \Lambda_K \, \Big| \, h(t) = \sum_{-p \leq k} a_k
t^k\Biggr\} = \bigoplus_{j \leq p} P_{j N_L}.$$ This filtration
induces an increasing filtration on the chain complex $(\mathcal{C},
d)$ by setting $\mathcal{F}_p \mathcal{C} = \mathcal{C} \otimes
\mathcal{F}_p \Lambda_K$ or more specifically:
$$(\mathcal{F}_p \mathcal{C})_l = \bigoplus_{j \leq p} C_{l - j N_L}
\otimes P_{j N_L}, \quad \forall \, p, l \in \mathbb{Z}.$$ The fact
that the differential preserves the filtration follows
from~\eqref{eq:pearl-diff}. Note also that for degree reasons the
filtration $\mathcal{F}_{\bullet} \mathcal{C}$ is bounded.

According to standard spectral sequence
theory~\cite{Weibel:book-hom-alg} the filtration
$\mathcal{F}_{\bullet}\mathcal{C}$ induces a spectral sequence $\{
E^r_{p, q}, d^r \}_{_{r \geq 0}}$ which converges to $H_*(\mathcal{C},
d) = QH_*(L; \Lambda_K)$.

The following theorem is an obvious homological adaptation of
Theorem~5.2.A from~\cite{Bi:Nonintersections}.
\begin{thm} \label{t:spectral-seq} The spectral sequence $\{ E^r_{p,
     q}, d^r \}$ has the following properties:
   \begin{enumerate}
     \item[(1)] $E^0_{p,q} = C_{p + q - pN_L} \otimes P_{p N_L}$, $d^0
      = \partial_0 \otimes 1$;
     \item[(2)] $E^1_{p,q} = H_{p + q - pN_L}(L;K) \otimes P_{p N_L}$,
      $d^1 = [\partial_1] \otimes t$, where
      \[
      [\partial_1]: H_{p + q - pN_L}(L; K) \longrightarrow H_{p + q -1
        - (p-1) N_L}(L; K)
      \]
      is induced by the map $\partial_1$.
     \item[(3)] $\{ E^r_{p, q}, d^r \}$ collapses at the $\nu + 1$
      step, namely $d^r = 0$ for every $r \geq \nu + 1$ (hence we
      denote $E^{\infty}_{p, q} = E^r_{p, q}$ for $r \geq \nu + 1$).
      Moreover, the sequence converges to $QH_*(L; \Lambda_K)$. In
      particular, when $K$ is a field we have:
      \[
      \bigoplus_{p+q = l} E^{\infty}_{p, q} \cong QH_l(L; \Lambda_K)
      \; \; \forall \, l \in \mathbb{Z}.
      \]
   \end{enumerate}
\end{thm}

\subsection{Quantum homology of Lagrangian spheres}
\label{sb:HF-lag-spheres} \cntrsb

\begin{prop} \label{p:HF-lag-spheres} Let $L \subset M$ be an
   $n$-dimensional monotone Lagrangian submanifold which is a
   $\mathbb{Q}$-homology sphere. Then:
   \begin{enumerate}[(i)]
     \item If $n$ is even then $QH_*(L; \Lambda_{\mathbb{Q}}) \cong
      H_*(L; \mathbb{Q}) \otimes \Lambda_{\mathbb{Q}}$.
     \item Assume $n$ is odd. If $N_L \centernot| n + 1$ or $N_L \,|\,
      n + 1$ and $[L] \neq 0$, then $QH_*(L; \Lambda_{\mathbb{Q}})
      \cong H_*(L; \mathbb{Q}) \otimes \Lambda_{\mathbb{Q}}$. If $N_L
      \,|\, n + 1$ and $[L] = 0$, then $QH_*(L; \Lambda_{\mathbb{Q}})$
      is either 0 or isomorphic to $H_*(L; \mathbb{Q}) \otimes
      \Lambda_{\mathbb{Q}}$.
   \end{enumerate}
\end{prop}
Note that the isomorphisms in~(i) might not be canonical in case $N_L
| n$ (for more on this phenomenon see~\S 4.5
in~\cite{Bi-Co:rigidity}).

\begin{proof}
   The proof is based on the spectral sequence of~\S\ref{sb:spec-seq}
   and on Theorem~\ref{t:spectral-seq}.

   Before we start recall that $N_L$ must be even since $L$ is
   orientable.

   Assume that $n$ is even. Then $E^1_{p,q} = 0$ if $p +q =$ odd,
   since $N_L$ is even. Thus for $r \geq 1$ the higher differentials
   $d^r: E^r_{p,q} \to E^r_{p-r,q + r - 1}$ all vanish, hence
   $E^1_{p,q} = E^{\infty}_{p,q}$. This gives us $QH_*(L;
   \Lambda_{\mathbb{Q}}) \cong H_*(L; \mathbb{Q}) \otimes
   \Lambda_{\mathbb{Q}}$.

   Assume now that $n$ is odd. If $p + q = $ odd, then the only
   non-trivial terms in $E^1_{p,q}$ are
   \[
   E^1_{p,q} = H_n(L; \mathbb{Q}) \otimes P_{p N_L},
   \]
   where $p + q = n + p N_L$. If $p + q =$ even, then the only
   non-trivial terms are
   \[
   E^1_{p,q} = H_0(L; \mathbb{Q}) \otimes P_{p N_L},
   \]
   where $p + q = p N_L$. Now for degree reasons the maps $d^r:
   E^r_{p,q} \to E^r_{p-r,q+r-1}$ are $0$ if $p + q =$ odd, since
   either $E^r_{p,q} = 0$ or $E^r_{p-r,q+r-1} = 0$ or both are
   trivial. It remains to consider the maps $d^r: E^r_{p,q} \to
   E^r_{p-r,q+r-1}$ for $p + q =$ even.

   We now assume that $N_L \centernot| n + 1$. Then $d^1: H_0(L;
   \mathbb{Q}) \otimes P_{pN_L} \to H_{N_L-1}(L; \mathbb{Q}) \otimes
   P_{(p-1)N_L}$ and the assumption implies that this operator is $0$.
   By the same reasoning the higher differentials $d^r$ vanish for all
   $r \geq 2$. Thus we obtain $QH_*(L; \Lambda_{\mathbb{Q}}) \cong
   H_*(L; \mathbb{Q}) \otimes \Lambda_{\mathbb{Q}}$.

   Assume $N_L \,|\, n + 1$ and $[L] \neq 0$. Since $i_L(e_L) = [L]
   \neq 0$, this implies that $e_L \in QH_n(L; \Lambda_{\mathbb{Q}})$
   is non-zero and hence not a boundary (we are using $\mathbb{Q}$ as
   our ground ring). Therefore the operators $d^r$ must vanish for all
   $r \geq 1$. We obtain the desired isomorphism.

   In the case $N_L \,|\, n + 1$ and $[L] = 0$ there exists either an
   $r \geq 1$ such that $d^r$ is non-zero or $d^r$ is always zero.
   This corresponds to both cases in the assertion.
\end{proof}

\begin{rem}
   In Proposition~\ref{p:HF-lag-spheres} the case $N_L \,|\, n + 1$
   and $[L] = 0$ leads to two possibilities for $QH(L;
   \Lambda_{\mathbb{Q}})$. One can distinguish between them by
   counting the algebraic number of pseudo-holomorphic disks of Maslov
   index $n + 1$ through two generic points of $L$. If this number is
   $0$ then $QH(L; \Lambda_{\mathbb{Q}}) \cong H_*(L;\mathbb{Q})
   \otimes \Lambda_{\mathbb{Q}}$, otherwise $QH(L;
   \Lambda_{\mathbb{Q}})$ vanishes.
\end{rem}

%% file: lcubic-ar.bbl
\def\cprime{$'$}